\documentclass[a4paper,10pt,oneside]{article}
\pdfoutput=1

\usepackage[T1]{fontenc}
\usepackage[utf8]{inputenc}
\usepackage[british]{babel}

\usepackage[backend=biber,style=alphabetic,maxbibnames=99,sorting=nyt]{biblatex}

\NewBibliographyString{appendix,by}
\DefineBibliographyStrings{english}{
  appendix = {with an appendix},
  by       = {by},
}

\newbibmacro*{related:appendix}[1]{%
  \entrydata{#1}{%
    \bibstring{by}%
    \setunit{\addspace}%
    \printnames{author}}}

\renewbibmacro{in:}{}

\DeclareFieldFormat{pages}{#1}

\usepackage[babel,activate={true,nocompatibility},final,tracking=true,kerning=true,factor=1100,stretch=0,shrink=0]{microtype}
 \SetTracking{encoding={*}, shape=sc}{30}

\usepackage{amsfonts}
\usepackage{amsmath} 
\usepackage{amsthm}  
\usepackage{amssymb}
\numberwithin{equation}{section}
\usepackage{mathtools}
\usepackage{url}
\usepackage{stmaryrd}
\usepackage{titling}
\usepackage{centernot}
\usepackage{indentfirst}

\usepackage[hang,flushmargin]{footmisc}
\usepackage[affil-it]{authblk}
\usepackage{orcidlink}
\usepackage{hyperref}
\hypersetup{colorlinks=false, pdfborder={0 0 0}}

\newcommand{\email}[1]{\href{mailto:#1}{\texttt{#1}}}

\usepackage[inline]{enumitem}
\setlist[enumerate,1]{label=(\arabic*)}

\usepackage{fancyhdr}
\usepackage[capitalize]{cleveref}
\Crefname{pr}{Proposition}{Propositions}
\Crefname{co}{Corollary}{Corollaries}
\Crefname{subsection}{Subsection}{Subsections}

\usepackage{tikz}
\usetikzlibrary{matrix,arrows,decorations.pathmorphing, shapes, decorations.fractals, patterns, positioning}
\tikzset{snake it/.style={decorate, decoration={snake, segment length=5.3pt, amplitude=0.25mm}}}

\newcommand{\from}{\colon}
\renewcommand{\phi}{\varphi}
\renewcommand{\epsilon}{\varepsilon}
\renewcommand{\models}{\vDash}
\newcommand{\monster}{\mathfrak U}
\newcommand{\bla}[4]{{#1}_{#2}#3\ldots#3{#1}_{#4}}
\newcommand{\inverse}{^{-1}}
\newcommand{\allora}{\Rightarrow}
\newcommand{\into}{\hookrightarrow}
\newcommand{\ar}{\mathrm{ar}}
\newcommand{\indu}{\mathrm{ind}}
\newcommand{\id}{\mathrm{id}}
\newcommand{\inp}{\mathrm{inp}}
\newcommand{\lex}{\mathrm{lex}}
\newcommand{\mor}{\mathrm{Mor}}
\newcommand{\alg}{\mathrm{alg}}

\DeclareMathOperator{\trop}{trop}
\DeclareMathOperator{\ini}{in}
\DeclareMathOperator{\tp}{tp}
\DeclareMathOperator{\EM}{EM}
\DeclareMathOperator{\qftp}{qftp}
\DeclareMathOperator{\ot}{ot}
\DeclareMathOperator{\dcl}{dcl}
\DeclareMathOperator{\aut}{Aut}
\DeclareMathOperator{\orb}{orb}
\DeclareMathOperator{\im}{im}
\DeclareMathOperator{\bdn}{bdn}
\DeclareMathOperator{\res}{res}
\DeclareMathOperator{\ac}{ac}
\DeclareMathOperator{\VF}{VF}

\DeclarePairedDelimiter{\set}{\{}{\}}
\DeclarePairedDelimiter{\abs}{\lvert}{\rvert}
\DeclarePairedDelimiter{\strgen}{\langle}{\rangle}

\theoremstyle{definition}
\newtheorem{defin}{Definition}[section]
\newtheorem{thm}[defin]{Theorem}
\newtheorem{pr}[defin]{Proposition}
\newtheorem{co}[defin]{Corollary}
\newtheorem{lemma}[defin]{Lemma}
\newtheorem{notation}[defin]{Notation}
\newtheorem{eg}[defin]{Example}
\newtheorem{rem}[defin]{Remark}
\newtheorem{fact}[defin]{Fact}
\newtheorem{ass}[defin]{Assumption}
\newtheorem{prob}[defin]{Problem}
\newtheorem{claim}{Claim}[defin]
\newtheorem{alphthm}{Theorem}

\let\oldqed\qedsymbol
\newcommand{\qedclaim}{\mbox{$\underset{\textsc{claim}}{\oldqed}$}}
\makeatletter
\newenvironment{claimproof}[1][\it Proof of Claim]{
\let\qedsymbol\qedclaim
  \par
  \pushQED{\qed}%
  \normalfont \topsep6\p@\@plus6\p@\relax
  \trivlist
\item[\hskip\labelsep
  \upshape
  #1\@addpunct{.}]\ignorespaces
}{%
  \popQED\endtrivlist\@endpefalse
}
\makeatother
\let\qedsymbol\oldqed

\makeatletter
\newcommand{\subjclass}[2][2020]{%
  \let\@oldtitle\@title%
  \gdef\@title{\@oldtitle\footnotetext{\hspace*{-2em}#1 \emph{Mathematics subject classification.} #2}}%
}

\newcommand{\keywords}[1]{%
  \let\@@oldtitle\@title%
  \gdef\@title{\@@oldtitle\footnotetext{\hspace*{-2em}\emph{Keywords:} #1.}}%
}
\makeatother

\title{Automorphisms of valued fields:\\ amalgamation and existential closedness}

\author[1]{Jan~Dobrowolski\orcidlink{0000-0003-3435-4782}%
  \thanks{email: \email{dobrowol@math.uni.wroc.pl}}}

\author[2]{Francesco~Gallinaro\orcidlink{0000-0002-9171-8522}%
  \thanks{email: \email{francesco.gallinaro@dm.unipi.it}}}

\author[ ]{Rosario~Mennuni\orcidlink{0000-0003-2282-680X}%
  \thanks{email: \email{R.Mennuni@posteo.net}}}

\affil[1]{\small Xiamen University Malaysia,
	Department of Mathematics,
	Jalan Sunsuria, Bandar Sunsuria, 43900 Sepang, Selangor Darul Ehsan, Malaysia \emph{and}\newline 
 Instytut Matematyczny, Uniwersystet Wrocławski, Wrocław, Poland}
\affil[2]{\small Dipartimento di Matematica, Universit\`a di Pisa, Largo Bruno Pontecorvo 5, 56127 Pisa, Italy}

\date{}

\subjclass{Primary: 12J10, 12L12, 03C45. Secondary: 12H10, 03C60, 03C98.}

\keywords{valued difference field, model theory, positive logic, amalgamation, tree property}

\begin{document}

\maketitle

\vspace{-1cm}

\begin{abstract}We study valued fields equipped with an automorphism. We prove that all of them have an extension admitting an equivariant cross-section of the valuation. In residual characteristic zero, and in the presence of such a cross-section, we show that amalgamation problems are solvable precisely when the induced residual problem is, characterise the existentially closed objects of this category, and prove that its positive theory does not have the tree property of the second kind. We prove analogous results with cross-sections replaced by angular components. Along the way, we show that array modelling does not require thickness.
\end{abstract}

\tableofcontents

\section*{Introduction}
\label{sec:intro}
\addcontentsline{toc}{section}{\nameref{sec:intro}}
\markboth{Introduction}{}

The study of difference equations may be carried out inside a suitable ring of functions, say in the indeterminate $x$, equipped with the endomorphism $\sigma(f(x))=f(x+1)$. For this reason, a field $K$ with a distinguished endomorphism $\sigma$ is called a \emph{difference field}~\cite{cohnDifferenceAlgebra1965}; when $\sigma$ is an automorphism, $(K,\sigma)$ is said to be \emph{inversive}. The typical object of interest in this work will be an inversive \emph{valued difference field}: a valued field $(K,v)$ together with an automorphism $\sigma$ compatible with the valuation $v$, that is, fixing the valuation ring $\mathcal O$ setwise. If we denote by $\Gamma$ the value group and by $v\from K\to \Gamma\cup\set\infty$ the (surjective) valuation, we could equivalently specify a pair of automorphisms, one of the field $K$ and one of the ordered abelian group $\Gamma$, that we will both denote by $\sigma$, such that for all $a\in K$ one has $v(\sigma(a))=\sigma(v(a))$.

Automorphisms of valued fields appear for instance in \cite{senAutomorphismsLocalFields1969,vanderputDifferenceEquationsOverpadic1972,kellerValuedCompleteFields1986,kuhlmannAutomorphismGroupValued2022,kaplanDecomposingAutomorphismGroup2025}. Their model theory has been studied intensively since the beginning of the millennium, for example in~\cite{Sc,Sc2,belairModelTheoryFrobenius2007,Az,AvD,Pa,DO,Ri,blaszkiewiczNoteValuedFields2025}, and it has been used to obtain Lang--Weil-type estimates in finite difference fields~\cite{hilsLangWeilTypeEstimates2024} and a version of Kapranov's Theorem for tropical difference algebra~\cite{aliyariDifferenceKapranovTheorem2025}. A connection with algebraic dynamics was envisaged in the introduction to \cite{chatzidakisDifferenceFieldsDescent2008}.
Although the techniques have recently been extended to the non-inversive case \cite{dorSpecializationDifferenceEquations2022,dorContractingEndomorphismsValued2023,ramelloModelTheoryValued2024},  all difference fields in this paper will be inversive, hence we will from now on omit this qualifier and tacitly assume the surjectivity of all of our field endomorphisms, unless otherwise specified.

When studying fields, it is often convenient to place oneself in a \emph{universal domain}: a large enough algebraically closed field. It has long been known that this is not a peculiarity of field theory, as one can give a precise definition of ``universal domain'' via the notion of \emph{existentially closed} object of a category of structures with embeddings (\Cref{defin:ec}). If such a category is closed under inductive limits, every object embeds in an existentially closed one. This happens to be the case for the category of valued difference fields; here, a consequence of existential closedness is a form of the Nullstellensatz for difference varieties which takes the valuation into account: for example, whenever a difference variety defined over $K$ has a point in some valued difference field extension $(L, v_L, \sigma_L)\supseteq (K,v,\sigma)$ with all coordinates of valuation zero, then such a point already exists in $K$. A sufficiently \emph{saturated} (\Cref{defin:saturated}) existentially closed valued difference field can then be seen as a universal domain for \emph{tropical difference algebra}.

 For categories of structures with an automorphism it is customary to refer to the automorphism of an  existentially closed object as \emph{generic}~\cite{Ma,chatzidakisGenericStructuresSimple1998}. Note that not every structure carries a generic automorphism; for example,  every field that does must be algebraically closed; this follows from the fact that every automorphism of a field can be extended to its algebraic closure. In the valued difference case, some consequences of genericity can be easily deduced using standard results from the model theory of valued fields: for instance (\Cref{pr:ecbasics}) the valued field needs to be algebraically closed,  and its residue field and value group, equipped with the induced automorphisms, need to be existentially closed in their respective categories.

After observing this, one is immediately confronted with the fact that, in almost all of the existing literature on fields with a valuation and an automorphism, the behaviour of $\sigma$ on $\Gamma$ is severely constrained, and certainly not generic. Namely, most of the currently available results are only valid in the \emph{multiplicative} case~\cite{Pa}, two important special cases being the \emph{isometric} one---where $v(\sigma(x))=v(x)$ for every $x$, that is, the induced automorphism of $\Gamma$ is the identity---and the \emph{strongly contractive} one---where $v(\sigma(x))>n\cdot v(x)$ for every $x$ of positive valuation and every $n\in \mathbb N$. Exceptions are~\cite{beyarslanFieldsAutomorphismValuation2020}, which deals with existentially closed  $(K,v,\sigma)$ where $\sigma$ is only required to preserve the field structure but not $v$---hence there is no induced automorphism of $\Gamma$ at all---and~\cite{DO,Ri,ramelloModelTheoryValued2024}, that black-box the behaviour of $\sigma$ on $\Gamma$. Another non-multiplicative example is, on the field of transseries, precomposition by a positive infinite element, studied in~\cite[Chapter 5]{vanderhoevenTransseriesRealDifferential2006a}  and~\cite[Section~8]{DO}.

Existential closedness for ordered abelian groups with an automorphism---from now on, \emph{difference oags}---has recently been studied in~\cite{DM}. Having started to grasp an understanding of existentially closed difference oags, which occur as value groups of existentially closed valued difference fields, it is natural to wonder whether one may leverage knowledge of the former to obtain results on the latter. In this paper, we will see that the answer is affirmative.

The first property one may expect from a universal domain for tropical difference algebra is the existence of a  \emph{cross-section} of $v$---a group morphism $s\from\Gamma\to K^\times$ such that $v(s(\gamma))=\gamma$---that is $\sigma$-equivariant. In turn, such an object induces a $\sigma$-equivariant \emph{angular component}, that is, a multiplicative homomorphism from $K$ to the residue field $k$ extending the residue map. The existence of these maps implies that the inclusion  of $\mathbb{Z}[\sigma]$-modules $\mathcal O^\times\into K^\times$ is \emph{pure}: every system of multiplicative difference equations with parameters in $\mathcal O^\times$ that has a solution in $K^\times$ must have one in $\mathcal O^\times$. If said equations are not allowed to mention $\sigma$, then this is well known to follow from the fact that $\Gamma$ is torsion-free. When $\sigma$ enters the picture, showing this becomes more intricate, and constitutes our first main result.
 \begin{alphthm}[\Cref{thm_purity}]\label{intro_thm_purity}
Suppose $(K,v,\sigma)$ is an existentially closed valued difference field. Then $\mathcal O^\times$ is a pure $\mathbb{Z}[\sigma]$-submodule of $K^\times$.  It follows that every valued difference field embeds into one with a cross-section commuting with $\sigma$.
\end{alphthm}
 While adding a cross-section to the language enlarges the family of definable sets, we prove that this does not drastically change the notion of universal domain. Namely, every valued difference field with a cross-section that is existentially closed is also existentially closed as a valued difference field with no extra structure (\Cref{co:secacec}); we also obtain similar results for other expansions of valued difference fields.

We prove the above results in \Cref{sec:purity} by combining logical and tropical tools, such as the Compactness Theorem, saturation, and initial varieties. We also use a version of the Fundamental Theorem of Tropical Geometry for valuations with arbitrary rank (\Cref{thm:fundamental-arbitrary-rank}), which we prove in \Cref{sec:prelim} by a compactness argument using upper bounds from \cite{joswigDegreeTropicalBasis2018} on the minimum degree of a tropical basis. The existence of an equivariant cross-section allows to define equivariant initial forms (\Cref{trop:eqvini}), thereby providing further evidence that this framework is suitable to develop tropical difference algebra. We will also see that, by an argument due to Philip Dittmann, existentially closed valued difference fields are not existentially closed as difference fields (\Cref{co:acvfanotacfa}).

In \Cref{sec:ap} we work in the case of residual characteristic zero, and study \emph{amalgamation problems}: given a base object $A$ and morphisms $f\from A\to B$ and $g\from A\to C$, we seek an object $D$ and morphisms making the diagram in \Cref{fig:amaldiag} commute.
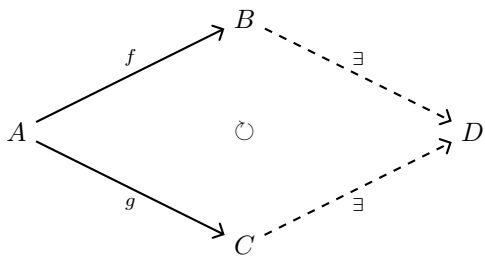
\begin{figure}
  \begin{center}
    \begin{tikzpicture}[scale=3]
  \node(a) at (0,0){$A$};
  \node(b) at (1,0.5){$B$};
  \node(c) at (1,-0.5){$C$};
  \node(d) at (2,0){$D$};
  \node(comm) at (1,0){$\circlearrowright$};
\path[->, thick,  font=\scriptsize,>= angle 90]
(a) edge node [above]  {$f$} (b)
(a) edge node [below]  {$g$} (c)
(b) edge [dashed] node [above]  {$\exists$} (d)
(c) edge [dashed] node [below]  {$\exists$} (d)
;
\end{tikzpicture}
\end{center}
\caption{An amalgamation diagram.}\label{fig:amaldiag}
\end{figure}

It is well known that, in the category of valued fields, as well as in that of difference fields, where in both cases we allow as arrows all possible embeddings, not every amalgamation problem can be solved. For difference fields, one may solve amalgamation problems where the arrows are required to preserve certain predicates expressing the possibility to add solutions to difference equations of a specific form~\cite{Ma}.  In our setting, even adding the aforementioned predicates to the residue field does not allow to solve all amalgamation problems (by \Cref{eg:apfailure}).

 As for valued fields, one remedy is to use  \emph{$\mathrm{RV}$-structures}, or equivalently~\cite{linziHyperfieldsAssociatedValued2025}  \emph{stringent valued hyperfields}; to the best of our knowledge, a theory of generic automorphisms of such structures is still lacking. For this reason, we adopt another common approach, and work in the category of valued difference fields endowed with a $\sigma$-equivariant angular component---briefly, $\mathrm{ac}$-valued difference fields---of residual characteristic zero, with embeddings required to commute with angular components. We characterise which amalgamation problems have solutions by proving an embedding result (\Cref{pr:embeddingvdf}) that also allows us to characterise the existentially closed objects of this category in terms of genericity of the automorphisms induced on $k$ and $\Gamma$ and a generalisation of henselianity commonly used in the context of valued difference fields.
\begin{alphthm}[\Cref{AP_VDF}]\label{intro_thm_amalg}
An amalgamation problem of $\mathrm{ac}$-valued difference fields of residual characteristic zero has a solution if and only if the induced amalgamation problem of difference residue fields does.
\end{alphthm}
\begin{alphthm}[\Cref{ec_characterisation}]\label{intro_thm_char_ec}
  An $\mathrm{ac}$-valued difference field of residual characteristic zero is existentially closed if and only if it is $\sigma$-henselian and the automorphisms induced on its value group and residue field are both generic.
\end{alphthm}
Both theorems also hold if we work with cross-sections instead of angular components, and in both cases we may also add a lift of the residue field to the language.  Combining \Cref{intro_thm_char_ec} with results from~\cite{DO} we show that, if $\operatorname{char}(k)=0$ and both  $(k,\sigma)$ and $(\Gamma,\sigma)$ are existentially closed, then so is their Hahn product $k((\Gamma))$ (\Cref{eg:ecHahn}).

As observed in~\cite{pillay_forking_2000,BY,BYP}, even if a class of existentially closed structures is not \emph{elementary}, which happens to be the case in our setting by \cite{kikyoStrictOrderProperty2002}, one can still analyse it through the lenses of model theory by using \emph{positive logic}, an approach already followed in e.g.~\cite{HK,kamsmaBilinearSpacesFixed2023,DKN,DM} that constitutes one of the main motivations for the general study of this framework, carried out for instance in \cite{Ha,DK,DGK,Ka}.  Positive logic has also been used to introduce \emph{definability patterns}~\cite{hrushovskiDefinabilityPatternsTheir2020,segelPositiveDefinabilityPatterns2025} and obtain applications to approximate subgroups~\cite{hrushovskiLascarGroup2022}.

 The relevant neostability-theoretic \emph{dividing line} for this paper is the \emph{tree property of the second kind}: the existence of an unbounded definable family of independent partitions. The absence of this property is a tameness condition, known as $\mathsf{NTP}_2$ and studied intensively in recent years, see e.g.~\cite{Chernikov,CH,montenegroSTABILIZERSGROUPSGENERICS2020,Ka}. In our final \Cref{sec:ntp2} we take the opportunity to advance the study of $\mathsf{NTP}_2$ in the positive setting by generalising to arbitrary positive theories the \emph{Array Modelling Theorem} (\Cref{ar_mod}), a Ramsey-theoretic statement established in~\cite{Ka} for \emph{thick} positive theories.  Array modelling implies a number of desirable general properties such as \emph{submultiplicativity of burden} (\Cref{co:submulbur}) that, combined with \Cref{intro_thm_amalg}, allow us to adapt the arguments of~\cite{CH} to prove that our class enjoys the above tameness property, which for example fails in existentially closed exponential fields~\cite{HK}.
\begin{alphthm}[\Cref{ntp2}]
Existentially closed valued difference fields of residual characteristic zero are $\mathsf{NTP}_2$ in the sense of positive logic.
\end{alphthm}

\paragraph{Acknowledgements}
 \Cref{pr:philip} and \Cref{eg:angus} are due respectively to Philip Dittmann and Angus Matthews; we are grateful to them for allowing us to include their results here. We thank Simone Ramello and Silvain Rideau-Kikuchi for extensive and very inspiring discussions on several problems addressed in this paper.  Thanks to Philip Dittmann, Martin Hils, Ehud Hrushovski, Mark Kamsma and Martin Ziegler for helpful comments.

 \paragraph{Funding information}JD was supported by EPSRC grant EP/V03619X/1 and by Xiamen University Malaysia Research Fund (Grant No: XMUMRF/2025-C15/IMAT/0037).
 FG was partially supported by GeoMod ANR-19-CE40-0022-01 (ANR-DFG).
 FG and RM were supported by the project PRIN 2022: ``Models, sets and classifications'' Prot.~2022TECZJA, Italian Ministry of University and Research (MUR). RM was supported by the project PRIN 2022 ``Logical methods in combinatorics'', 2022BXH4R5 and is a member of the INdAM research group GNSAGA.  We acknowledge the MUR Excellence Department Project awarded to the Department of Mathematics, University of Pisa, CUP I57G22000700001.

\section{Preliminaries}\label{sec:prelim}
\subsection{Model theory}
Lowercase letters such as $a,b,x,y,\ldots$ will usually denote finite (except in \Cref{sec:ntp2}) tuples of elements or of variables, e.g.\ $a=(\bla a0,{n-1})$, and their length (in the previous case, $n$) will be denoted by $\abs a$. We abuse the notation and write e.g.\ $a\in M$ instead of $a\in M^{\abs a}$. Similarly, for a function $f$ defined on $M$, we write $f(a)$ for $(f(a_0),\ldots,f(a_{n-1}))$, etc.
We refer the reader to~\cite[Section~1.2]{DM} for a quick tour of some model-theoretic terminology and facts, to \cite{Ho} for a more extensive treatment of the fundamentals, and to \cite{Dri} for the basics of valued fields and of their model theory. We nevertheless recall the following standard definition.
\begin{defin}\label{defin:ec}
  Let $\mathcal L$ be a first-order language, $T$ an $\mathcal L$-theory axiomatised by $\forall \exists$-sentences, and $\mathcal K$ the category of models of $T$  with embeddings. An object $M\in \mathcal K$ is \emph{existentially closed} iff, for every embedding $f\from M \to N$, every tuple $a$ from $M$, and every quantifier-free formula $\phi(x,y)$ without parameters,
if $N\models \exists x\; \phi(x,f(a))$ then $M\models\exists x\; \phi(x,a)$.
\end{defin}
\begin{rem}\label{rem:AEbasics}The following is well known.
  \begin{enumerate}
  \item  Every such $\mathcal K$ is closed under inductive limits~\cite[Theorem~6.5.9 and Exercise 5 in Section 6.5]{Ho}, and  this in turn ensures that every object embeds in an existentially closed one~\cite[Theorem~8.2.3]{Ho}. 
  \item By enlarging $\mathcal L$, hence restricting the class of embeddings between models, every first-order theory can be $\forall \exists$-axiomatised~\cite[Theorems~2.6.5 and~2.6.6]{Ho}. However, this process (known as \emph{Morleyisation} or \emph{atomisation}) may change the notion of existentially closed model.
  \end{enumerate}
\end{rem}

In our context, the following assumption is natural; see \Cref{rem:stuffforfree}\ref{point:inversiveforfree}.
\begin{ass}
 All difference fields and oags are assumed inversive.
\end{ass}

\begin{defin}\label{defin:languages}
  We will mostly deal with the following languages.
  \begin{enumerate}
  \item The \emph{three-sorted language of valued fields} $\mathcal L_{\mathrm{vf}}$ has
    \begin{enumerate}
    \item a sort $\mathrm{VF}$ for the valued field, carrying a copy of the language $\mathcal L_\mathrm{field}\coloneqq\set{+_{\mathrm{VF}},0_{\mathrm{VF}},-_{\mathrm{VF}},\cdot_{\mathrm{VF}},1_{\mathrm{VF}}, (-)\inverse_{\mathrm{VF}}}$ of fields, interpreted in the natural way with the convention that $0\inverse=0$,
    \item a sort $\Gamma$ for the value group, carrying a copy of the language $\mathcal L_\mathrm{oag}\coloneqq \set{+_{\Gamma},0_{\Gamma},-_{\Gamma},<_{\Gamma}}$ of oags and an additional constant symbol $\infty$, to be interpreted as an infinite absorbing element, with $-\infty=\infty$, equal to the valuation of $0_\mathrm{VF}$,
    \item a sort $k$ for the residue field, carrying a copy of the language of fields $\set{+_{k},0_{k},-_{k},\cdot_{k},1_{k}, (-)\inverse_{k}}$, interpreted in the natural way with the convention that $0\inverse=0$,
    \item a function symbol $v\from \mathrm{VF}\to \Gamma$,
    \item a function symbol $\operatorname{res}\from \mathrm{VF}\to k$.
    \end{enumerate}
    We view a valued field as an $\mathcal L_{\mathrm{vf}}$-structure by interpreting $v$ as the (surjective) valuation, $\operatorname{res}$ as the map sending $x$ to the residue of $x$ if $x$ lies in the valuation ring, that we denote by $\mathcal O$, and to $0$ otherwise, and the rest of the symbols in the obvious way suggested by the symbols themselves. 
  \item The \emph{language of $\mathrm{ac}$-valued fields} $\mathcal L_{\mathrm{ac}}$ is the expansion of $\mathcal L_{\mathrm{vf}}$ by a function symbol $\ac\from \mathrm{VF}\to k$, to be interpreted as an angular component map.
  \item The \emph{language of $s$-valued fields} $\mathcal{L}_{s}$ is the expansion of $\mathcal{L}_{\mathrm{vf}}$ by a function symbol $s\from \Gamma \to \mathrm{VF}$, to be interpreted as a cross-section of the valuation.
  \item\label{point:lsig} We write $\mathcal L_{\sigma}$ (resp.\ $\mathcal L_{\mathrm{ac},\sigma},\mathcal{L}_{s,\sigma}$) for the expansion of $\mathcal L_{\mathrm{vf}}$ (resp.\ $\mathcal L_{\mathrm{ac}},\mathcal{L}_{s}$) by symbols $\sigma_\mathrm{VF},\sigma_\mathrm{VF}\inverse, \sigma_\Gamma, \sigma_\Gamma\inverse, \sigma_k, \sigma_k\inverse$, interpreted as automorphisms of the sorts indicated in the subscripts and their inverses, and commuting with the other maps in the language (see below). 
  \item We also expand all of the above languages by a function symbol $\iota\from k\to \mathrm{VF}$, to be interpreted---in the equicharacteristic case---as a lift of the residue field commuting with $\sigma$. The resulting languages are denoted e.g.\ by $\mathcal L_{\ac,\iota,\sigma}$.
  \item If $\mathcal L$ is one of the languages introduced above, we will write \emph{$\mathcal{L}$-valued difference field} to mean a valued difference field expanded by the corresponding structure.
  \end{enumerate}
\end{defin}
As customary, we treat $\Gamma$ as an oag even if, strictly speaking, it also contains  $\infty$,  conflate $K$ with $\VF(K)$ whenever convenient, and drop subscripts if no confusion may arise. If $K$ is clear from context, we may also write $k$ in place of $k(K)$, etc. We also say e.g.\ ``$\sigma$ commutes with $v$'' instead of ``$\sigma_\Gamma\circ v=v\circ \sigma_\mathrm{VF}$'', as already done in \Cref{defin:languages}\ref{point:lsig}.
When dealing with extensions we may or may not use subscripts to differentiate the interpretation of a particular symbol, typically $\sigma$. For instance, we sometimes write $(K,\sigma)\subseteq (L,\sigma)$, and sometimes write $(K, \sigma_K)\subseteq (L, \sigma_L)$ for emphasis. For brevity, we write e.g.\ $(K,v,\sigma)$ to refer to a valued difference field $K$ viewed as an $\mathcal L_\sigma$-structure, even if $\mathcal L_\sigma$ contains other symbols, e.g.\ $\sigma\inverse$.

Recall that $\mathsf{DOAG}$ denotes the $\mathcal L_\mathrm{oag}$-theory of nontrivial divisible ordered abelian groups. We denote by $\mathsf{ACF}$ the $\mathcal L_\mathrm{field}$-theory of algebraically closed fields, by $\mathsf{ACVF}$  the $\mathcal{L}_{\mathrm{vf}}$-theory of nontrivially valued algebraically closed fields, by $\mathsf{ACVF}_\mathrm{ac}$  the $\mathcal{L}_{\mathrm{ac}}$-theory of nontrivially valued algebraically closed fields equipped with an angular component map, and similarly for the other languages mentioned above. For $p,q$ prime or zero, we write $\mathsf{ACVF}_
{p,q}$ for the completion of $\mathsf{ACVF}$ obtained by specifying that $\operatorname{char}(\mathrm{VF})=p$ and $\operatorname{char}(k)=q$. As customary, in the aforementioned theories of valued fields, the valuation map $v\from \mathrm{VF}\to \Gamma$ is assumed surjective.
 All of these theories are easily checked to be $\forall\exists$-axiomatisable in their language, hence by \Cref{rem:AEbasics} each of their models embeds in an existentially closed one.\footnote{In fact, all of  $\mathsf{DOAG}$, $\mathsf{ACF}$, $\mathsf{ACVF}$, $\mathsf{ACVF}_\mathrm{ac}$ are \emph{model complete}, that is, all of their models are existentially closed.} After adding $\sigma, \sigma^{-1}$ to the languages above, together with axioms saying that these are respectively an automorphism and its inverse, $\forall \exists$-axiomatisability is clearly preserved. 

\begin{rem}\label{rem:vflanguages}
Our choice of language is slightly different from the usual one, in that we have symbols for multiplicative inverses on the sorts $\mathrm{VF}$ and $k$ and a unary $\res$ map, while most of the literature uses the language of rings and a binary map $\mathrm{Res}(x,y)$ defined as $\res(x/y)$ if $y\ne 0$ and $x/y\in \mathcal O$, and $0$ otherwise. Once a symbol for the multiplicative inverse is available, $\mathrm{Res}(x,y)$ is piecewise definable by a term, and each piece is quantifier-free definable, hence many results from the literature, e.g.\ on quantifier elimination, are easily shown to transfer.
\end{rem}
\begin{eg}\label{eg:Hahnsigma}
  Given $(k,\sigma_k)$ and $(\Gamma,\sigma_\Gamma)$, the Hahn field $k((\Gamma))$ can be equipped with the structure of a valued difference field, with the natural valuation and $\sigma\from \VF(L)\to \VF(L)$ given by $(\sigma (f))(\gamma)\coloneqq \sigma_{k}(f(\sigma^{-1}_{\Gamma}(\gamma)))$ or, in power series notation,
  \[
    \sigma\left(\sum_{\gamma\in \Gamma} a_\gamma t^\gamma\right)=\sum_{\gamma\in \Gamma} \sigma_k(a_\gamma) t^{\sigma_\Gamma(\gamma)}.
  \]
As a special case, consider the field of complex Puiseux series with the automorphism sending $\sum_{n\in \mathbb Q} a_nt^n$ to $\sum_{n\in \mathbb Q} \overline{a_n}t^{2n}$, where $\bar z$ denotes the complex conjugate of $z$.
  
  Observe that such structures have a natural cross-section, namely $\gamma\mapsto t^\gamma$, which is easily checked to commute with $\sigma$. In particular, the associated angular component, which maps $f=\sum_{\gamma\in \Gamma} a_\gamma t^\gamma$ to $a_{v(f)}$, is $\sigma$-equivariant.
\end{eg}

\subsection{Tropical geometry}
In what follows, we will use results from tropical geometry, for which we refer the reader to~\cite{MS}. In the cited source, all value groups are Archimedean, which is typically not the case in this paper. Nevertheless, all the results we will use from~\cite{MS} still hold in our setting; since this sometimes requires justification, and because it may be of independent interest, we gather in this subsection the results that we will need. We especially want to highlight that many of these are, or can be reduced to, first-order statements; therefore, by the model-theoretic properties of $\mathsf{ACVF}$, they do not depend on the rank of the valuation.

For the rest of this section we use $K[x^{\pm 1}], \mathcal{O}[x^{\pm 1}], k[x^{\pm 1}]$ as shorthands for $K[x_1^{\pm 1}, \ldots, x_n^{\pm 1}], \mathcal{O}[x_1^{\pm 1}, \ldots, x_n^{\pm 1}], k[x_1^{\pm 1}, \ldots, x_n^{\pm 1}]$ respectively, for some fixed $n$. For $\mathbf{u}=(u_1,\ldots,u_n) \in \mathbb{Z}^n$ and $\gamma=(\gamma_1,\ldots,\gamma_n) \in \Gamma^n$, we will use $\langle \mathbf{u}, \gamma \rangle$ to denote the scalar product $u_1 \gamma_1 +\ldots + u_n \gamma_n$.

\begin{defin}
	Let $(K,v)$ be a valued field.
	\begin{enumerate}
        \item  A \emph{tropical polynomial} is an expression of the form $\min\set{\gamma_i+\langle\mathbf{u}_i,x\rangle\mid i\le \ell}$, where $x$ is a variable ranging over $\Gamma^n$,  $\mathbf{u}_i\in \mathbb Z^n$, and $\gamma_i\in \Gamma$.
        \item We call $\gamma\in \Gamma^n$ a \emph{tropical root} of a tropical polynomial $g$  iff the minimum in $g$ is achieved at least twice at $\gamma$.
        \end{enumerate}
 Let $f \in K[x^{\pm 1}]$ be a Laurent polynomial; write $f$ as $\sum_{\mathbf{u} \in S} c_{\mathbf{u}}x^{\mathbf{u}}$, for some finite subset $S \subseteq \mathbb{Z}^n$, with $c_{\mathbf{u}} \neq 0$ for all $\mathbf{u} \in S$, and let $\gamma \in \Gamma^n$.
        \begin{enumerate}[resume]
        \item The \emph{tropicalisation} $\trop(f)\from \Gamma^n\to \Gamma$ is defined as \[(\trop(f))(\gamma)\coloneqq\min\set{v(c_\mathbf{u})+\langle \mathbf{u},\gamma\rangle\mid \mathbf{u} \in S}.\]
        \end{enumerate}
         Fix now  a cross-section of the valuation $s\from \Gamma \rightarrow K$.
          \begin{enumerate}[resume]
      \item 	The \emph{initial form of $f$ with respect to $\gamma$} is the Laurent polynomial over the residue field $k$ defined as $\ini_\gamma(f)\coloneqq \sum_{\mathbf{u} \in S'} \res(c_{\mathbf{u}}s(-v(c_{\mathbf{u}}))) x^{\mathbf{u}}$, where $S'\coloneqq \{\mathbf{u} \in S \mid  \forall \mathbf{u}' \in S, v(c_{\mathbf{u}})+\langle \mathbf{u}, \gamma \rangle \leq v(c_{\mathbf{u}'}) + \langle \mathbf{u}', \gamma \rangle \}$.	
		\item 	Given an ideal $I$ of $K[x^{\pm 1}]$, the \emph{initial ideal of $I$ with respect to $\gamma$} is the ideal of $k[x^{\pm 1}]$ that is generated by the set $\{\ini_\gamma(f) \mid f \in I \}$.
		\item 	If $V \subseteq \mathbb{G}_{\mathrm{m}}^n(K)$ is an algebraic subvariety and $I$ is the ideal of Laurent polynomials over $K$ vanishing on $V$, then we denote by $\ini_\gamma(V)$ the algebraic subvariety of $\mathbb{G}_{\mathrm{m}}^n(k)$ associated to $\ini_\gamma(I)$. We call this the \emph{initial variety of $V$ with respect to $\gamma$}.
	\end{enumerate}
\end{defin}
Note that initial forms depend on the choice of a cross-section, even if this is suppressed from the notation. Below, if we mention initial forms without specifying a cross-section, we assume implicitly that a one exists and has been fixed. Recall that cross-sections always exist if the value group is divisible, e.g.\ in the algebraically closed case, see for example \cite[Lemma 3.3.32]{aschenbrennerAsymptoticDifferentialAlgebra2017}. 
\begin{rem}\*\label{rem:svsac}
	\begin{enumerate}
		\item In order to define initial forms one does not really need the cross-section $s$, rather the corresponding angular component $\res(xs(-v(x)))$. We mostly use the cross-section as that is customary in the literature, but observe that many results in tropical geometry (say from \cite[Sections 2 and 3]{MS}) generalise to initial forms when they are defined using an arbitrary angular component.
		\item Initial varieties are sometimes called \emph{tropical reductions} or \emph{tropical degenerations} in the literature.
	\end{enumerate}
      \end{rem}

It is almost immediate to see that the codimension 1 case of the Fundamental Theorem of Tropical Geometry, also known as Kapranov's Theorem \cite[Theorem 3.1.3]{MS}, is a first-order statement. In higher codimension we will need a bound of Joswig and Schr\"oter on the minimum degree of a \emph{tropical basis}.

\begin{defin}
	Let $K$ be a valued field.
        \begin{enumerate}
        \item Let $I$ be a homogeneous ideal of $K[\bla x0,n]$. A \emph{universal Gr\"obner basis} $\mathcal U$ of $I$ is a finite subset of $I$ such that, for every $\gamma\in \Gamma(K)^{n+1}$, the set  $\ini_\gamma(\mathcal U)$ generates $\ini_\gamma(I)$.
        \item  Let $I$ be an ideal in $K[x^{\pm 1}]$. A \emph{tropical basis} of $I$ is a finite generating subset $\mathcal{T} \subseteq I$ such that for every $\gamma \in \Gamma^n$, if there is $f \in I$ such that $\gamma$ is not a tropical root of $\trop(f)$, then there is $g \in \mathcal{T}$ such that $\gamma$ is not a tropical root of $\trop(g)$.
        \end{enumerate}
      \end{defin}

      Given the ideal $I$, denote by $I_{\mathrm{proj}}$ its homogenisation: the homogeneous ideal in $K[x_0,\ldots,x_n]$ obtained by homogenising the elements of $K[x^{\pm 1}] \cap I$. 

\begin{pr}\label{fact:trop_basis_deg_bound}
Let $K$ be an algebraically closed valued field with value group $\mathbb Q$. For all finite $S\subseteq\mathbb{Z}^n$ there are $d\in \mathbb N$ and a finite  $H\subseteq\mathbb{N}^n$ such that the following holds. Let  $I\subseteq K[x^{\pm 1}]$ be an ideal and $\mathcal G$ a finite set generating of $I$ such that every polynomial in $\mathcal{G}$ can be written as $\sum_{\mathbf{u} \in S} c_{\mathbf{u}}x^{\mathbf{u}}$ (with possibly some $c_{\mathbf{u}}=0$).
\begin{enumerate}
\item There is a universal Gr\"obner basis $\mathcal U$ of $I_\mathrm{proj}$ such that every polynomial in $\mathcal{U}$ can be written as $\sum_{\mathbf{u} \in \set{0,\ldots,d}\times H} b_{\mathbf{u}}x^{\mathbf{u}}$.
\item There is a tropical basis $\mathcal T\supseteq \mathcal G$ of $I$ such that every polynomial in $\mathcal{T}$ can be written as $\sum_{\mathbf{u} \in H} b_{\mathbf{u}}x^{\mathbf{u}}$. 
 \end{enumerate}
\end{pr}

\begin{proof}
Because $\Gamma(K)=\mathbb Q$, every ideal of $K[x_0,\ldots,x_n]$ is defined over a subfield of $K$ with discrete valuation. Hence we may apply  \cite[Theorem~10]{joswigDegreeTropicalBasis2018} to $I_{\mathrm{proj}}$ and obtain that there are $d$, depending only on $n$ and on the minimal degree of a set of generators of $I_{\mathrm{proj}}$, a tropical basis $\mathcal{T}_0$, and a universal Gr\"obner basis $\mathcal U$ of $I_{\mathrm{proj}}$, both consisting of polynomials of degree at most $d$. We can then take $\mathcal{T}$ to be the union of $\mathcal G$ with the set of polynomials from $\mathcal{T}_0$ evaluated in $x_0=1$; we will show that $\mathcal{T}$ is a tropical basis, from which it follows that we can take as $H$ the set of exponents $[0,d]^n\subseteq \mathbb N^n$.

Equip $K$ with a cross-section, so that we may define initial forms. By \cite[Theorem 3.2.3]{MS}, $\mathcal T$ is a tropical basis if and only if, for every $\gamma$, whenever there is a monomial in $\ini_\gamma(I)$, then there is some $g\in \mathcal T$ such that $\ini_\gamma(g)$ is a monomial.

Let $\gamma \in \Gamma^n$, and assume that $\ini_\gamma(I)$ contains a monomial. Then by \cite[Proposition 2.6.1]{MS}, $\ini_{(0,\gamma)}(I_{\mathrm{proj}})$ also contains a monomial, so there is $f \in \mathcal{T}_0$ such that $\ini_{(0,\gamma)}(f)$ is a monomial. Since $\ini_{(0,\gamma)}(f(x_0,\ldots,x_n))=\ini_\gamma(f(1,x_1,\ldots,x_n))$ as shown in the proof of \cite[Proposition 2.6.1]{MS}, $\mathcal{T}$ is then a tropical basis.
\end{proof}

\begin{fact}\label{fact:262}
Let $I$ be an ideal in $K[x^{\pm 1}]$ and $\gamma\in \Gamma^{n}$. For every $g\in \ini_\gamma(I)$ there is $h\in I$ such that $g=\ini_\gamma(h)$.
\end{fact}
\begin{proof}
This is \cite[Lemma~2.6.2(1)]{MS}. One sees by inspection that in its proof (and in the proofs of Lemma~2.4.2 and Proposition~2.6.1, which are used within it) the rank $1$ assumption is never used.
\end{proof}
\begin{thm}\label{thm:fundamental-arbitrary-rank}[Fundamental Theorem of Tropical Geometry for valuations of arbitrary rank]
	Let $K$ be an algebraically closed valued field and $I \subseteq K[x^{\pm 1}]$ an ideal. Then the following properties are equivalent for all $\gamma \in \Gamma^n$.
	\begin{enumerate}
		\item\label{funtrop1} For all $f \in I$, the point $\gamma$ is a tropical root of $\trop(f)$.
		\item\label{funtrop2} The ideal $\ini_\gamma(I) \subseteq k[x^{\pm 1}]$ is proper.
		\item\label{funtrop3} There is $z \in (K^\times)^n$ such that $f(z)=0$ for every $f \in I$ and $v(z)=\gamma$.
	\end{enumerate}
\end{thm}
\begin{proof}
  If $K$ is trivially valued then the statement is obvious, so we may assume the valuation is nontrivial.
  
The implications $\ref{funtrop2}\allora \ref{funtrop1}$ and $\ref{funtrop3}\allora \ref{funtrop1}$ are obvious, and $\ref{funtrop1}\allora \ref{funtrop2}$ follows from \Cref{fact:262}.

For $\ref{funtrop1}\allora \ref{funtrop3}$, let $V \subseteq \mathbb{G}_\mathrm{m}^n$ be the algebraic subvariety defined by $I$. Fix a set $\mathcal{G}$ of generators of $I$, with support $S$. Let $H \subseteq \mathbb{N}^n$ be given by \Cref{fact:trop_basis_deg_bound} applied to $S$ and to an arbitrary algebraically closed valued field $L$ of the same characteristic $(p,q)$ with value group $\mathbb Q$. 
We identify a finite family of Laurent polynomials in $K[x_1^{\pm1},\ldots,x_n^{\pm1}]$ with exponents in $S$ with the tuple $t$ of their coefficients; write $\mathcal G_t$ for such a family and $I_t$ for the ideal it generates.

Let $\phi_H(t,\gamma)$ be the formula saying that for every tuple $(y_\mathbf{u})_{\mathbf{u} \in H}$, if the polynomial $f=\sum_{\mathbf{u} \in H}y_{\mathbf{u}}x^{\mathbf{u}}$ vanishes on all $z$ such that all elements of $\mathcal{G}_t$ vanish on $z$ (this is a finite conjunction expressing the fact that $f$ is in the radical of $I_t$), then $\gamma$ is a tropical root of $\trop(f)$. In an arbitrary valued field, for all $t$ and $\gamma$, the implication $\ref{funtrop1}\allora \phi_H(t,\gamma)$ holds trivially.
  
By \Cref{fact:trop_basis_deg_bound}, in $L$ we have that  $\phi_H(t,\gamma)$ holds if and only if $\gamma$ satisfies property~\ref{funtrop1} for $I_t$; by the Fundamental Theorem of Tropical Geometry~\cite[Theorem 3.2.3]{MS}, in rank $1$ properties~\ref{funtrop1} and~\ref{funtrop3} are equivalent, so in $L$ we have that  $\phi_H(t, \gamma)$ holds if and only if $\gamma$ satisfies property~\ref{funtrop3} for $I_t$. This equivalence is a first-order statement, so it holds in  algebraically closed nontrivially valued fields of arbitrary rank by completeness of $\mathsf{ACVF}_{p,q}$. 
\end{proof}

The following proposition, which generalises a step of the proof of the Fundamental Theorem of Tropical Geometry \cite[Proposition 3.2.11]{MS}, will be needed later on. We consider initial forms as computed with respect to a given angular component map, see \Cref{rem:svsac}.

\begin{pr}\label{pr:zardense}
  Let $V \subseteq \mathbb{G}_\mathrm{m}^n$ be an algebraic subvariety defined over an algebraically closed nontrivially $\mathrm{ac}$-valued field $K$, and $\gamma \in \Gamma^n$.
  \begin{enumerate}
  \item\label{point:iniac} The variety $\ini_\gamma(V)(k(K))$ coincides with the set of angular components of points of $V(K)$ of valuation $\gamma$.
  \item\label{point:aczardense}  For each $\alpha\in \ini_\gamma(V)(k(K))$, the set of $a \in V(K)$ such that $v(a)=\gamma$ and $\ac(a)=\alpha$ is Zariski dense in $V(K)$.
  \end{enumerate}
\end{pr}

\begin{proof}
Let $I \subseteq K[x^{\pm 1}]$ be the ideal associated to $V$.
Let $d$ and $H$ be given by \Cref{fact:trop_basis_deg_bound} applied to the set $S$ of exponents of a fixed basis of $I$. Similarly as in the proof of \Cref{thm:fundamental-arbitrary-rank}, write $I=I_t$ where $t$ is the tuple of coefficients of a set of generators with exponents in $S$.
Let $J_{t,\gamma}$ be the ideal generated by $\ini_\gamma(f)$, as $f$ varies amongst polynomials in $I_t$ with exponents in $H$. In every nontrivially $\ac$-valued field, for every $t$ and $\gamma$, it follows from the definition of initial form that
\[
\set{\ac(a)\mid a\in V(I_t), v(a)=\gamma}\subseteq
V(\ini_\gamma(I_t))\subseteq
V(J_{t,\gamma}).
\]
By \cite[Proposition 3.2.11]{MS}, in every algebraically closed $\mathrm{ac}$-valued field of rank 1, the first inclusion above is an equality. If furthermore $\Gamma(K)=\mathbb Q$, then by \Cref{fact:trop_basis_deg_bound} there is a universal Gr\"obner basis of $I_\mathrm{proj}$ with exponents in $\set{0,\ldots,d}\times H$, and it follows that the second inclusion is also an equality.

As the sets $\set{\ac(a)\mid a\in V(I_t), v(a)=\gamma}$ and $V(J_{t,\gamma})$ are first-order definable uniformly in $t$ and $\gamma$, completeness of $\mathsf{ACVF}_{p,q,\mathrm{ac}}$ yields that they coincide in every algebraically closed nontrivially $\ac$-valued field. As $V(\ini_\gamma(I_t))$ lies between them, this proves~\ref{point:iniac}.

As for~\ref{point:aczardense}, fix a polynomial $g \in K[x^{\pm 1}]$. The conclusion follows from the previous point and the first-order expressibility of ``if there are points $z_0, z_1\in V$ with $v(z_0)=\gamma$, $\ac(z_0)=\alpha$ and  $g(z_1) \neq 0$, then there is a point $z_2 \in V$ such that $v(z_2)=\gamma$, $\ac(z_2)=\alpha$, and $g(z_2) \neq 0$''.
\end{proof}

\begin{lemma}\label{lemma:initial-of-pol-over-O}
	Let $I$ be an ideal in $K[x^{\pm 1}]$. Then $\ini_0(I)=\res(I \cap \mathcal{O}[x^{\pm 1}])$. 
\end{lemma}

\begin{proof}
	$(\supseteq)$ Let $f \in \res(I \cap \mathcal{O}[x^{\pm 1}])$; if $f=0$ then there is nothing to prove, so assume $f\neq 0$. Then there is $g \in I \cap \mathcal{O} [x^{\pm 1}]$ such that $f=\res(g)$. Since $\res(g) \neq 0$, $g$ has at least one coefficient of valuation $0$, and it follows from the definitions that $f=\res(g)=\ini_0(g)$.
	
	$(\subseteq)$ Let $f \in \ini_0(I)$. By \Cref{fact:262} there is $g \in I$ such that $f=\ini_0(g)$, say  $g=\sum_{\mathbf{u} \in S}c_\mathbf{u}x^\mathbf{u}$. Let  $S' \subseteq S$ be the set of $\mathbf{u}$ for which $v(c_\mathbf{u})$ is minimal. If $\mathbf{u}_0 \in S'$, then \[\ini_0(g)=\sum_{\mathbf{u} \in S'}\res(c_\mathbf{u}s(-v(c_\mathbf{u})) x^\mathbf{u}=\res\left( \frac{g}{s(v(c_{\mathbf{u}_0}))} \right).\]
	
	Now $g/(s(v(c_{\mathbf{u}_0}))) \in I \cap \mathcal{O}[x^{\pm 1}]$, and the minimal valuation of its coefficients is $0$, so we have \[\ini_0 \left(\frac{g}{s(v(c_{\mathbf{u}_0}))} \right)=\res \left(\frac{g}{s(v(c_{\mathbf{u}_0}))}\right) =\ini_0(g)=f.\qedhere\]
\end{proof}

\begin{pr}\label{prop:sigma-commutes-ini0}
	Let $I$ be an ideal in $K[x^{\pm 1}]$ and $\sigma$ an automorphism. Then $\sigma(\ini_0(I))=\ini_0(\sigma(I))$.
\end{pr}

\begin{proof}
	By \Cref{lemma:initial-of-pol-over-O}, it is enough to show $\sigma(\res(I \cap \mathcal{O}[x^{\pm 1}])) = \res (\sigma(I) \cap \mathcal{O}[x^{\pm 1}])$. This follows easily from $\sigma \circ \res=\res \circ \sigma$.
\end{proof}

\section{Purity and reducts}
\label{sec:purity}

As already recalled, every algebraically closed valued field admits a cross-section of the valuation. In this section, we show that universal domains for tropical difference algebra admit one that is $\sigma$-equivariant.

For this purpose, we endow the multiplicative group of a valued difference field $(K,v,\sigma)$ with a module structure over the polynomial ring $\mathbb{Z}[\sigma]$ by defining $\left( \sum_{i=0}^n m_i \sigma^i \right) \cdot z= \prod_{i=0}^{n} \sigma^i(z^{m_i})$. We show that, assuming existential closedness, $\mathcal{O}^\times$ is a \emph{pure $\mathbb{Z}[\sigma]$-submodule} of $K^\times$, and we use standard arguments to deduce the existence of a cross-section commuting with $\sigma$ if enough saturation is present. 

Recall that a formula is \emph{existential} iff it is of the form $\exists x\; \phi(x)$ with $\phi(x)$ quantifier-free, and \emph{positive} (resp.\ \emph{positive primitive}) iff it is obtained from atomic formulas by using $\land,\lor,\exists$ (resp.\ $\land, \exists$).

\begin{defin}\*\label{defin:saturated}
	\begin{enumerate}
		\item Let $M$ be an $\mathcal L$-structure. A \emph{partial type} $\pi(x)$ is a set of $\mathcal L$-formulas with parameters from $M$ such that, for every $n$ and every $\phi_0(x),\ldots, \phi_n(x)\in \pi(x)$, there is $a\in M$  such that, for every $i\le n$, we have $M\models \phi_i(a)$.
		\item The \emph{quantifier-free type} $\qftp(a/B)$ of a tuple $a$ in an $\mathcal L$-structure $M$ over  a set $B\subseteq M$ is the set of quantifier-free $\mathcal L$-formulas $\phi(x)$ with parameters from $B$ such that $M\models\phi(a)$.
		\item Let $\kappa$ be an infinite cardinal.  A structure $M$ is \emph{existentially} (resp.\ \emph{positively}, resp.\ \emph{quantifier-free}) \emph{$\kappa$-saturated}  iff, whenever $\pi(x)$ is a partial type consisting of existential (resp.\ positive, resp.\ quantifier-free) formulas  with parameters from a set $A\subseteq M$ of size $\abs A<\kappa$, there is $a\in M$ satisfying all formulas in $\pi(x)$.
	\end{enumerate}
\end{defin}
So, for example,  the $\mathbb Z[\sigma]$-module $\mathcal O^\times$ is positively $\aleph_1$-saturated if and only if, for every system of countably many equations in the language of $\mathbb Z[\sigma]$-modules with parameters from $\mathcal O^\times$, if every finite subsystem has a solution in $\mathcal O^\times$  then the whole system does.\footnote{Note that every partial type consisting of positive formulas is implied by a partial type consisting of positive primitive formulas.} This is in particular the case if $(K,v,\sigma)$ is existentially $\aleph_1$-saturated.

\begin{rem}\label{rem:ecsatexist}
  By standard model-theoretic arguments, for every cardinal $\kappa$, every valued difference field has an existentially $\kappa$-saturated existentially closed extension.
\end{rem}

\begin{pr}\label{pr:ecbasics}
	In every existentially closed $(K,v,\sigma)$, the following statements are true.
	\begin{enumerate}
		\item\label{point:ecbasicsACF}  The field $\operatorname{VF}(K)$ is algebraically closed.
		\item\label{point:ecbasicsOAGA} The difference oag $\Gamma(K)$ is existentially closed.
		\item\label{point:ecbasicsACFA} 
		The difference field $k(K)$ is existentially closed.
	\end{enumerate}
\end{pr}
\begin{proof}
As commonly done in model theory, we embed the valued field $(K,v)$ into a \emph{$\abs K^+$-strongly homogeneous} valued field $\monster$, that is, one where every \emph{elementary map} with domain of size at most $\abs K$ extends to an automorphism. Recall that, whenever a theory has \emph{quantifier elimination}, an elementary map is nothing more than a partial isomorphism, and that quantifier elimination holds for $\mathsf{ACVF}$ in $\mathcal L_\mathrm{vf}$, see~\cite[Theorem~2.1.1(iii)]{haskellDefinableSetsAlgebraically2006}\footnote{Our language is slightly different, see \Cref{rem:vflanguages}.}.
	\begin{enumerate}[wide=0pt]
		\item As $(K,v,\sigma)$ is existentially closed, it follows easily from the definitions that, in order to prove its algebraic closedness, it suffices to embed it into an algebraically closed valued difference field. By quantifier elimination, $\sigma$ may be regarded as an elementary map $\monster\to \monster$, and by $\abs{K}^+$-strong homogeneity we may extend it to some $\sigma_{\monster}\in \aut(\monster)$, thus obtaining the required extension. 
		\item\label{point:ecbasics2proof}
		View again $\sigma$ as an elementary map $\monster\to \monster$. Let $(\Gamma', \sigma')\supseteq (\Gamma(K), \sigma_{\Gamma(K)})$. 
		Since $\Gamma(\monster)\models \mathsf{DOAG}$,\footnote{In characteristic $(0,p)$ we need to add to the language on $\Gamma$ a constant symbol, interpreted as $v(p)$. We leave to the reader the easy task of modifying this part of the proof accordingly.} up to replacing $\monster$ by a larger $\monster'$, we may embed $\Gamma'$ in it over $\Gamma(K)$.
		Again by quantifier elimination, $\sigma_K\cup \sigma'$ is an elementary map, hence extends to an automorphism of $\monster$. 
		Every existential formula over $\Gamma(K)$ realised in $\Gamma'$ is as well realised in $\Gamma(\monster)$.  By existential closedness of $K$, it must be also realised in $\Gamma(K)$. This shows that $\Gamma(K)$ is existentially closed.
		
		\item This can be proved similarly to~\ref{point:ecbasics2proof}, using $\mathsf{ACF}$ in place of $\mathsf{DOAG}$.\qedhere
	\end{enumerate}
\end{proof}

\begin{rem}\label{rem:stuffforfree}\*
  \begin{enumerate}
  \item\label{point:surjforfree} In our framework, the residue and valuation maps are assumed surjective from the start. In fact, as far as existentially closed models are concerned, this is not necessary, as surjectivity of these maps can be proved with arguments similar to those of the proof of \Cref{pr:ecbasics}. We thank Ehud Hrushovski for pointing this out. 
  \item\label{point:inversiveforfree} Every possibly non-inversive difference field embeds canonically in its inversive closure. In particular, even if we work with non-inversive valued difference fields, the existentially closed ones are going to be automatically inversive. We thank Simone Ramello for pointing this out.
  \end{enumerate}
\end{rem}

\subsection{Purity}

Recall that, if $R$ is a ring, a submodule $N$ of an $R$-module $M$ is \emph{pure} iff, for every matrix $A \in \operatorname{Mat}_{d \times n}(R)$ and vector $b \in N^d$, if the equation $A \cdot x =b$ has a solution in $M$ then it has a solution in $N$. This passes to localisations, and is in fact a local property, see~\cite[§4, Exercise~34]{lamLecturesModulesRings1999}. The case of interest for us is that purity as a $\mathbb Z[\sigma]$-module implies purity as a $\mathbb Z[\sigma,\sigma\inverse]$-module, hence it suffices to work with $\mathbb Z[\sigma]$.

\begin{rem}\label{rem:OxpureinKx}
  It follows easily from the definitions that, since $\Gamma=K^\times/\mathcal O^\times$ is torsion-free, $\mathcal O^\times$ is a pure abelian subgroup of $K^\times$.
\end{rem}

The following example shows that, if we do not assume existential closedness, purity may indeed fail.

\begin{eg}
Consider the \emph{$\mathbb R$-exponentiation} (see \cite[Definition~3.3.1]{kuhlmannAutomorphismGroupValued2022}\footnote{The reader familiar with this work will see that, if we decompose the automorphism group of $\mathbb C ((\mathbb R))$ as in \cite[Theorem~3.7.1]{kuhlmannAutomorphismGroupValued2022} with respect to the natural cross-section, then any counterexample to purity must have a nontrivial factor in $\operatorname{Hom}(\mathbb R, \mathbb C^\times)$.})  $\sigma\from \mathbb C((\mathbb R))\to \mathbb C((\mathbb R))$ given by \[\sum c_rt^r\mapsto \sum e^rc_rt^r.\] 

In this example, $\mathcal O^\times$ is not a pure $\mathbb Z[\sigma]$-submodule of $K^\times$:  the equation $\sigma(x)x\inverse=et^0$ has parameters in $\mathcal O^\times$ and is solved by $t^1$, but there are no solutions in $\mathcal O^\times$ because $\sigma$ induces the identity on the residue field.

In particular, this automorphism admits no equivariant angular component, hence no equivariant cross-section. 
\end{eg}

\begin{notation}\*
  \begin{enumerate}
  \item For $r \leq s \in \mathbb{Z}$, we denote by $[r,s]$ the interval $\{h \in \mathbb{Z} \mid r \leq h \leq s\}$.  
  \item For a tuple $a$ with all coordinates in the domain of an automorphism $\sigma$ and $r\leq s \in \mathbb{Z}$ we will write $\orb_{r,s}(a)$ to denote the tuple $(\sigma^r(a),\ldots,\sigma^s(a))$.
  \end{enumerate}
\end{notation}

Note that, if $\mathbf{u}_1,\ldots,\mathbf{u}_{n_2} \in \mathbb{Z}^{n_1}$ and $\psi\from \mathbb{G}_{\mathrm{m}}^{n_1} \rightarrow \mathbb{G}_{\mathrm{m}}^{n_2}$ is the \emph{monomial map} $z \mapsto (z^{\mathbf{u}_1},\ldots,z^{\mathbf{u}_{n_2}})$, then $\trop(\psi)\from \Gamma^{n_1} \rightarrow \Gamma^{n_2}$ is the morphism \[\gamma \mapsto \left(\langle \mathbf{u}_1, \gamma \rangle, \ldots, \langle \mathbf{u}_{n_2}, \gamma \rangle \right).\]

\begin{lemma}\label{lemma:matrix-to-coset}
Let $K$	be a valued difference field, let $A \in \operatorname{Mat}_{d \times n}(\mathbb{Z}[\sigma])$, and let $b \in (\mathcal{O}^\times)^d$. Assume there is $\beta \in (K^\times)^n$ such that $A \cdot \beta=b$.
	
	There are $\ell \in \mathbb{N}$, an algebraic subgroup $G$ of $\mathbb{G}_{\mathrm{m}}^{n(\ell+1)}$, and $t \in (\mathcal{O}^\times)^{n(\ell+1)}$  such that, for all $z \in (K^\times)^n$, we have $A\cdot z=b$ if and only if $\orb_{0,\ell}(z) \in t \cdot G$.
\end{lemma}

\begin{proof}
	For each entry $a_{ij}$ of $A$, there are $\ell_{ij} \in \mathbb{Z}_{\geq 0}$  and $m_{ij0},\ldots, m_{ij\ell_{ij}} \in \mathbb{Z}$ such that 
	\[
	a_{ij}=m_{ij\ell_{ij}}\sigma^{\ell_{ij}}+m_{ij(\ell_{ij}-1)}\sigma^{\ell_{ij}-1}+\ldots+m_{ij1}\sigma+m_{ij0}.
	\]
	Let $\ell\coloneqq \max_{ij}\ell_{ij}$. For $i\in [1,d]$, let $C_i \in \operatorname{Mat}_{n \times (\ell+1)}(\mathbb{Z})$ be the matrix
	\[
	C_i=(c_{ijh})_{j\in[1,n], h\in[0, \ell]}
	\]
	in which the entry $c_{ijh}$ is the (integer, possibly $0$) coefficient of $\sigma^{h}$ in the entry $a_{ij}$ of $A$ (see \Cref{eg:C} below). For $i\in[1,d]$, let $p_i((y_{jh})_{j\in[1,n],h\in[0,\ell]})$ be the monomial 
	\[p_i=\prod_{j=1}^n \prod_{h=0}^\ell y_{jh}^{c_{ijh}}. \] 
	
	Let $G\coloneqq \{y \in \mathbb{G}_{\mathrm{m}}^{n(\ell+1)} \mid \forall i\in[1, d]\; p_i(y)=1   \}$ and $H\coloneqq \{y \in \mathbb{G}_{\mathrm{m}}^{n(\ell+1)} \mid \forall i\in[1, d]\; p_i(y)=b_i\}$. Then $G$ is an algebraic subgroup of $\mathbb{G}_{\mathrm{m}}^{n(\ell+1)}$, and $H$ is a coset of $G$. By construction, for all $z \in (K^\times)^n$, we have  $A \cdot z=b$ if and only if $p_i(\orb_{0,\ell}(z))=b_i$ for all $i\in[1, d]$, if and only if $(\sigma^i(z_j))_{i\in[0,\ell],j\in[1,n]} \in H$. 
	
	Since $A \cdot \beta =b$, we have that $(\orb_{0,\ell}(\beta)) \in H(K)$. 
Moreover, by \Cref{rem:OxpureinKx}, there is some $t \in (\mathcal{O}^\times)^{n(\ell+1)}$ such that $H=t \cdot G$.
\end{proof}

\begin{eg}\label{eg:C}
	Let $A=\begin{pmatrix} 1 -\sigma\\ 1-\sigma^2 \end{pmatrix}$ and $b=\begin{pmatrix} 1 \\ 1 \end{pmatrix}$. Then the matrices $C_1$ and $C_2$ are $C_1 = \begin{pmatrix} 1 & -1 & 0 \end{pmatrix}$ and $C_2=\begin{pmatrix} 1 & 0 &-1 \end{pmatrix}$, and for all $a \in K^\times$ we have that $A \cdot a=b$ if and only if $(a,\sigma(a),\sigma^2(a))$ lies in the algebraic subgroup of $(K^\times)^3$ defined by $z_1z_2^{-1}=1\land z_1z_3^{-1}=1$.
\end{eg}

\begin{lemma}\label{lem:monomial-map}
	Let $\psi\from \mathbb{G}_{\mathrm{m}}^{n_1} \rightarrow \mathbb{G}_{\mathrm{m}}^{n_2}$ be a monomial map, defined by $\psi(z)=(z^{\mathbf{u}_1},\ldots,z^{\mathbf{u}_{n_2}})$ for appropriate vectors $\mathbf{u}_1,\ldots,\mathbf{u}_{n_2} \in \mathbb{Z}^{n_1}$.
	
	There is a monomial map $\xi\from \mathbb{G}_{\mathrm{m}}^{n_1} \rightarrow \mathbb{G}_{\mathrm{m}}^{n_2}$, defined by $\xi(z)=(z^{\mathbf{v}_1},\ldots,z^{\mathbf{v}_{n_2}})$, such that $\ker\xi$ is the connected component of the identity of $\ker\psi$ and that each $\mathbf{v}_i$ is a rational multiple of an element of the group generated by $\mathbf{u}_1,\ldots,\mathbf{u}_{n_1}$.
\end{lemma}

\begin{proof}
It suffices to let $\bla {\mathbf{v}}1,{n_2}$ be generators of the unique primitive lattice containing the one generated by $\bla {\mathbf{u}}1,{n_1}$ and of the same rank. See~\cite[Lemma~A.7]{scheferCountingTorsionPoints2025}.
\end{proof}

\begin{lemma}\label{irr}
	Let $G \leq \mathbb{G}_{\mathrm{m}}^n$ be a connected algebraic subgroup. For all $t \in (\mathcal{O}^\times)^n$, the initial variety $\ini_0(t \cdot G)$ is irreducible.
\end{lemma}

\begin{proof}
By~\cite[Theorem~A.13]{scheferCountingTorsionPoints2025} and connectedness, $G$ is defined by a set of equations of the form $z^{\mathbf{u}}-1=0$ and the associated lattice $\Lambda\subseteq \mathbb Z^n$ is primitive. The subvariety $\ini_0(t\cdot G)$ of $\mathbb G_\mathrm{m}^n(k)$ is contained in the coset by $\res(t)$ of the group $H$ defined by the equations $\ini_0(z^{\mathbf{u}} -1)=0$, and it is easily checked that $\ini_0(z^{\mathbf{u}} -1)=z^{\mathbf{u}} -1$. The lattice $\Lambda'$ associated to $H$ must satisfy $\Lambda'\supseteq \Lambda$. Moreover,  $\ini_0(t\cdot G)\subseteq \res(t)\cdot H$, and  by~\cite[Lemma~3.2.6]{MS}\footnote{This is for rank $1$ algebraically closed valued fields, but note that every algebraic group in this lemma is defined over the prime field.}, $\dim(\ini_0(G))=\dim(G)$. It follows that $\dim(H)=\dim(G)$, hence $\Lambda'$ and $\Lambda$ have the same rank, and because $\Lambda$ is primitive $\Lambda=\Lambda'$.
\end{proof}

\begin{defin}\*
  \begin{enumerate}
  \item We define $\sigma_\to$ as the map sending a Laurent polynomial  $f= \sum c_{ijh}y_{jh}^i\in K[y_{jh}^{\pm1}\mid j\in[1,n], h \in \mathbb{Z}]$ to the Laurent polynomial
    \[
      \sigma_\to(f)\coloneqq \sum \sigma(c_{ijh}) y_{j(h+1)}^i.
    \]
  \item   We write $\sigma_\to^p$ for $(\sigma_\to)^p$.
  \end{enumerate}
\end{defin}
Note that $\sigma_\to$ is invertible and that $f(a)=0$ if and only if $(\sigma_\to(f))(\sigma(a))=0$.

\begin{lemma}\label{prop:ideals-of-subgroups}
	Let $(K,\sigma)$ be an algebraically closed difference field, and let $I \subseteq K[y_{jh}^{\pm1}\mid j\in[1,n], h \in[0, \ell]]$ be a proper ideal such that the variety $V(I)$ is a coset of an algebraic subgroup of $\mathbb{G}_{\mathrm{m}}^{n(\ell+1)}$. Assume moreover that there is $\alpha\in (K^\times)^n$ such that, for all $r\le s\in \mathbb Z$, we have that $\orb_{r,s}(\alpha) \in V(\sum_{p \in [r,s]} \sigma_\to^p(I))$.
	
	For each $r \leq s \in \mathbb{Z}$ there is an ideal $J_{r,s}$ of $K[y_{jh}^{\pm1}\mid j\in[1,n], h \in[r,s+ \ell]]$ such that $J_{r,s} \supseteq \sum_{p \in [r,s]} \sigma_\to^p(I)$ and the set $\set{J_{r,s}}_{[r,s] \in \mathbb{Z}}$ satisfies the following properties.
	\begin{enumerate}
		\item\label{point:irreducible} For all $r\leq s \in \mathbb{Z}$, the variety $H_{r,s}\coloneqq V(J_{r,s})$ is a coset of a connected algebraic subgroup of $\mathbb{G}_{\mathrm{m}}^{n(s-r+\ell+1)}$.
		\item\label{point:sigmainvariant} For all $r \leq s \in \mathbb{Z}$, we have $H_{r+1,s+1}=\sigma(H_{r,s})$.
		\item\label{point:overO} If $v$ is a valuation on $K$ commuting with $\sigma$, and $V(I)$ intersects $(\mathcal{O}^\times)^{n(\ell+1)}$, then each $H_{r,s}$ intersects $(\mathcal{O}^\times)^{n(s-r+\ell+1)}$.
	\end{enumerate}
\end{lemma}

\begin{proof}
 For each interval $[r,s] \subseteq \mathbb{Z}$, define ideals $I_{r,s}$ and $I'_{r,s}$ of $K[y_{jh}^{\pm 1} \mid j\in[1,n], h\in[r,s+\ell]]$  by \[I_{r,s}\coloneqq \sum_{p \in [r,s]} \sigma_\to^p(I)\] and \[I'_{r,s}\coloneqq  \sum_{[r_0,s_0] \supseteq [r,s]} (I_{r_0,s_0} \cap K[y_{jh}^{\pm1}\mid j\in[1,n], h\in[r,s+\ell]]).\] 
By noetherianity of $K[y_{jh}^{\pm1}\mid j\in[1,n], h\in[r,s+\ell]]$, this sum stabilises, so given $[r,s]$, there is $[r_0,s_0]$ such that for all $[r_1,s_1] \supseteq [r_0,s_0]$ we have
 \begin{multline*}
   I_{r_1,s_1} \cap K[y_{jh}^{\pm1}\mid j\in[1,n], h\in[r,s+\ell]]\\=I_{r_0,s_0} \cap K[y_{jh}^{\pm1}\mid j\in[1,n], h\in[r,s+\ell]]=I'_{r,s},
 \end{multline*}
hence $V (I_{r,s}')$ is still a coset of an algebraic subgroup of $\mathbb{G}_{\mathrm{m}}^{n(s-r+\ell+1)}$.	Let $J_{r,s}$ be the prime ideal containing $I'_{r,s}$ of polynomials vanishing on the irreducible component $H_{r,s}$ of $V(I'_{r,s})$ which contains the point $\orb_{r,s}(\alpha)$. Then, for all $r \leq s$, the irreducible variety $H_{r,s}$ is again a coset of an algebraic subgroup of $\mathbb{G}_{\mathrm{m}}^{n(s-r+\ell+1)}$, proving~\ref{point:irreducible}, and by construction $H_{r+1,s+1}=\sigma(H_{r,s})$, proving~\ref{point:sigmainvariant}. 
	
	Finally, for~\ref{point:overO} note that, under these assumptions, $I$ is generated by polynomials of the form $x^\mathbf{u}-b$, with $b\in \mathcal O^\times$. It follows that so is $I_{r,s}'$, so $V(I_{r,s}')$ has a point with all coordinates in $\mathcal O^\times$ by \Cref{rem:OxpureinKx}. Different irreducible components of a subgroup of $\mathbb G_\mathrm{m}^N$ are translates of each other by a root of unity. As roots of unity have valuation $0$, if one of them intersects $(\mathcal O^\times)^N$, then so do all of them, and the conclusion follows.
\end{proof}

\begin{thm}\label{thm_purity}
	Suppose $K$ is an existentially closed valued difference field. Then $\mathcal O^\times$ is a pure $\mathbb{Z}[\sigma]$-submodule of $K^\times$.
\end{thm}

\begin{proof}
	Let $A \in \operatorname{Mat}_{d \times n}(\mathbb{Z}[\sigma])$, let $b \in (\mathcal{O}^\times)^d$, let $\beta \in (K^\times)^n$ be such that $A \cdot \beta=b$, and denote by $\ell$ the highest order of an iterate of $\sigma$ appearing in the matrix $A$.
	
	Let $t \cdot G$ be the coset of an algebraic subgroup $G \leq \mathbb{G}_{\mathrm{m}}^{n(\ell+1)}$ obtained from  \Cref{lemma:matrix-to-coset}, and let $I$ be the ideal of polynomials in the ring $K[y_{jh}^{\pm 1} \mid j\in[1,n], h\in[0,\ell]]$ which vanish on $t \cdot G$. Let $\set{H_{r,s}}_{[r,s] \subseteq \mathbb{Z}}$ be the set of varieties obtained by \Cref{prop:ideals-of-subgroups}. For each $[r,s]$, as $H_{r,s}$ is irreducible, by \Cref{irr} so is $\ini_0(H_{r,s})$.
	
	Let $\Theta$ be the set of formulas in the variables $(y_{jh})_{j\in[1,n], h \in \mathbb{Z}}$, which contains, for each finite subtuple $(y_{jh})_{j\in[1,n], h\in[r,s+\ell]}$, the following.
	\begin{enumerate}
		\item The formulas saying that $(y_{jh})_{j\in[1,n], h\in[r,s+\ell]}$ is  a $K$-generic point\footnote{Recall that a point of an irreducible algebraic variety is $K$-generic iff it does not lie in any proper subvariety definable over $K$.} of the irreducible  variety $H_{r,s}$, all of whose coordinates have valuation $0$.
		\item The formulas saying that $(\res(y_{jh}))_{j\in[1,n], h\in[r,s+\ell]}$ is a $k$-generic point of the irreducible variety $\ini_0(H_{r,s})$.
	\end{enumerate}
	Let $\Delta$ be a finite fragment of $\Theta$. Without loss of generality, $\Delta$ says that a finite subtuple of variables $(y_{jh})_{j\in[1,n],h\in[r,s+\ell]}$ lies in a certain Zariski open dense subset $H'$ of $H_{r,s}$, that all of its coordinates have valuation $0$, and that its residue lies in a certain Zariski open dense subset $H''$ of $\ini_0(H_{r,s})$. Since the residue field $k$ is algebraically closed, there is a point $(a_{jh})_{j\in[1,n],h\in[r,s+\ell]} \in H''(k)$, and by \Cref{pr:zardense} the set of points of valuation $0$ in $H_{r,s}(K)$ with residue $(a_{jh})_{j\in[1,n],h\in[r,s+\ell]}$ is Zariski dense in $H_{r,s}$, so in particular there is such a point in $H'(K)$. Therefore $\Theta$ is finitely satisfiable, so by the Compactness Theorem it is consistent.

	Hence $\Theta$ is realised in some valued field extension $(K',v')$ of $(K,v)$ by some infinite tuple $(\alpha_{jh})_{j\in[1,n], h \in \mathbb{Z}}$. Without loss of generality we assume $K'=K((\alpha_{jh})_{j\in[1,n], h \in \mathbb{Z}})$.  By using  \Cref{prop:sigma-commutes-ini0}, we see that $\sigma(\ini_0(H_{r,s}))=\ini_0(\sigma(H_{r,s}))=\ini_0(H_{r+1,s+1})$. 
 Because $H_{r,s}$ is a coset of an algebraic subgroup isomorphic (as an algebraic group) over the prime field to a power of $\mathbb G_\mathrm{m}$, and because in every power of $\mathbb G_\mathrm{m}$ there is a unique type of a tuple of valuation $0$ with $k$-generic residue, it follows that $\Theta$ is a complete quantifier-free $\mathcal L_{\mathrm{vf}}$-type over $K$.
        
        Therefore, extending the action of $\sigma_\to$ from polynomials to formulas in the obvious way, as $(\alpha_{jh})_{j\in[1,n], h \in \mathbb{Z}}$ realises $\Theta$, for all $r\leq s \in \mathbb{Z}$ we have, in $\mathcal L_{\mathrm{vf}}$, \[\sigma_\to \left( \qftp((\alpha_{jh})_{j\in[1,n], h\in[r,s+\ell]}/K)\right)=\qftp((\alpha_{jh})_{j\in[1,n], h\in[r+1,s+\ell+1]}/K).\]

	Therefore we may extend $\sigma$ to an automorphism $\sigma'$ of $(K',v')$ by putting $\sigma'(\alpha_{jh})\coloneqq \alpha_{j(h+1)}$ for each $j\in[1,n]$ and $h \in \mathbb{Z}$. Then the tuple $(\alpha_{10},\ldots,\alpha_{n0}) \in (K'^\times)^n$ is a solution of the equation $A \cdot z=b$, whose coordinates have valuation $0$. Hence,  by existential closedness there is a solution with coordinates of valuation $0$ in $K$.
\end{proof}

\begin{co}\label{co:section} Let $K$ be an existentially closed valued difference field. If $\mathcal O^\times$ is positively $\aleph_1$-saturated as a $\mathbb Z[\sigma]$-module (e.g.\ if $K$ is existentially $\aleph_1$-saturated), then $K$ admits a cross-section $s\from \Gamma \to K$ satisfying $s\circ\sigma=\sigma\circ s$, hence an angular component $\ac\from K\to k$ satisfying $\mathrm{ac}\circ\sigma=\sigma\circ\mathrm{ac}$. In particular, every valued difference field embeds into one admitting an equivariant cross-section.
\end{co}
\begin{proof}
 Recall that, in an $R$-module, a \emph{partial positive primitive type} (or \emph{partial pp-type}) is specified by a filter of pp-definable cosets, hence can be written by using at most $\abs R+\aleph_0$-many formulas (see e.g.\ \cite{prestModelTheoryModules1988} for model theory of modules). In particular, under our assumptions the $\mathbb Z[\sigma]$-module $\mathcal O^\times$ realises all consistent partial pp-types. By \Cref{thm_purity} the short exact sequence of $\mathbb Z[\sigma]$-modules $1\to \mathcal O^\times\to K^\times\to \Gamma\to 0$ is pure, hence the existence of a cross-section follows from~\cite[Theorem 2.8]{prestModelTheoryModules1988}. As standard, a cross-section $s$ induces an angular component $\ac_s(x)\coloneqq \res(x/s(v(x)))$, and clearly if $s$ commutes with $\sigma$ then so does $\ac_s$, since both $\res$ and $v$ do.  The ``in particular'' part follows from \Cref{rem:ecsatexist}.
\end{proof}

\begin{rem}\label{trop:eqvini}
	 If $(K,v,\sigma)$ is a valued difference field and $s\from \Gamma \rightarrow K$ is a $\sigma$-equivariant cross-section, then it is easy to see that, for every Laurent polynomial $f \in K[x^{\pm 1}]$ and $\gamma \in \Gamma^n$, we have $\sigma(\ini_\gamma(f))=\ini_{\sigma(\gamma)}(\sigma(f))$; similarly, for initial ideals we get $\sigma(\ini_\gamma(I))=\ini_{\sigma(\gamma)}(\sigma(I))$.
\end{rem}

\subsection{Good reducts}\label{subsec:ec_and_reducts}
By \Cref{co:section}, every  existentially closed valued difference field that is sufficiently existentially saturated can be expanded by an equivariant cross-section. In order to use the cross-section as an auxiliary object, and later remove it from the language, it is important to understand whether, given an existentially closed $(K,v,\sigma,s)$, its reduct $(K,v,\sigma)$ is still existentially closed. Similar questions can be formulated when passing from $(K,v,\sigma,s)$ to $(K,v,\sigma,\ac)$, etc.

Below, we will prove that these questions have positive answers.
As this will be proved separately for each case of interest, and as there will be a nontrivial amount of algebra involved, the reader may wonder whether some abstract argument shows that existential closedness passes to reducts. This is in fact false.
We thank Angus Matthews for the following example, and for his permission to include it here. We will encounter another instance of this phenomenon in \Cref{co:acvfanotacfa}.
\begin{eg}[A.~Matthews]\label{eg:angus}
  Let $T_0$ be an inductive theory that is not model complete (e.g.\ the theory of oags with an automorphism, or the theory of fields), and let $T_1$ be its Morleyisation (see the references given in \Cref{rem:AEbasics}). Observe that every model of $T_0$ expands to a model of $T_1$, and that every model of $T_1$ is existentially closed by quantifier elimination. In particular, if $M_0\models T_0$ is not existentially closed, then its natural expansion to $M_1\models T_1$ is existentially closed but has a non-existentially closed reduct.
\end{eg}

		\begin{pr}\label{claim:extendac}
		Let $(K,v,\sigma,\ac)$ be an $\mathrm{ac}$-valued difference field closed under roots, and let $(L,v,\sigma)$ be an existentially closed, existentially $\aleph_1$-saturated valued difference field extension. There is an equivariant angular component map on $L$ extending the angular component map of $K$.
		\end{pr}
		\begin{proof}
			Let $T$ be the $\mathbb Z[\sigma]$-submodule of $L^\times$ generated by $K^\times$ and $\mathcal{O}(L)^\times$, and let $\theta\from T \rightarrow k(L)^\times$ be the unique morphism which coincides with $\ac$ on $K^\times$ and with $\res$ on $\mathcal{O}(L)^\times$. Enumerate $T$ in a tuple $\bar{b}$.  If we prove that the pp-type of $\bar{b}$ in the $\mathbb{Z}[\sigma]$-module $L^\times$ is contained in the pp-type of $\theta(\bar{b})$ in $k(L)^\times$, then we may conclude by using existential saturation of $k(L)^\times$ as a $\mathbb Z[\sigma]$-module to apply Point (v) in \cite[Theorem 2.8]{prestModelTheoryModules1988}, obtaining a morphism $\ac_L\from L^\times \rightarrow k(L)^\times$ which extends $\theta$.

                        Let $a$ be some finite subtuple of $\bar{b}$ and $A \in \mathrm{Mat}_{d \times n}(\mathbb{Z}[\sigma])$ a matrix such that $A \cdot z = a$ has a solution in $(L^\times)^n$. Fix such a solution $\alpha$. We need to show that $A\cdot z = \theta(a)$ has a solution in $(k(L)^\times)^n$. Since $k(L)$ is existentially closed as a difference field by \Cref{pr:ecbasics}, it suffices to find an extension of $k(L)$ where a solution exists.

                        As in the proof of \Cref{lemma:matrix-to-coset}, we associate to $A$ some $d$ matrices $C_1,\ldots,C_d \in \mathrm{Mat}_{n \times (\ell+1)}(\mathbb{Z})$, where $\ell$ is maximum such that  $\sigma^\ell$ appears in $A$, such that the entry $c_{ijh}$ of the matrix $C_i$ is the integer coefficient of $\sigma^{h}$ in the entry $a_{ij} \in \mathbb{Z}[\sigma]$ of $A$.

For all $r \leq s \in \mathbb{Z}$, there is a group homomorphism $\psi_{r,s}\from \mathbb{G}_{\mathrm{m}}^{n(s+\ell-r+1)} \rightarrow \mathbb{G}_{\mathrm{m}}^{d(s-r+1)}$ satisfying $\psi_{r,s}(\orb_{r,s+\ell}(\alpha))=\orb_{r,s}(a)$, namely we can take the homomorphism which maps $(y_{jh})_{j\in[1,n],h\in[r,s+\ell]}$ to
\[\begin{pmatrix}
	\prod_{j=1}^n \prod_{h=r}^{r+\ell} y_{jh}^{c_{1jh}} \\
	\vdots \\
	\prod_{j=1}^n \prod_{h=r}^{r+\ell} y_{jh}^{c_{djh}}\\
	\prod_{j=1}^n \prod_{h=r+1}^{r+1+\ell} y_{jh}^{c_{1jh}}\\
	\vdots\\
	\prod_{j=1}^n \prod_{h=r+1}^{r+1+\ell} y_{jh}^{c_{djh}}\\
	\vdots \\
	\prod_{j=1}^n \prod_{h=s}^{s+\ell} y_{jh}^{c_{1jh}} \\
	\vdots\\
	\prod_{j=1}^n \prod_{h=s}^{s+\ell} y_{jh}^{c_{djh}}
\end{pmatrix}.\]

By \Cref{lem:monomial-map} there is a monomial map $\xi_{r,s}$ such that $\ker\xi_{r,s}$ is the connected component of $\ker\psi_{r,s}$. Moreover, each of the exponents of $\xi_{r,s}$ is a rational multiple of an element of the group spanned by the exponents of $\psi_{r,s}$, and therefore every entry of $\xi_{r,s}(\orb_{r,s}(\alpha))$ is a root of a product of the entries of $\orb_{r,s}(a)$.  By assumption, $T$ is a multiplicative subgroup that is closed under taking roots hence, since each entry of $\orb_{r,s}(a)$ lies in $T$,  we have $\xi_{r,s}(\orb_{r,s}(\alpha)) \in T^{d(s-r+1)}$.

Let $e=(e_1,\ldots,e_d)$ denote the tuple $\xi_{0,0}(\orb_{0,\ell}(\alpha))$.  Consider the set $\Omega$ of $\mathcal{L}_{\mathrm{field}}$-formulas over $k(L)$ in variables $(y_{jh})_{j\in[1,n],h \in \mathbb{Z}}$ which says that each finite subtuple $(y_{jh})_{j\in[1,n],h\in[r,s+\ell]}$ is a $k(L)$-generic point of the algebraic variety $V_{r,s}$ defined by the equations $\xi_{r,s}((z_{jh})_{j\in[1,n],h\in[r,s+\ell]})=\theta(\orb_{r,s}(e))$. Note that the genericity requirement makes sense because, as $V_{r,s}(k(L))$ is a coset of $\ini_0(\ker \xi_{r,s})$ (which does not depend on the choice of an angular component by \Cref{lemma:initial-of-pol-over-O}), it is irreducible by \Cref{irr}.

To see that $\Omega$ is consistent, consider a finite fragment mentioning formulas in the finite subtuple of variables $(y_{jh})_{j\in[1,n],h\in[r,s+\ell]}$. As $k(L)^\times$ is a divisible abelian group, there is a (possibly non-equivariant) angular component $\ac'\from  L\to k(L)$ extending $\theta$. Since, for all $r\le s$, we have $\orb_{r,s}(e)\in \im(\xi_{r,s})$, and $\ac'$ is a multiplicative homomorphism, $V_{r,s}(k(L))$ is not empty.

Hence the set $\Omega$ is realised in some field extension of $k(L)$; let $(\beta_{jh})_{j\in[1,n],h \in \mathbb{Z}}$ be a realisation, and consider the field $k(L)((\beta_{jh})_{j\in[1,n],h \in \mathbb{Z}})$. We now have that, for all $r \leq s$, the finite tuple $(\beta_{j(h+1)})_{j\in[1,n],h\in[r,s+\ell]}$ is generic in the algebraic variety defined by the equations \[\xi_{r+1,s+1}((z_{j(h+1)})_{j\in[1,n],h\in[r,s+\ell]})=\theta(\orb_{r+1,s+1}(e)),\] which since $\sigma$ and $\theta$ commute is the same as \[\xi_{r+1,s+1}((z_{j(h+1)})_{j\in[1,n],h\in[r,s+\ell]})=\sigma(\theta(\orb_{r,s}(e))).\] Hence, applying  $\sigma_{\to}$ to  the  $\mathcal L_\mathrm{field}$-quantifier-free type  over $k(L)$ of the tuple  $(\beta_{jh})_{j\in[1,n],h\in[r,s+\ell]}$ yields the one of $(\beta_{jh})_{j\in[1,n],h\in[r+1,s+1+\ell]}$. So the automorphism $\sigma$ of $k(L)$ extends to an automorphism of $k(L)((\beta_{jh})_{j\in[1,n],h \in \mathbb{Z}})$, in which the equation $A \cdot z=\theta(a)$ is solved by the tuple $\beta\coloneqq (\beta_{10},\ldots,\beta_{1n})$. 
\end{proof}

\begin{co}\label{co:acec}
 If $(K,v,\sigma,\ac)$ is existentially closed, then $(K,v,\sigma)$ is existentially closed.
\end{co}
\begin{proof}
As $\mathsf{ACVF}_{\mathrm{ac}}$ eliminates quantifiers\footnote{This is folklore, and follows from e.g.\ the much more general \cite[Theorem~1]{haleviEliminatingFieldQuantifiers2019}. For a direct proof, use~\cite[Fact~2.4]{hilsImaginariesSeparablyClosed2018} in combination with~\cite[Observation~1.51]{touchardBurdenHenselianValued2023}.}, arguing as in the proof of \Cref{pr:ecbasics}\ref{point:ecbasicsACF} we see that $K$ is algebraically closed, and in particular closed under taking roots.  Suppose that $(L_0,v_{L_0},\sigma_{L_0})\supseteq (K,v,\sigma)$ is such that $L_0\models \exists x\;\phi(x)$, with $\phi(x)$ a quantifier-free $\mathcal L_{\sigma}(K)$-formula. Let $(L,v_L,\sigma_L)$ be an existentially $\aleph_1$-saturated existentially closed extension of $L_0$. After applying \Cref{claim:extendac}, $L$ extends $K$ as $\mathrm{ac}$-valued difference fields, and $L\models \exists x\; \phi(x)$. Since $\mathcal L_{\mathrm{ac},\sigma}\supseteq \mathcal L_{\sigma}$ and $(K,v,\sigma,\ac)$ is existentially closed, we have $K\models \exists x\;\phi(x)$.
\end{proof}
\begin{lemma}\label{lemma:secbasics}
  If $(K,v,\sigma,s)$ is existentially closed, then $(\Gamma(K),\sigma)$ is existentially closed.
\end{lemma}
\begin{proof}
 It is well known that the theory $\mathsf{ACVF}_s$ of algebraically closed valued fields with a cross-section eliminates quantifiers; see e.g.~\cite[Theorem~2.2]{Kesting} for a proof for $\mathsf{ACVF}_{s,\iota}$ that can be easily adapted.

Embed $(K,v,s)$ in a monster model $\monster\models \mathsf{ACVF}_s$.  Let $(\Delta,\sigma_\Delta)$ be an extension of $(\Gamma(K),\sigma_{\Gamma(K)})$; we may assume that $\Delta$ is embedded in $\Gamma(\monster)$ over $\Gamma(K)$. If the map $f\coloneqq\sigma\cup \sigma_\Delta$, viewed as a partial map $\monster\to \monster$, is elementary, then it extends to an automorphism of $\monster$ and we conclude as in \Cref{pr:ecbasics}. 

  By quantifier elimination it suffices to show that $f$ preserves and reflects atomic formulas. Up to applying the same arguments to $f\inverse$, it is enough to show that atomic formulas are preserved. It suffices to check formulas of the form $\phi(a,\gamma)\coloneqq v(t_0(a,\gamma))\ge v(t_1(a,\gamma))$, where $t_0, t_1$ are terms in the language with a cross-section $s$, and the tuples $a$,$\gamma$ are from the sorts $\mathrm{VF}$, $\Gamma$ respectively. We do this by induction on the maximum depth of nested occurrences of $s$ in $t_0, t_1$; in the same induction, we simultaneously prove that the valuation of such terms lies in $\Delta$.

  If $s$ does not appear at all, then $\phi$ is an $\mathcal L_\mathrm{vf}$-formula, hence the conclusion follows from the fact that $\sigma$ is an automorphism. If $s$ never appears nested more than once, then $\phi$ is, up to equivalence, of the form $v(p(a,s(\gamma)))\ge v(q(a,s(\gamma)))$, where $\gamma$ is a tuple from $\Delta\setminus \Gamma(K)$ that is $\mathbb Q$-linearly independent over $\Gamma(K)$ and $p,q$ are polynomials over $\mathbb Z$, say. Then $p(a,x)$ is a polynomial over $K$; write it as the sum of its monomials, say $p(a,x)=\sum_i m_i(x)$. By linear independence of $\gamma$ over $\Gamma(K)$, the $v(m_i(s(\gamma)))$ are pairwise distinct, hence there is $i_0$ such that  $v(p(a,s(\gamma)))=v(m_{i_0}(s(\gamma)))$. Similarly we may write $v(q(a,s(\gamma)))=v(n_{j_0}(s(\gamma)))$. As $\sigma_\Delta$ is an automorphism, it is easy to see that $v(p(f(a),f(s(\gamma))))=v(f(m_{i_0})(f(s(\gamma))))$, and similarly for $q$, hence $f$ preserves $\phi(a,\gamma)$; because $m_{i_0}$ is a monomial, we also have $v(f(m_{i_0})(f(s(\gamma))))\in \Delta$.

  If $s$ appears nested more than once in a term, then there must be an occurrence of $v$ in the middle, that is, the term contains a subterm of the form $s(t_2(a,\gamma,v(t_3(a,\gamma))))$, where $t_2,t_3$ are terms with fewer nested occurrences of $s$. As by inductive assumption $v(t_3(a,\gamma))\in \Delta$,  we fall back into the previous case.
\end{proof}
\begin{pr}\label{co:secacec}
  If $(K,v,\sigma,s)$ is existentially closed, and we set $\ac(x)\coloneqq\res(x/s(v(x)))$, then $(K,v,\sigma,\ac)$ is existentially closed. Therefore, so is $(K,v,\sigma)$.
\end{pr}
\begin{proof}
Given an extension $(L,v_L,\sigma_L,\ac_L)$ of $(K,v,\sigma,\ac)$, which we may assume existentially closed and existentially $\aleph_1$-saturated, we will show that there is a cross-section $s_L\from \Gamma(L) \rightarrow L$ which extends $s$ and such that $\ac_L(z)=\res(z/s_L(v_L(z)))$ for every $z \in L$. This will conclude the proof, as it shows that then $(K,v,\sigma,\ac)$ is existentially closed as an $\ac$-valued field, and then $(K,v,\sigma)$ is existentially closed by \Cref{co:acec}.

It is enough to show that there is a cross-section $s_L$ such that $\ac_L(s_L(\gamma))=1$ for every $\gamma \in \Gamma(L)$: if this is the case, then for every $z \in L$ we have $\res(z/s_L(v_L(z)))=\ac_L(z/s_L(v_L(z)))=\ac_L(z)/\ac_L(s_L(v_L(z)))=\ac_L(z)$ as we wanted. 
        
Consider the $\mathbb{Z}[\sigma]$-submodule $N$ of $L$ consisting of elements of $L$ with angular component 1; note that the image of $s$ is contained in $N$. Let $\bar{b}$ be a tuple enumerating $\Gamma(K)$. If we prove that the pp-type of $\bar{b}$ in $\Gamma(L)$ is contained in the pp-type of $s(\bar{b})$ in $N$, then we may apply \cite[Theorem 2.8(v)]{prestModelTheoryModules1988} and obtain the desired extension of $s$ to $\Gamma(L)$. In other words, we need to show that if $a$ is a finite subtuple of $\bar b$, and $A \in \operatorname{Mat}_{d \times n}(\mathbb{Z}[\sigma])$ is such that $A \cdot z=a$ has a solution in $\Gamma(L)$, then $A \cdot z=s(a)$ has a solution in $N$. In fact, it suffices to show that $A \cdot z=a$ has a solution in $\Gamma(K)$, as then if $\gamma$ is such a solution we get that $A \cdot s(\gamma)=s(a)$ and $s(\gamma)\in N$.

This follows from \Cref{lemma:secbasics}, concluding the proof. 
\end{proof}
\begin{co}\label{co:acec_basics}
	Let $K$ be an existentially closed $\ac$-valued (resp.\ $s$-valued) difference field. Then $\VF(K)$ is algebraically closed and $(\Gamma,\sigma)$ and $(k,\sigma)$ are existentially closed in their respective categories. In particular, $(K,v,\ac)$ (resp.\ $(K,v,s)$) and $(K,v)$ are existentially closed.
\end{co}

\begin{proof}
  By \Cref{co:secacec} it suffices to deal with the $\ac$-valued case.
  By \Cref{co:acec} $(K,v,\sigma)$ is existentially closed, so we may apply \Cref{pr:ecbasics}. Alternatively, we may adapt its proof by using quantifier elimination in $\mathsf{ACVF}_\mathrm{ac}$.   The ``in particular'' part follows from the above and  \Cref{pr:ecbasics}\ref{point:ecbasicsACF}.
\end{proof}

\subsection{Bad reducts}
Another natural reduct of $(K,v,\sigma)$ to consider is the difference field $(K,\sigma)$. In this subsection we will see that, in contrast to what we saw in \Cref{subsec:ec_and_reducts}, forgetting the valuation does not preserve existential closedness.

Recall that existentially closed difference fields are axiomatisable by a theory denoted by $\mathsf{ACFA}$, in the language $\set{+,0,-, \cdot, 1, (-)\inverse, \sigma, \sigma\inverse}$. The axioms of this theory say that $(K,\sigma)\models \mathsf{ACFA}$ iff $K$ is algebraically closed and, for all $n$, all  irreducible algebraic subvarieties $V,W$ of $K^n$ and $K^{2n}$ respectively such that $W \subseteq V \times \sigma(V)$ and the projections of $W$ to $V$ and to $\sigma(V)$ are dominant, there is $x\in V(K)$ such that $(x,\sigma(x))\in W$. See \cite[Theorem 7]{Ma}, \cite[Theorem on p.~3007]{CHr}.

We will see in \Cref{eg:ecHahn} that if $k\models \mathsf{ACFA}$ has characteristic zero and $\Gamma$ is an existentially closed difference oag, then the Hahn field $k((t^\Gamma))$, made into a valued difference field as in \Cref{eg:Hahnsigma}, is existentially closed. The following argument, due to Philip Dittmann, shows that it is not existentially closed as a difference field, that is, it is not a model of $\mathsf{ACFA}$.

Below, when we speak of fixed points, fixed fields, etc., we always refer to fixed points of $\sigma$, fixed fields of $\sigma$, etc.

\begin{pr}[P.~Dittmann]\label{pr:philip}
  Let $(K,v,\sigma)$ be a henselian valued difference field with a fixed point of nonzero valuation. Then $(K,\sigma)\centernot \models \mathsf{ACFA}$.
\end{pr}
\begin{proof}
It is well known (see e.g.~\cite[Proposition~1.2]{CHr}) that the fixed field of every model of $\mathsf{ACFA}$ is pseudofinite, and in particular not separably closed. By~\cite[Corollary~12.5.5]{friedFieldArithmetic2023}\footnote{In the third edition of the same book (2008), this is Corollary~11.5.5.} every pseudo-algebraically closed (in particular, every pseudofinite) nontrivially valued henselian field is separably closed. Since there is a fixed point of nonzero valuation, the induced valuation on the fixed field $K_0$ of $K$ is nontrivial. Hence, to obtain a contradiction,  
it suffices to show that  $K_0$ is henselian.

We check henselianity directly. Let $f\in \mathcal O(K_0)[x]$ and $a\in \mathcal O(K_0)$ be such that $\res(a)$ is a simple root of $\res(f)$. By henselianity of $K$ there is $b\in \mathcal O(K)$ with $f(b)=0$ and $v(b-a)>0$. If $\sigma(b)=b$ we are done, so assume this is not the case. As $f$ is over the fixed field, $\sigma(b)$ is a root of $f$. Therefore $(x-b)(x-\sigma(b))$ divides $f$, and since $\res(\sigma(b))=\sigma(\res (b))=\sigma(\res(a))=\res(a)$ this contradicts that $\res(a)$ is a simple root of $\res(f)$.
\end{proof}
\begin{co}[P.~Dittmann]\label{co:philip}
  If $k\models \mathsf{ACFA}$ and $\Gamma$ is a difference oag with nonzero fixed points, then the Hahn field $k((t^\Gamma))$ is not a model of $\mathsf{ACFA}$.
\end{co}
\begin{proof}
 By inspection, if $k_0$ is the fixed field of $k$, and $\Gamma_0$ is the fixed subgroup of $\Gamma$, then the fixed field of $k((t^\Gamma))$  is $k_0((t^{\Gamma_0}))$, hence we are in the assumptions of \Cref{pr:philip}.
\end{proof}

By mining the proof of \Cref{pr:philip}, we can obtain an explicit example of failure of the $\mathsf{ACFA}$ axioms. 

\begin{lemma}\label{lemma:fixedfieldclosed}
The fixed field is closed in the valuation topology. Moreover,  if $v(a)=0$ and $\res(\sigma(a))\ne \res(a)$ then the ball $\set{x\mid v(x-a)>0}$  does not intersect the fixed field.
\end{lemma}
\begin{proof}
  Closedness is  clear. For the moreover part, suppose  $v(\sigma(x)-\sigma(a))>0$. If $\sigma(x)=x$, then $v(a-\sigma(a))>0$, against $\res(\sigma(a))\ne \res(a)$.
\end{proof}

\begin{pr}\label{pr:noacfaaepsilon}
  Let $(K,v,\sigma)$ be a valued difference field containing a fixed point $\epsilon=\sigma(\epsilon)$ of positive valuation and an element $a\in \mathcal O^\times$ of valuation zero that, for some $n>1$, is a fixed point of $\sigma^n$ but such that, for every $m$ with $1\le m<n$, its residue $\res(a)$ is not a fixed point of $\sigma_k^m$. Then $(K,\sigma)\centernot\models \mathsf{ACFA}$.
\end{pr}
\begin{proof} For notational simplicity, we will only deal with the case $n=2$.
Define $P(x)\coloneqq(x-a)(x-\sigma(a))$, observe that it is over the fixed field $K_0$, and set $Q(x,y)\coloneqq P(x)P(y)-\epsilon$.
By the Eisenstein criterion (in the version of~\cite[Lemma~2.3.10(b)]{friedFieldArithmetic2023}) applied to $(K[x])[y]$, the vanishing variety $V$ of $Q$ is absolutely irreducible.

Let us show it does not have points in $K_0$.
Assume it does, say $(b_1, b_2)\in V(K_0)$. Then $P(b_1)P(b_2)=\epsilon$, hence $v(P(b_1))+v(P(b_2))>0$, so for some $i$ we have $v(b_i-a)>0$ or $v(b_i-\sigma(a))>0$. By assumption on $\res(a)$ and \Cref{lemma:fixedfieldclosed} we must have $\sigma(b_i)\ne b_i$.

  Let $W$ be the diagonal of $V\times V$. Then $W$, which is absolutely irreducible because it is isomorphic to $V$, projects surjectively on $V$ and on $\sigma(V)=V$, but a $\sigma$-point on $W$ would contradict that $V(K_0)=\emptyset$.
\end{proof}

\begin{co}\label{co:acvfanotacfa}
  Every existentially closed valued difference field is not existentially closed as a difference field.
\end{co}
\begin{proof}
Let us show that every existentially closed valued difference field satisfies the assumptions of \Cref{pr:noacfaaepsilon}. The existence of $\epsilon$ follows from existential closedness of $\Gamma$ (\Cref{pr:ecbasics}\ref{point:ecbasicsOAGA}). To find some $a$ with $v(a)=0$,  $\sigma^2(a)=a$, and $\res(\sigma(a))\ne \res(a)$, it suffices to take two Gauss extensions (see~\cite[Corollary 2.2.2]{EP}), let $\sigma$ swap the added points, and apply existential closedness.
\end{proof}

\begin{rem}
The proof above may be adapted, for instance, to the existentially closed isometric case studied in \cite{AvD}, but not to the existentially closed strongly contractive case, where the valued field sort is indeed a model of $\mathsf{ACFA}$ \cite[Fact~2.13]{CH}.  
\end{rem}

\Cref{co:acvfanotacfa} implies that there must exist a valued difference field $(K,v,\sigma)$ and a difference field extension $(L,\sigma_L)\supseteq (K,\sigma)$ such that no valuation on $L$ extending $v$ is compatible with $\sigma_L$. It is indeed possible to extract a concrete example from the above proofs.

\begin{eg}
  Let $K=\mathbb C((t))$ with the usual valuation and
  \[
    \sigma\left(\sum_{m}a_mt^m\right)=\sum_{m}\overline{a_m}t^m
  \]
  where $\bar z$ denotes the complex conjugate of $z$.   Take, in the  notation of the proof of \Cref{pr:noacfaaepsilon}, $a=i$ and $\epsilon=t$, i.e.\ $P(x)=(x-i)(x+i)=x^2+1$ and $Q(x,y)=(x^2+1)(y^2+1)-t$. 

  By quantifier elimination in $\mathsf{ACF}$ and the fact that $V$ is defined over the fixed field, hence its generic type is invariant over the fixed field, we can find $L\coloneqq K(b_1,b_2)$ with $\sigma(b_j)=b_j$ and $(b_1,b_2)\in V(Q)$. Assume towards a contradiction that there is an extension of the valuation to $L$  preserved by $\sigma$. Because $Q(b_1,b_2)=0$, there is $j$ such that $v(b_j^2+1)>0$. So $\res(b_j)=\pm i$ is not a fixed point, hence neither is $b_j$, contradiction.

  In other words, what is happening is that the valuation forces points on the variety $(x^2+1)(y^2+1)-t$ to have a coordinate infinitesimally close to a square root of $-1$; continuity with respect to the valuation then implies that this coordinate cannot be fixed.

  In characteristic $(p,p)$, working with $\mathbb F_p^\mathrm{alg}((t))$ and replacing complex conjugation by the Frobenius automorphism allows a  similar argument to go through. 
\end{eg}

\section{Amalgamation}
\label{sec:ap}

An amalgamation problem of (expanded) valued fields induces amalgamation problems at the residue field and value group levels. Therefore, an obvious obstruction to the solution of an amalgamation problem is that one of these induced problems cannot be solved.

In the case of valued difference fields, value difference groups can always be amalgamated by \cite[Theorem~5.10]{DM}, but residual obstructions may still be present.
\begin{eg} 
Let $K$, $M$ and $L$ be the trivially valued difference fields with valued field sorts $\VF(K)=\mathbb Q$ and $\VF(M)=\VF(L)=\mathbb Q(i)$, and automorphisms $\sigma_M(i)=i$ and $\sigma_L(i)=-i$. Then $M$ and $L$ do not amalgamate over the common substructure $K$.
\end{eg}

 Even in the absence of obstacles coming from the residue field (e.g.\ algebraically closed difference fields are amalgamation bases, see \Cref{AP_VDF} and the following discussion), amalgamation problems may fail to have solutions.

\begin{eg}\label{eg:apfailure}
	Let $k$ be algebraically closed, let $\sigma\in \aut(k)$ be arbitrary, and consider the Hahn product $k((\mathbb R))$ of $(k,\sigma)$ and $(\mathbb{R},0, +, \leq,\id_{\mathbb{R}})$. Let $K_0\coloneqq k((\mathbb{Z})) \subseteq k((\mathbb{R}))$, and let $K\coloneqq K_0(t^{\pi})$. By the Zariski--Abhyankar inequality \cite[Corollary 3.25]{Dri}, the residue field of $K$ is an algebraic extension of $k$, so it is equal to $k$. Let $L$ and $M$ be valued fields isomorphic to $K(\sqrt{t^\pi})$, endowed with automorphisms $\sigma_L$ and $\sigma_M$ extending the automorphism on $K$ so that $\sigma_L$ fixes the square roots of $t^{\pi}$ while $\sigma_M$ swaps them. Then $(L,\sigma_L)$ and $(M,\sigma_M)$ cannot be amalgamated over $K$.
\end{eg}

This is not surprising, as similar issues are well known to exist when one tries to amalgamate e.g.\ henselian valued fields of residual characteristic zero (with no extra structure). Common solutions are to expand the language by an angular component or a cross-section, which in our case will be assumed to be equivariant. In \Cref{AP_VDF} we will show that, in the residual characteristic zero case, this allows to characterise which amalgamation problems have solutions. Our result also works when a lift of the residue field is added to the language.

\subsection{Set-up}

\begin{defin}
	A difference field $(k,\sigma)$ is \emph{linearly difference closed} iff for all $n \in \mathbb{N}$ and $\alpha_0,\ldots, \alpha_n\in k$ with $\alpha_n\neq 0$ the equation $1+\alpha_0 x+\ldots + \alpha_n \sigma^n(x)=0$ has a solution in $k$.
\end{defin} 

It is an easy and well-known fact that every model of $\mathsf{ACFA}$ is linearly difference closed.

The following definitions appear for example in \cite[Section 2]{AvD}.

\begin{defin}
	Let $(K,\sigma)$ be a difference field. A \emph{$\sigma$-polynomial over $K$} is an expression of the form $G(x)=g(x,\sigma(x),\ldots,\sigma^n(x))$, where $g \in K[T_0,\ldots,T_n]$.
	
	The \emph{order} of the $\sigma$-polynomial $G(x)$ is the smallest $n$ for which there exists $g \in K[T_0,\ldots,T_n]$ such that $G(x)=g(x,\ldots,\sigma^n(x))$. The \emph{complexity} of a $\sigma$-polynomial $G(x)= g(x,\sigma(x),\ldots, \sigma^n(x))$ of order $n$ is the ordered triple $(n, \deg_{x_n} g, \deg g)$, where $\deg g$  denotes the total degree of $g$.
\end{defin}

We regard complexities of $\sigma$-polynomials as being linearly ordered by the lexicographical ordering on $\mathbb{N}^3$.

\begin{defin}
	Let $(K,\sigma)$ be a difference field, $(K_0,\sigma_{|K_0})$ a difference subfield. An element $a \in K$ is \emph{$\sigma$-algebraic} over $K_0$ iff there is a nonzero $\sigma$-polynomial $G$ over $K_0$ such that $G(a)=0$, and \emph{$\sigma$-transcendental} over $K_0 $ otherwise. If $a$ is $\sigma$-algebraic over $K_0$, a \emph{minimal $\sigma$-polynomial} for $a$ over $K_0$ is a nonzero $\sigma$-polynomial $G$ of minimal complexity such that $G(a)=0$.
\end{defin}

For $G$ a $\sigma$-polynomial, we refer to~\cite[Section 2]{AvD} for the definition of the  $\sigma$-polynomials $G_{(\mathbf i)}$. Note that if $\mathbf{i}\ne 0$ then $G_{(\mathbf i)}$ has strictly lower complexity than $G$.

For $G$ a $\sigma$-polynomial and $a \in K$, we refer to~\cite[Definition~4.2]{Az} or~\cite[Definition~4.1]{DO} for the definitions of $(G,a)$ being in \emph{$\sigma$-Hensel configuration} and of the value $\gamma(G,a)$. Recall that a valued difference field is \emph{$\sigma$-henselian} iff, whenever $(G,a)$ is in $\sigma$-Hensel configuration, there is $b$ with $G(b)=0$ and $v(a-b)=\gamma(G,a)$.\footnote{There is another notion of $\sigma$-henselianity, designed to work for valued difference fields with analytic structure, see \cite[Definition~4.10]{Ri}. In the case where the analytic structure is trivial and $\operatorname{char}(k)=0$, the two notions are equivalent, as they both imply (and are implied by) elementary equivalence to a maximally complete valued difference field with linearly difference closed residue difference field. We thank Silvain Rideau-Kikuchi for pointing this out.} Our uses of $\sigma$-henselianity will be mediated by the following observation.

\begin{fact}[{\cite[Remark~4.2]{DO}}]\label{fact:fourtwo}
Let $G$ be nonconstant, $G(a)\ne 0$, $v(G(a))>0$ and $v(G_{(\mathbf{i})})=0$ for all $\mathbf{i}\ne 0$ with $G_{(\mathbf{i})}\ne 0$. Then $(G,a)$ is in $\sigma$-Hensel configuration with $\gamma(G,a)>0$.
\end{fact}

\begin{fact}\label{fact:sigmahenslindifcl}
  If $(K,v,\sigma)$ is $\sigma$-henselian, then $(k(K),\sigma)$ is linearly difference closed.
\end{fact}
\begin{proof}
See~\cite[Lemma~4.6]{Az}  or~\cite[Lemma~4.6]{DO}.
\end{proof}

\begin{fact}\label{fact:liftavdd}
	Assume $K$ is $\sigma$-henselian and $\operatorname{char}(k)=0$. For every difference subfield $K_0\subseteq \mathcal O$ there is a difference subfield $K_1$ of $K$ with $K_0\subseteq K_1\subseteq \mathcal O$ and $\res(K_1)=k(K)$.
\end{fact}
\begin{proof}
	This is \cite[Theorem 4.8]{AvD}. Note that its proof does not assume that the automorphism $\sigma$ is isometric (i.e.\ $\sigma(\gamma)=\gamma$ for all $\gamma\in \Gamma$), even if the main results of the cited paper do. Moreover, even though the notion of $\sigma$-henselianity used there is different, it can be seen by using \Cref{fact:fourtwo} that the proof still works with the notion of $\sigma$-henselianity used here.
\end{proof}

\begin{lemma}\label{henselianpr}  Let $\mathcal L$ be either $\mathcal L_\mathrm{\sigma}$,   $\mathcal L_{\ac,\sigma}$, or the expansion of one of these by a lift. Every existentially closed $\mathcal L$-valued difference field of residual characteristic zero is $\sigma$-henselian.
\end{lemma}
\begin{proof}
  If $\mathcal L$ equals $\mathcal L_\mathrm{\sigma}$ (resp.\ $\mathcal L_{\ac,\sigma}$), this follows from  \Cref{pr:ecbasics}\ref{point:ecbasicsACFA} (resp.\ \Cref{co:acec_basics}), the fact that every $\sigma$-algebraically maximal (i.e.\ having no proper $\sigma$-algebraic immediate extension) valued difference field with linearly difference closed residue field and residual characteristic zero is $\sigma$-henselian \cite[Corollary 5.6(2)]{DO}, and the fact that $\sigma$-henselianity is expressible by a set of $\forall\exists$-sentences in $\mathcal L$. The last fact can be checked by inspection, by observing that, in \cite[Definition~4.1]{DO}, for $G$ of fixed complexity, it suffices to check finitely many of the $G_{(\mathbf{j})}$ because the other ones will be null regardless of the coefficients of $G$.

Assume now $\mathcal L_0$ is one of the languages above and $\mathcal L=\mathcal L_0\cup \set\iota$. Let $K$ be an existentially closed $\mathcal L$-valued difference field and let $K_0$ be its reduct to $\mathcal L_0$. Extend $K_0$ to some $\sigma$-algebraically maximal $M_0$, which is $\sigma$-henselian as recalled above. Now apply \Cref{fact:liftavdd} and the syntactical considerations above.
\end{proof}

\begin{co}\label{co:reductswithlifts}
	Assume $\operatorname{char}(k)=0$. 
        \begin{enumerate}
        \item If $(K,v,\sigma,\ac,\iota)$ is existentially closed, then so are all the reducts obtained by forgetting some subset of $\set{\sigma,\ac,\iota}$.
\item If $(K,v,\sigma,s,\iota)$ is existentially closed, and $\ac(x)=\res(x/s(v(x)))$, then so is $(K,v,\sigma,\ac,\iota)$.
        \end{enumerate}
\end{co}
\begin{proof}
  This is proved similarly to the results in \Cref{subsec:ec_and_reducts}.

  For  $\mathcal L_{\ac,\sigma}$ we use that, by \Cref{fact:liftavdd}, partial lifts of the residue field may always be extended. This also implies the result for $\mathcal L_{\sigma}$ by \Cref{co:acec}, and for $\mathcal L_{\ac}$ and $\mathcal L$ by \Cref{co:acec_basics}. For $\mathcal L_{\ac, \iota}$ we apply \cite[Proposition 1.8]{Kesting}, which allows us to repeat the argument with quantifier elimination as in \Cref{pr:ecbasics}. 

        For $\mathcal L_{\iota,\sigma}$,  let $(L,v,\sigma,\iota)$ be an extension of $(K,v,\sigma,\iota)$. Then, by taking an existentially closed, existentially $\aleph_1$-saturated extension if necessary, we may assume that $L$ has a $\sigma$-equivariant angular component map extending the one on $K$ by \Cref{claim:extendac}. Then $(L,v,\sigma,\ac,\iota)$ is an extension of $(K,v,\sigma,\ac,\iota)$, as there are no extra compatibility conditions to be checked between $\ac$ and $\iota$. This proves that $(K,v,\sigma,\iota)$ is existentially closed. Finally, for $(K,v,\iota)$, we know by the above that $(K,v,\ac,\iota)$ is existentially closed so we can again extend the angular component map to any sufficiently saturated $\mathcal L_\iota$-valued field extension of $(K,v,\iota)$. This completes the proof of the first part.

The proof of the second part uses the same ideas as the proof of \Cref{henselianpr}, together with \Cref{co:secacec}, and is left to the reader.
\end{proof}

\begin{co}\label{henselian}
Let $\mathcal L$ be $\mathcal L_\sigma, \mathcal L_{\ac,\sigma}, \mathcal L_{s,\sigma}$, or an expansion of the previous by a lift. 
  If $K$ is an existentially closed  $\mathcal L$-valued difference field of residual characteristic zero, then it is $\sigma$-henselian.
\end{co}
\begin{proof}
  By \Cref{co:reductswithlifts} and \Cref{henselianpr}.
\end{proof}

\begin{notation}
  If $K$ is an $\mathrm{ac}$-valued difference field, $A\subseteq K$, and $b\in K$,  then by $\langle A \rangle$ and $A\langle b \rangle$ we mean the $\mathrm{ac}$-valued difference fields generated by $A$ and by $A\cup \{b\}$, respectively, unless otherwise specified. We use the same notation for structures in other languages, e.g.\ $k(E)\strgen\alpha$ is the difference field generated by $k(E)$ and $\alpha$.
\end{notation}

The proof of the following fact is straightforward.

\begin{fact}\label{fact:generating+nontorsion}
    Let $(K,v_K)\subseteq (L,v)$ be an extension of valued fields.
    \begin{enumerate}
\item \label{fact_generating} Let $a \in L$ be such that $0,v(a),\ldots, v(a^n)\in \Gamma(L)$ are in different cosets modulo $\Gamma(K)$. Then for all $c_0,\ldots,c_n\in K$ we have 
\[v\left(\sum_ic_i a^i\right) = \min_i\set{v_K(c_i)+iv(a)}.\]
In particular $v(K+Ka+\ldots+Ka^n)\subseteq \Gamma(K) + \mathbb Z v(a)$.
\item \label{fact_nontorsion} If there is $a \in L$ such that $\Gamma(L) = \Gamma(K) + \mathbb Z v(a)$, then every element of $L$ is of the form  $bca^n$, for suitable $b\in K$, $c\in \mathcal O(L)^{\times}$, and $n\in \mathbb Z$.
    \end{enumerate}
\end{fact} 

The following facts about valued difference fields appear in Sections 2 and 6 of \cite{AvD}. As in \Cref{fact:liftavdd}, the isometricity assumption is not used in these three lemmas, hence they hold for arbitrary valued difference fields.

Suppose $(K,v,\sigma)$ and $(K',v,\sigma)$ are valued difference fields. Put $\mathcal O\coloneqq  \mathcal O(K)$ and $\mathcal O'\coloneqq  \mathcal O(K')$.  Let $(E,v,\sigma)$ be a valued difference subfield of both $K$ and $K'$.

\begin{fact}[{\cite[Lemma 2.5]{AvD}}]\label{2.5} Let $a \in \mathcal O$ and assume
$\alpha=\res({a})$ is $\sigma$-transcendental over $k(E)$. 
Then \begin{enumerate}
\item $v(P(a))= \min\limits_{i} \{v(b_{i})\}$  for each
$\sigma$-polynomial $P(x)=\sum b_{i} {\sigma}^{i}(x)$ over
$E$;
\item\label{point:astres}  $v(E \langle a \rangle^{\times})=v(E^{\times})=\Gamma(E)$, and
 $E \langle a \rangle$ has residue field $k(E) \langle
\alpha \rangle$;
\item\label{point:astvf} if $b\in \mathcal O'$ is such that $\beta=\res(b)$ is $\sigma$-transcendental over $k(E)$, then there is
a valued difference field isomorphism $E\langle a \rangle \to  E \langle b \rangle$ over $E$ sending $a$ to $b$. 
\end{enumerate} 
\end{fact}

\begin{fact}[{\cite[Lemma 2.6]{AvD}}]\label{2.6}
Assume $\operatorname{char}(k(E))=0$, and let $G(x)$ be a nonconstant
$\sigma$-polynomial over the valuation ring of $E$ such that 
$\res(G(x))$ has the same complexity as $G(x)$. Let $a\in \mathcal O$, 
$b\in \mathcal O'$, and assume that
$G(a)=0$, $G(b)=0$, and that $\res(G(x))$ is a minimal $\sigma$-polynomial of $\alpha\coloneqq \res{a}$ and of $\beta\coloneqq  \res{b}$ over $k(E)$. Then
\begin{enumerate}
\item $E \langle a \rangle$
has value group $v(E^{\times})=\Gamma(E)$ and residue field 
$k(E) \langle \alpha \rangle$;
\item if there is a difference field isomorphism $k(E)\langle \alpha \rangle \to k(E) \langle \beta \rangle$ over $k(E)$ sending $\alpha$ to $\beta$, then there is
a valued difference field isomorphism $E\langle a \rangle \to  E \langle b \rangle$ over $E$ sending $a$ to $b$. 
\end{enumerate} 
\end{fact}

\begin{fact}\label{fact:avd6364}
Let $E\subseteq F$ and $E'\subseteq F'$ be substructures of $\mathrm{ac}$-valued fields, with $v(F^\times)=v(E^\times)$. Let $f\from  E \to E'$ be an $\mathcal L_{\mathrm{ac}}$-isomorphism, and $g\from  F \to F'$ be an  $\mathcal L_\mathrm{vf}$-isomorphism extending $f$. Suppose that 
$\res(\mathcal O(F))\subseteq k({E})$ and that  $f(\res(u))=\res'(g(u))$
for all $u\in \mathcal O(F)$. 

Then $\ac(F)\subseteq  k({E})$
and $f(\ac(a))= \ac'(g(a))$ for all $a \in F$.  
\end{fact}
\begin{proof}
  The proof of~\cite[Lemma 6.3]{AvD} does not use the elementarity requirements in the definition of a good map in~\cite[Section 6]{AvD}, thus we can relax the assumption of being a good map to being an isomorphism of $\mathrm{ac}$-valued fields.
\end{proof}
Below, we extend the tropicalisation operator to $\sigma$-polynomials in the obvious way. For example, $\trop(\sigma^2(x^3)+\sigma(x^4))=\min \set{3\sigma^2(x), 4\sigma(x)}$.
\begin{defin}
  Let $B$ be a valued difference field. An element $a$ of a valued difference field extension of $B$ is \emph{regular over $B$} iff for every $\sigma$-polynomial $f$ over $B$ we have $v(f(a))=\trop(f)(v(a))$.
\end{defin}
\begin{rem}\label{rem:regacv}
If $a$ is regular over $B$, and $c$ is such that  $(\ac(a), v(a))=(\ac(c), v(c))$, then $c$ is also regular over $B$.
\end{rem}
\begin{proof}
  This is essentially \cite[Lemma~3.2]{DO}, together with the observation that $v(c-a)>v(a)$ if and only if $(\ac(a), v(a))=(\ac(c), v(c))$. This equality can be easily shown by using that the pair of maps $(\ac,v)$ induces an isomorphism of $k\times \Gamma$ with the \emph{$\mathrm{RV}$-structure} $\mathrm{VF}^\times/(1+\mathfrak m)$, where $\mathfrak m$ denotes the maximal ideal.
\end{proof}

\subsection{Amalgamation}

\begin{ass}
From now on, all fields will be assumed to have characteristic zero.
\end{ass}

\begin{lemma}\label{lemma:doreg}
Let $U,N$ be valued difference fields such that no $\sigma^\ell$ is the identity on $k(U)$ nor on $k(N)$. Let $E$ be a common valued difference subfield,  and let $\gamma\in \Gamma(N)\setminus\Gamma(E)$.   Assume $U$ is quantifier-free $\abs E^+$-saturated.
\begin{enumerate}
\item There is $a\in U$ that is regular over $E$ with $v(a)=\gamma$.
\item The value group of $E\strgen{a}$ is $\Gamma(E)\strgen{\gamma}$, and its residue field has size at most $\abs E$.
\item Suppose $b\in N$ is regular over $E$ and the identity on $\Gamma(E)$ extends to an isomorphism $\Gamma(E)\strgen{\gamma}\to \Gamma(E)\strgen{v(b)}$ sending $\gamma$ to $v(b)$. Then the identity on $E$ extends to a valued difference field isomorphism $E\strgen{a}\to E\strgen{b}$ sending $a$ to $b$.
\end{enumerate}
\end{lemma}
\begin{proof}
This is a variant of \cite[Lemma~8.8]{Pa}. One modifies its proof by replacing the tacitly used \cite[Lemma~8.5]{Pa} by \cite[Lemma~3.6]{DO} (as in \cite[Lemma~6.1]{DO}\footnote{Point (ii) of the cited lemma incorrectly states that, in its notation, the residue field of $E\strgen a$ is the same as that of $E$. In fact, the residue field will in general grow. We thank Simone Ramello for bringing this fact to our attention.}) and observes that the saturation assumption in \cite[Lemma~8.8]{Pa} is only used to realise a quantifier-free type in $U$.
\end{proof}

\begin{lemma}\label{rem:acofreg}
  Let $M$ be an $\mathrm{ac}$-valued difference field, and let $a\in M$ be regular over an $\mathrm{ac}$-valued difference subfield $B$. Let $g_0(x),\ldots, g_n(x)$  be $\sigma$-monomials with coefficients from $B$ such that, for all $i\neq j$, the $\sigma$-monomial $g_i(x)$ is not of the form $bg_j(x)$ for any $b\in B$.
If all $g_i(a)$ have the same valuation, then $\ac(g_0(a)+\ldots+g_n(a))=\ac(g_0(a))+\ldots+\ac(g_n(a))$.
\end{lemma}
\begin{proof}
  By regularity, $g_0(a)+\ldots+g_n(a)$ has the same valuation as each  $g_i(a)$, say $\gamma$. In some extension of $M$, let $c$ be such that $v(c)=\gamma$ and $\ac(c)=1$. Then both sides of the equality equal $\res((g_0(a)+\ldots+g_n(a))/c)$.
\end{proof}

\begin{notation}\label{notation:mcL}
Until the end of the section, $\mathcal{L}$ will denote one of $\mathcal{L}_{\mathrm{ac},\sigma}$, $\mathcal{L}_{s,\sigma}$, $\mathcal{L}_{\mathrm{ac},\iota,\sigma}$, or $\mathcal{L}_{s,\iota,\sigma}$.
\end{notation}
		
\begin{defin}
	 Let $M$ and $N$ be $\mathcal{L}$-structures. A \emph{$k$-elementary map} is an isomorphism $f\from A \to B$ of $\mathcal{L}$-substructures of $M$ and $N$ respectively such that $f_{|k(A)}$ is an elementary map $k(M) \rightarrow k(N)$.
\end{defin} 

We now show how to extend $k$-elementary maps to make sure that the valuation and residue maps on the domain substructure are surjective (note that, in the absence of a cross-section, the valuation map in a substructure need not be surjective, and similarly for lift and residue). 

\begin{lemma}\label{lem:k}
	Let $E$ be a common $\mathcal{L}$-substructure of the $\mathcal{L}$-valued difference fields $U$ and $N$. Let $f\from A \rightarrow B$ be a $k$-elementary map extending $\id_{E}$. Assume that $U$, $N$ are $\sigma$-henselian and that $k(N)$ is $\abs{k(A)}^+$-saturated. 

Then, for all  $\alpha \in k(U)$, the map $f$ extends to a $k$-elementary map $f'\from A' \rightarrow B'$, with domain of cardinality $\abs{A'} = \abs A$, such that $\alpha \in \res(\mathcal{O}(A'))$ and $\Gamma(A')=\Gamma(A)$.
\end{lemma}

\begin{proof}
  Let $\alpha\in k(U)\setminus \res(\mathcal O(A))$. As $f|_{k(A)}\from k(U)\to k(N)$ is an elementary map, by saturation there is some $\beta\in k(N)$ such that the map $f|_{k(A)}\cup \{(\alpha,\beta)\}$ is elementary. We now give a separate argument for each language.

If $\mathcal L=\mathcal L_{\ac,\iota,\sigma}$, then we may take as $A'$ and $B'$ the structures generated by $A\alpha$ and $B\beta$ respectively. It is then clear that $f \cup \{(\iota (\alpha),\iota(\beta))\}$ extends to the desired $k$-elementary map.

Suppose now $\mathcal L=\mathcal L_{s,\iota,\sigma}$.
	For all $i \in \mathbb{Z}$, let $a_i\coloneqq \iota(\sigma^i(\alpha))$ and $b_i\coloneqq \iota(\sigma^i(\beta))$. We prove by induction on $i \in \mathbb{N}$ that the $k$ sort of the $\mathcal L_{s,\iota}$-structure generated by $k(A)\cup \operatorname{VF}(A)(a_{-i},\ldots, a_{i})$ equals $k(A)(\sigma^{-i}(\alpha),\ldots, \sigma^{i}(\alpha))$ and that there is an $\mathcal L_{s,\iota}$-isomorphism $f_i$ extending $f$ and sending the tuple $(a_{-i},\ldots ,a_i)$ to $(b_{-i},\ldots ,b_i)$. Suppose this holds for some $i \in \mathbb{N}$. Applying  \cite[Proposition~1.7]{Kesting} we obtain an $\mathcal L_{s,\iota}$-isomorphism extending $f_i$ and sending $a_{i+1}$ to $b_{i+1}$. Applying \cite[Proposition~1.7]{Kesting} again, this time to $a_{-i-1}$, we obtain the desired $k$-elementary map $f_{i+1}$.  We now conclude by taking $f'\coloneqq \bigcup_{i \in \mathbb{N}} f_i$.

Suppose now that $\mathcal L=\mathcal L_{\ac,\sigma}$. We may assume $\alpha\in k(A)$ and $\beta\in k(B)$. If $\alpha$ is $\sigma$-transcendental over $\res(\mathcal O(A))$, then $\beta$ is $\sigma$-transcendental over $\res(\mathcal O(B))$. Pick $a\in \mathcal O(U)$ and $b\in \mathcal O(N)$ with $\res(a)=\alpha$ and $\res(b)=\beta$. By \Cref{2.5}\ref{point:astvf} $f$ extends to an $\mathcal L_{\sigma}$-isomorphism $f'\from A \langle a\rangle\to B \langle b\rangle$ sending $a$ to $b$, and by \Cref{2.5}\ref{point:astres} $\Gamma(A)=\Gamma(A\strgen a)$. By the choice of $\beta$, $f'|_{k(A \langle a\rangle)}$ is an elementary map $k(U)\to k(N)$. We then use \Cref{fact:avd6364} to see that $f'$ commutes with $\ac$.  Thus, $f'$ is a $k$-elementary map.
	
	Now assume that $\alpha$ is $\sigma$-algebraic over $\res(\mathcal O(A))$, and fix a minimal $\sigma$-polynomial $g(x)$ witnessing this. Let $G(x)$ be a $\sigma$-polynomial over $\mathcal O(A)$ with residue $g(x)$ and of the same complexity as $g$. Let $a_0\in \mathcal O(U)$ be such that $\res(a_0)=\alpha$. We will find $a\in \mathcal O(U)$ with $G(a)=0$ and $\res(a)=\alpha$. If $G(a_0)=0$ we simply set $a\coloneqq a_0$. Otherwise, as $U$ is $\sigma$-henselian, it suffices to show that $(G,a_0)$ is in $\sigma$-Hensel configuration with $\gamma(G,a_0)>0$. To this end, we check that the assumptions of \Cref{fact:fourtwo} are satisfied.\footnote{While we will not needs this, let us remark that one can check the definition $\sigma$-henselianity directly, which yields that $\gamma(G,a_0)\in \set{v(G(a_0)), \sigma^n(v(G(a_0)))}$.} Clearly $G$ is nonconstant, and we have assumed $G(a_0)\ne 0$. Furthermore, by construction $v(G(a_0))>0$. If there was   $\mathbf{i}$  such that $G_{(\mathbf{i})}\ne 0$ but $v(G_{(\mathbf{i})}(a_0))>0$, then we would obtain a $\sigma$-polynomial of lower complexity over $\res(\mathcal O(A))$ vanishing on $\alpha$, contradicting minimality of $g=\res(G)$. Similarly, $\sigma$-henselianity of $N$ allows us to obtain $b\in \mathcal O(N)$ with $(f(G))(b)=0$ and $\res(b)=\beta$. By  \Cref{2.6} $f$ extends to an $\mathcal L_{\sigma}$-isomorphism $f'\from A\strgen{a}\to B\strgen{b}$ sending $a$ to $b$ and $\Gamma(A)=\Gamma(A\strgen a)$.
	By  \Cref{fact:avd6364}, $f'$ commutes with $\ac$, and by the choice of $\beta$, the map $f'|_{k(A\strgen a)}\from k(U)\to k(N)$ is partial elementary. Thus, $f'$ is a $k$-elementary map.

If $\mathcal L=\mathcal L_{s,\sigma}$, we use the same proof as in the previous case, except that we do not need to invoke \Cref{fact:avd6364}. Commutation with the cross-section is given by the fact that $\Gamma$ does not grow.
\end{proof}

\begin{lemma}\label{lem:Gamma}
	Let $E$ be a common substructure of the $\mathcal L$-valued difference fields $M$ and $N$ such that there is an embedding $e$ of $\Gamma(M)$ into $\Gamma(N)$ over $\Gamma(E)$. Let $f\from A \rightarrow B$ be a $k$-elementary map extending $\id_E$ with $f|_{\Gamma(A)}=e|_{\Gamma(A)}$.
Let $U$ be an extension of $M$, and assume that $U$, $N$ are $\sigma$-henselian, that $k(N)$ is $\abs{k(A)}^+$-saturated, and that $U$ is quantifier-free $\abs M^+$-saturated.

Then, for all $\gamma \in \Gamma(M)$, the map $f$ extends to a $k$-elementary map $f'\from A' \rightarrow B'$, with domain $A'\subseteq U$ of cardinality  $\abs{A'} =\abs{A}$, such that $\gamma \in v(A')$ and $\Gamma(A')\subseteq \Gamma(M)$.
\end{lemma}
\begin{proof}
Let $\gamma\in\Gamma(M)\setminus v(A)$ and $\delta\coloneqq e(\gamma)$.  We give a separate argument for each language.

Suppose $\mathcal L=\mathcal L_{\ac,\sigma}$. By \Cref{fact:sigmahenslindifcl} we are in the assumptions of \Cref{lemma:doreg}, hence there is $a\in U$ regular over $\operatorname{VF}(A)$ with $v(a)=\gamma$.  By \Cref{lem:k}, we may extend $f$ to some element of $\mathcal{O}(U)$ with residue $\ac(a)$. Pick $b\in N$  with $v(b)=\delta$ and $\ac(b)=f(\ac(a))$. By \Cref{rem:regacv}, $b$ is regular over $B$. By \Cref{lemma:doreg}, we may extend $f$ to a valued difference field isomorphism sending $a$ to $b$.  From the fact that $\ac(b)=f(\ac(a))$ and  \Cref{rem:acofreg}, we conclude that $f$ commutes with $\ac$.

If $\mathcal L=\mathcal L_{\ac,\iota,\sigma}$, argue as in the previous case and then, to check that $f$ also commutes with the lift, use the fact $\ac(a)$ is transcendental by regularity of $a$.

Suppose now that $\mathcal L=\mathcal L_{s,\sigma}$.
	For all $i \in \mathbb{Z}$, let $a_i\coloneqq s(\sigma^i(\gamma))$
	and $b_i\coloneqq s(\sigma^i(\delta))$. Denote by $\strgen X$ the structure generated by $X$ in the language $\mathcal L_{s}$.
We will prove by induction on $i \in \mathbb{N}$ that  $\Gamma(A\langle a_{-i},\ldots, a_{i} \rangle)=\Gamma(A)+\mathbb Z \sigma^{-i}(\gamma)+\ldots +\mathbb Z\sigma^{i}(\gamma) $ and that there is an $s$-valued field isomorphism $f_i\from A\langle  a_{-i},\ldots ,a_i \rangle \to B \langle b_{-i},\ldots ,b_i  \rangle$ extending $f$ and sending the tuple $(a_{-i},\ldots ,a_i)$ to $(b_{-i},\ldots ,b_i)$, which is enough to prove the lemma as then $f'\coloneqq \bigcup_{i \in \mathbb{N}} f_i$ is the desired isomorphism. Suppose this holds for some $i \in \mathbb{N}$. We will extend $f_i$ to an $\mathcal L_s$-isomorphism \[f_i^+\from A\langle  a_{-i},\ldots , a_i ,a_{i+1}\rangle \to B \langle b_{-i},\ldots ,b_i ,b_{i+1}  \rangle.\] We then repeat the argument to deal with $a_{-i-1}$ and $b_{-i-1}$. 
	
	Suppose first that for all $\ell \in \mathbb{N} \setminus \{0\}$ we have $\ell\sigma^{i+1}(\gamma) \notin \Gamma(A)+\mathbb Z \sigma^{-i}(\gamma)+\ldots +\mathbb Z\sigma^{i}(\gamma)=\Gamma(A\langle a_{-i},\ldots, a_{i} \rangle)$. Then by  \Cref{fact:generating+nontorsion}\ref{fact_generating} we have that $a_{i+1}$ and $b_{i+1}$ are transcendental over $\VF(A\langle  a_{-i},\ldots ,a_i\rangle)$ and $\VF(B\langle b_{-i},\ldots, b_i   \rangle)$ respectively, and that the field isomorphism $f_i^+$ extending $f_i$ and sending $a_{i+1}$ to $b_{i+1}$ commutes with the valuation, and by construction it commutes with the cross-section. \Cref{fact:generating+nontorsion}\ref{fact_generating} also gives us that \[v(A\langle  a_{-i},\ldots , a_{i+1}\rangle)=\Gamma(A)+\mathbb Z a_{-i}+\ldots + \mathbb Z a_{i+1}.\] 
	
	If instead there is some positive $\ell \in \mathbb{N}$ such that $\ell\sigma^{i+1}(\gamma) \in \Gamma(A)+\mathbb Z \sigma^{-i}(\gamma)+\ldots +\mathbb Z\sigma^{i}(\gamma)$, take the minimal such $\ell$, together with integer coefficients $h_i$, and $\gamma_0 \in \Gamma(A)$ such that $\ell\sigma^{i+1}(\gamma)=\gamma_0+\sum_{-i\leq j\leq i} h_i \sigma^{i}(\gamma)$. Applying the cross-section, we get that $a_{i+1}^{\ell}=s(\gamma_0)\prod_{-i\leq j\leq i}a_j^{h_j}$
	and $b_{i+1}^\ell =s(f(\gamma_0))\prod_{-i\leq j\leq i} b_j^{h_j}$. It follows then by  \Cref{fact:generating+nontorsion}\ref{fact_generating} that $x^\ell -s(\gamma_0)\prod_{-i\leq j\leq i} a_j^{h_j}$ is the minimal polynomial of $a_{i+1}$ over $\VF(
	A\langle  a_{-i},\ldots ,a_{i}\rangle)$, and $x^\ell -f(s(\gamma_0))\prod_{-i\leq j\leq i} b_j^{h_j}$ is the minimal polynomial of $b_{i+1}$ over $\VF(B\langle  b_{-i},\ldots ,b_{i}\rangle)$. Because we have $f_i(\prod_{-i\leq j\leq i} a_j^{h_j} )=\prod_{-i\leq j\leq i} b_j^{h_j}$, we get that $f_i$ extends to a map  \[f_i^+\from A\langle  a_{-i},\ldots, a_{i+1}\rangle \to B \langle  b_{-i},\ldots ,b_{i+1}\rangle\] sending $a_{i+1}$ to $b_{i+1}$ and whose $\mathrm{VF}$ part is a field isomorphism. It also follows from  \Cref{fact:generating+nontorsion}\ref{fact_generating} that $f_i^+$ commutes with the valuation map, and it commutes with the cross-section by construction.  \Cref{fact:generating+nontorsion}\ref{fact_generating}  also gives us that $v(A\langle  a_{-i},\ldots ,a_{i+1} \rangle)=\Gamma(A)+\mathbb Z a_{-i}+\ldots + \mathbb Z a_{i+1}$, as required.

If $\mathcal L=\mathcal L_{s,\iota,\sigma}$, the proof is analogous to the proof of the case $\mathcal L_{s,\iota,\sigma}$ in  \Cref{lem:k}
	with the roles of the sorts $\Gamma$ and $k$ interchanged, the role of $\iota$ now played by $s$, and the applications of  \cite[Proposition 1.7]{Kesting}  replaced by applications of \cite[Proposition 2.1]{Kesting}.	
\end{proof}

\begin{rem}
Above, if $\mathcal L$ contains a cross-section symbol $s$, we may furthermore ensure that $A'\subseteq M$ and that $\res(\mathcal O(A'))=\res(\mathcal O(A))$. As we will never use this, we leave it to the reader to check that this follows from the constructions given.
\end{rem}

We are now ready to prove an embedding result, using which we will be able to characterise the solvable amalgamation problems, as well as the existentially closed $\mathcal L$-valued difference fields.

\begin{pr}\label{pr:embeddingvdf}
  Let $E,M,N$ be substructures of $\mathcal{L}$-valued difference fields of residual characteristic zero, with $M$ and $N$ extending $E$ and such that:
  \begin{enumerate}
  \item\label{point:embeddingvdf1}  $N$ is quantifier-free $\abs{M}^+$-saturated, $\sigma$-henselian, with $k(N)$ an $\abs{M}^+$-saturated model of $\mathsf{ACFA}$,
  \item\label{point:embeddingvdf2} $\Gamma(M)$ embeds in $\Gamma(N)$ over $\Gamma(E)$, and
  \item\label{point:embeddingvdf3} $k(M)$ embeds in $k(N)$ over $k(E)$.
  \end{enumerate}
  Then there is an embedding of $M$ into $N$ over $E$.
\end{pr}
\begin{proof}
	By extending $M$ and $N$, we may assume that they are $\mathcal{L}$-valued difference fields.
	
	Since $k(M)$ embeds into $k(N)$ over $k(E)$, there is an existentially closed (in particular, $\sigma$-henselian by \Cref{henselian}) and quantifier-free $\abs{M}^+$-saturated extension $U$ of $M$ such that the restrictions to the algebraic closure of $k(E)$ of the automorphisms of $k(U)$ and $k(N)$ are isomorphic. By~\cite[Theorem~1.3]{CHr} $\id_{E}$ is then a $k$-elementary map between substructures of $U$ and $N$.
	
	By a standard inductive argument, in order to embed $M$ into $N$ over $E$, it is enough to show that if $f\from A\to B$ is a $k$-elementary map with $\abs{A}\leq \abs{M}$ and $a\in U$, then there is a $k$-elementary map $f'\from A'\to B'$ with $a\in A'$ and $\abs{A'}\leq \abs{M}$.\footnote{Note that $A'$ may contain points of $U\setminus M$.}  Additionally, we inductively maintain the condition that $\Gamma(A)\subseteq \Gamma(M)$.

By repeated applications of \Cref{lem:Gamma}, we extend $f$ to a $k$-elementary map $f_0\from A_0\to B_0$ with $v(A_0)=\Gamma(A_0)=\Gamma(M)$ and $\abs{A_0}\le \abs M$. By repeated applications of \Cref{lem:k}, we then find a $k$-elementary map $f_1\from A_1\to B_1$ extending $f_0$ with $\res(\mathcal O(A_1))=k(A_1)\supseteq k(M)$, $v(A_1)=\Gamma(A_1)=\Gamma(M)$ and $\abs{A_1}\le \abs M$. By further iterating \Cref{lem:k} if necessary, we may assume that $k(A_1)$ is linearly difference closed. Since $A_1$ has the same value group as $M$ and larger residue field, the maximal immediate extension $A_2$ of $A_1$ inside $U$ contains $M$; observe that $\abs{A_2}\le\abs{k(A_1)}^{\abs{\Gamma(M)}}\le 2^{\abs M}$. 
\begin{claim}\label{imm}
The $k$-elementary map $f_1\from A_1\to B_1$ extends to a $k$-elementary map $f_2\from A_2\to B_2$.
\end{claim}
\begin{claimproof}
Observe that the proof of \cite[Corollary 5.10]{DO}\footnote{For the reader's convenience, we point out that most of the proofs in \cite[Section 5]{DO}, which culminates in the corollary of interest, are omitted as they heavily rely on similar material presented in \cite[Section 5]{AvD}.} only uses saturation to realise sets of quantifier-free formulas. Namely, the only step where saturation is used is in the proof of \cite[Lemma~5.9]{AvD}, in saying that, in our notation, certain \emph{pseudo-Cauchy sequences} from $B_1$ have pseudolimits (see for example \cite[Section 2]{DO} for the definitions of pseudo-Cauchy sequence, or \emph{pc-sequence}, and of pseudolimit). To conclude, observe that if $\pi(x)$ is a set of quantifier-free formulas over $B_1$ stating that $x$ is a pseudolimit of a certain pc-sequence then, by definition of pc-sequence, $\pi(x)$ is finitely satisfiable in $B_1$, and a fortiori in $N$.
\end{claimproof}
As the extension $A_1\subseteq A_2$ is immediate, in order to conclude that $f_2$ is a $k$-elementary map it remains to be shown that, if $\mathcal L$ contains an angular component symbol, then $f_2$ commutes with $\ac$. This follows from the fact, because $a\in A_2$ is a pseudolimit of a pseudo-Cauchy sequence,  there is $b\in A_1$ such that $v(b-a)>v(a)$.

We conclude the proof by simply restricting $f_2$ to $M$.
\end{proof}
Recall that $K$ is an \emph{amalgamation base} iff every amalgamation problem $M\leftarrow K\to L$ has a solution.
\begin{thm}\label{AP_VDF}
  An amalgamation problem of valued difference fields of residual characteristic zero with either an angular component or a cross-section, and possibly with a lift, has a solution if and only if the induced problem on the residue difference fields has a solution.

  In particular, $K$ is an amalgamation base if and only if $k(K)$ is an amalgamation base, if and only if $\sigma$ has a unique extension to the algebraic closure of $k(K)$ up to isomorphism.
\end{thm}
\begin{proof}
  If the amalgamation problem $M \leftarrow K \to L$ has a solution, then this clearly descends to the residue difference field level, so we prove the converse. Let $N_0$ be a valued difference field with one of the additional structures mentioned above, whose residue difference field $k(N_0)$ solves the amalgamation problem $k(M) \leftarrow k(K) \to k(L)$. Let $N$ be an existentially closed,  existentially $\abs M^+$-saturated extension of $N_0$, so that $k(M)$ and $k(L)$ embed into $k(N)$ over $k(K)$. By \Cref{co:acec_basics} and \Cref{henselian}, assumption~\ref{point:embeddingvdf1} of \Cref{pr:embeddingvdf} is satisfied. Moreover, as $\Gamma(N)$ is existentially closed and existentially $\abs{\Gamma(M)}^+$-saturated, it follows from \cite[Theorem~5.10]{DM} that assumption~\ref{point:embeddingvdf2} is satisfied as well, while assumption~\ref{point:embeddingvdf3} is satisfied by choice of $N_0$. Hence $M$ embeds into $N$ over $K$, solving the amalgamation problem.
  
  For the ``in particular'' clause, the first part is clear. The fact that if the residue automorphism extends uniquely to the algebraic closure then $k(K)$ is an amalgamation base for difference fields follows from \cite[Lemma~8]{Ma}. Conversely, if $(k(K),\sigma)$ is an amalgamation base and $\sigma_1$, $\sigma_2$ are two extensions to the algebraic closure, then any solution of the amalgamation problem $(k(K)^{\mathrm{alg}},\sigma_1) \leftarrow (k(K), \sigma) \to (k(K)^{\mathrm{alg}},\sigma_2)$ induces an isomorphism between $(k(K)^{\mathrm{alg}},\sigma_1)$ and $(k(K)^{\mathrm{alg}},\sigma_2)$.
\end{proof}

  Solvable amalgamation problems appear in the difference algebra literature under the name of \emph{compatible extensions}. See e.g.~\cite[Chapter~7]{cohnDifferenceAlgebra1965}. See \cite[Lemmas~2.8 and 2.9(4)]{CHr}  for a condition ensuring that a field automorphism extends uniquely to the algebraic closure.

\begin{thm}\label{ec_characterisation} 
	A valued difference field $K$ of residual characteristic zero with either an angular component or a cross-section, and possibly with a lift, is existentially closed if and only if it satisfies the following conditions.
	\begin{enumerate}
		\item\label{cond:hens} $K$ is $\sigma$-henselian.
		\item\label{cond:generic} The automorphisms induced on $\Gamma(K)$ and $k(K)$ are generic.
	\end{enumerate}
      \end{thm}
\begin{proof}
  The conditions are necessary by \Cref{co:acec_basics,co:reductswithlifts,henselian}.

  For the converse, suppose $K$ satisfies the given conditions.
Let $N$ be a $\abs{K}^+$-saturated elementary extension of $K$.
 By elementarity,  $N$ is $\sigma$-henselian and $k(N) \models \mathsf{ACFA}$.
 Suppose that $\phi(x,a)$ is a quantifier-free formula with $a\in K$ which is satisfied by some $b$ in some extension of $K$. 

    Let $\bar \gamma$ be a tuple enumerating $\Gamma(K\strgen{b})$. Note that $\qftp(\bar \gamma/\Gamma(K))$ is finitely satisfiable in $\Gamma(K)$, by existential closedness of the latter. In particular, it is finitely satisfiable in $\Gamma(N)$, and so by saturation of $N$ it is realised in $N$ by some tuple $\bar \gamma '$.  Hence, there is an embedding $\Gamma(K\strgen{b})\to \Gamma(N)$  over $\Gamma(K)$.  As $\mathsf{ACFA}$ is model complete, the embedding of $k(K)$ in any model of $\mathsf{ACFA}$ containing $b$ is elementary,  hence there is an embedding $k(K\strgen{b})\to k(N)$ by saturation of $k(N)$. We may therefore apply \Cref{pr:embeddingvdf}, hence $N\models \exists x\;\phi(x,a)$, and we conclude by elementarity.
\end{proof}
Note that, in this characterisation of existentially closed structures,
the only non-first-order part is the assumption that $\sigma_\Gamma$ is generic.

\begin{eg}\label{eg:ecHahn}
If $(k,\sigma)$ and $(\Gamma,\sigma)$ are both existentially closed and $\operatorname{char}(k)=0$, then the Hahn field $k((\Gamma))$, with $\sigma$ and $\ac$ as in \Cref{eg:Hahnsigma}, is $\sigma$-henselian by  \cite[Corollary~5.6(2)]{DO}. By \Cref{ec_characterisation}, it is existentially closed.
\end{eg}

The results above are in languages with at least a definable angular component. Without auxiliary maps, the following problems remain open.
\begin{prob}
  In the language $\mathcal L_\sigma$, characterise
  \begin{enumerate}
  \item the solvable amalgamation problems, and
  \item the existentially closed models.
  \end{enumerate}
\end{prob}

\section{The tree property of the second kind}
\label{sec:ntp2}
In this section we prove that, for $\mathcal L$ any of the languages considered above (see \Cref{notation:mcL}), the positive theory of $\mathcal L$-valued difference fields of residual characteristic zero is $\mathsf{NTP}_2$. We do this by adapting the proof of~\cite[Theorem~4.1]{CH}. This will require generalising certain statements about indiscernible arrays to the positive setting, which we do in the first two subsections. Along the way, we prove the Array Modelling Theorem for arbitrary positive theories, removing the thickness assumption used in~\cite{Ka}, and deduce submultiplicativity of burden from it. We also provide an example showing that Lemma 2.3 in~\cite{Chernikov} is false (this, however, does not affect the correctness of any other results of~\cite{Chernikov}).

Henceforth we assume familiarity with positive logic. We refer the reader to~\cite{BYP} for the fundamentals; condensed accounts are available for instance in~\cite{DK,DGK,DM}. Terminology-wise, let us just mention that we say \emph{positive theory} for \emph{h-inductive theory}, and that we talk of \emph{positively existentially closed} (\emph{pec}) models, while ``existentially closed'' is reserved for the classical notion. We refer to \cite[Definitions~4.1 and~4.5]{DGK} for the definitions of  $\mathsf{IP}$  and $\mathsf{TP_2}$ in the positive setting. From now on, unless otherwise stated, we work in a monster model $\monster$ of a positive theory with the \emph{Joint Continuation Property} (\emph{JCP}). By a \emph{formula} we will mean a positive formula, and by a \emph{type} we will mean a (partial) positive type, we write $\equiv_A$ for having the same positive type over $A$, if we call a sequence \emph{indiscernible} we mean with respect to positive formulas, etc. Similarly, every mention of $\mathsf{NTP}_2$ and similar properties refers to their versions for positive logic.

We now recall some definitions and facts regarding generalised indiscernibles.

\begin{defin}
Let $\mathfrak L$ be a language and $I$, $J$ two $\mathfrak L$-structures indexing some tuples $(a_i)_{i\in I}$ and $(b_j)_{j\in J}$.
\begin{enumerate}
\item The \emph{$\mathrm{EM}_{\mathfrak L}$-type} $\operatorname{EM}_{\mathfrak L}((a_i)_{i\in I}/A)$ of $(a_i)_{i\in I}$ over $A$ is the set of formulas $\phi(x_{\bar \eta})$ over $A$, as $\bar \eta$ varies among finite tuples in $I$, such that, for every tuple $\bar \xi$ from $I$ with the same quantifier-free $\mathfrak L$-type as $\bar \eta$, we have $\models \phi(a_{\bar \xi})$.
\item We say that $(a_i)_{i\in I}$ is \emph{$\mathfrak L$-indiscernible over $A$} iff \[\operatorname{EM}_{\mathfrak L}((a_i)_{i\in I}/A)=\tp((a_i)_{i\in I}/A).\]
\item We say that $(b_j)_{j\in J}$ is \emph{$\mathrm{EM}_{\mathfrak L}$-based on $(a_i)_{i\in I}$ over $A$} iff, for all $\phi(x_{\bar\eta})\in \operatorname{EM}_{\mathfrak L}((a_i)_{i\in I}/A)$ and all tuples $\bar \xi$ from $J$ with the same quantifier-free $\mathfrak L$-type as $\bar \eta$, we have $\phi(b_{\bar\xi})\in \operatorname{EM}_{\mathfrak L}((b_i)_{i\in I}/A)$.
  \item We say that $(b_j)_{j\in J}$ is \emph{based on $(a_i)_{i\in I}$ over $A$} iff, for all tuples $\bar \xi\subseteq J$ there is $\bar \eta\subseteq I$,  with the same quantifier-free $\mathfrak L$-type as $\bar \xi$, such that $b_{\bar\xi}\equiv_A a_{\bar\eta}$.
\end{enumerate}
\end{defin}
As usual, when using the above definitions, if we omit the ``over $A$'' we mean ``over $\emptyset$''.

\begin{rem}\label{rem:based}\*
  \begin{enumerate}
  \item\label{point:embased}  $\mathrm{EM}_{\mathfrak L}$-basedness is most significant (and will only be used below) when $I$ and $J$ realise the same finitary quantifier-free $\mathfrak L$-types, in which case it is a transitive notion~\cite[Remark~3.8]{Ka}. Note that, if $I=J$, then being ``$\mathrm{EM}_{\mathfrak L}$-based on'' is simply inclusion of $\mathrm{EM}_{\mathfrak L}$-types.

  \item  If $(b_j)_{j\in J}$  is based on   $(a_i)_{i\in I}$ over $A$, then it is also $\mathrm{EM}_{\mathfrak L}$-based on it over $A$, but the converse is false in general~\cite[Proposition~3.10 and Example~3.11]{Ka}.
  \end{enumerate}
\end{rem}

Recall that tuples are called \emph{compatible} when they belong to the same cartesian product of sorts.

When index sets are linear orders, we will write $\mathrm{EM}_<$ instead of $\mathrm{EM}_{\set <}$. In the same setting, the following fact is standard, see e.g.~\cite[Lemma 1.2]{ben-yaacovSimplicityCompactAbstract2003}.
\begin{fact}\label{fact:extracting}
	Let $A$ be a parameter set, $\kappa$ a cardinal, and let\footnote{For the sake of readability, we will use the notation $\beth(\alpha)$ instead of the more commonly used $\beth_\alpha$.} $\lambda \coloneqq \beth((2^{\abs{T} + \abs{A} + \kappa})^+)$. Then for every sequence $(a_i)_{i < \lambda}$ of compatible $\kappa$-tuples there is a sequence $(b_i)_{i < \omega}$ that is $A$-indiscernible and based on $(a_i)_{i < \lambda}$ over $A$.
\end{fact}

For some purposes, $\mathrm{EM}_<$-basing, as opposed to basing, suffices. In that case, one may start from shorter sequences.
\begin{fact}[{\cite[Corollary~5.4]{DK}}]\label{fact:DK54}
  Let $A$ be a parameter set. Then for every sequence $(a_i)_{i < \omega}$ of compatible tuples there is a sequence $(b_i)_{i < \omega}$ that is  $A$-indiscernible and $\mathrm{EM}_<$-based on $(a_i)_{i < \omega}$ over $A$.
\end{fact}

\subsection{Array Modelling and submultiplicativity of burden}
 The first objective of this subsection is to generalise \cite[Theorems~1.3 and 1.4]{Ka} to arbitrary positive theories, dropping the thickness assumption.
We follow notation from there.

Recall that $\mathfrak L_{\ar}$ is the language $\set{<_1,<_2}$, interpreted in a subset $Z$ of a product of linear orders $I\times J$ as $(i,j)<_1(i',j')$ iff $i<i'$ and $(i,j)<_2(i',j')$ iff $i=i'$ and $j<j'$. That is, $<_1$ is the order \emph{on} rows and $<_2$ is the order \emph{internal to} rows. For $\lambda$ an ordinal, $\mathfrak L_\lambda$ denotes the expansion of $\mathfrak L_{\ar}$ by predicates $P_\alpha$, for $\alpha<\lambda$, which we will interpret in a subset $Z$ of $\lambda\times \mu$ as $Z\cap(\{\alpha\}\times \mu)$.

If $Z$ is as above, we call $(a_{(ij)})_{(i,j)\in Z}$ an \emph{array}. In order to keep the notation concise, we will also simply write $a_Z$, and extend this usage to restrictions, e.g.\ if $Z'\subseteq Z$ we will write $a_{Z'}$ with the obvious meaning. Similarly for e.g.\ $a_{<\alpha, \ge \beta}$. We will use $\bar{}$  to indicate rows, e.g.\ $\bar a_\alpha$ typically will mean the $\alpha$-th row of $a_Z$, and we similarly write e.g.\ $\bar a_{<\alpha}$ for the sequence whose elements are the rows of $a_Z$ up to the $\alpha$-th, etc.

\begin{rem}\label{rem:lindischar}\*
  \begin{enumerate}
  \item An array is $\mathfrak L_\lambda$-indiscernible exactly when it is \emph{mutually indiscernible}, i.e.\ each of its rows is indiscernible over the union of all other rows.
  \item    An array indexed on $I\times J$ is $\mathfrak L_{\ar}$-indiscernible exactly when each of its rows is indiscernible over the union of all other rows and, furthermore, its rows form an indiscernible sequence.
  \end{enumerate}
\end{rem}
We use \emph{strongly indiscernible} as a synonym for ``$\mathfrak L_\mathrm{ar}$-indiscernible''; note that~\cite{Ka} uses the terminology \emph{array-indiscernible} instead.

\begin{lemma}\label{mod_inv}
  Let $I$ be an $\mathfrak L$-structure, $J,L$ be $\mathfrak L'$-structures, and $E$ a set of parameters.
If  $c_L$ is based on  $b_J$ over $a_IE$ and $a_I$ is $\mathfrak L$-indiscernible over $b_JE$, then $a_I$ is also $\mathfrak L$-indiscernible over $c_LE$.
\end{lemma}
\begin{proof}
By naming parameters, we may assume $E=\emptyset$.
Consider any finite tuples $\bar \eta, \bar \nu \subseteq I$ with $\qftp_{\mathfrak L}(\bar\eta)=\qftp_{\mathfrak L}(\bar\nu)$, as well as a finite tuple $\bar \ell\subseteq L$. As $c_L$ is based on $b_J$ over $a_I$,  there is a tuple $\bar \jmath\subseteq J$ with $b_{\bar \jmath}\equiv_{a_I} c_{\bar \ell}$.
As $a_I$ is $\mathfrak L$-indiscernible over $b_J$, we have that $a_{\bar \nu}\equiv_{b_{\bar \jmath}} a_{\bar \eta}$. It  follows that $a_{\bar \nu}\equiv_{c_{\bar \ell}} a_{\bar \eta}$, hence that $a_{\bar \nu}\equiv_{c_L} a_{\bar \eta}$.
\end{proof}

For an ordinal $\lambda$ and $I\subseteq \lambda$, by $\ot (I)$ we will mean the order type of $I$ with the order induced from $\lambda$. For $Z\subseteq \lambda\times \mu$ we write $Z_I\coloneqq Z\cap (I\times \mu)$ and consider it as an $\mathfrak L_{\ot(I)}$-structure in the obvious way.

By inspecting atomic $\mathfrak L_\lambda$-formulas, we easily see that if  $\lambda=I\dot{\cup}J$, $\bar\nu, \bar \nu' \subseteq Z_I$ and $\bar\xi,\bar\xi'\subseteq Z_J$, then $\qftp_{\mathfrak L_\lambda}(\nu,\xi)=\qftp_{\mathfrak L_\lambda}(\nu',\xi')$ if and only if  $\qftp_{\mathfrak L_{\ot(I)}}(\nu)=\qftp_{\mathfrak L_{\ot(I)}}(\nu')$ and $\qftp_{\mathfrak L_{\ot(J)}}(\xi)=\qftp_{\mathfrak L_{\ot(J)}}(\xi')$.

\begin{lemma}\label{I}
Let $\lambda=I\dot{\cup}J$ and $Z,Z'\subseteq \lambda\times \mu$ with $Z_J=Z'_J$.
Suppose $b_{Z'_I}$ is $\mathrm{EM}_{\mathfrak L_{\ot (I)}}$-based on $a_{Z_I}$ over $a_{Z_J}E$. Let

\[
  c_{ij}\coloneqq
  \begin{cases}
    a_{ij}&\text{ when }(i,j)\in Z'_J=Z_J\\
    b_{ij}&\text{ when }(i,j)\in Z'_I
  \end{cases}.
\]
Then $c_{Z'}$ is $\mathrm{EM}_{\mathfrak L_\lambda}$-based on $a_Z$ over $E$.
\end{lemma}
\begin{proof}
  We may assume that $E=\emptyset$.
Let $\bar{\nu}\subseteq Z_I$, $\bar{\xi}\subseteq Z_J$, and $\phi$ be such that $\phi(x_{\bar \nu}, x_{\bar \xi})\in \EM_{\mathfrak L_\lambda}(a_Z)$. We have to prove that $\models \phi(c_{\bar\nu}, c_{\bar\xi})$.

For every  $\bar \nu'$ with $\qftp_{\mathfrak L_{\ot(I)}}(\bar\nu')=\qftp_{\mathfrak L_{\ot(I)}}(\bar\nu)$ we have that $\qftp_{\mathfrak L_\lambda}(\bar \nu',\bar\xi)=\qftp_{\mathfrak L_\lambda}(\bar\nu,\bar\xi)$, so $\models\phi(a_{\bar \nu'}, a_{\bar \xi})$. Hence $\phi(x_{\bar\nu'},a_{\bar\xi})\in \EM_{\mathfrak L_{\operatorname{ot}(I)}}(a_{Z_I}/a_{Z_J})$ and so $\models \phi(b_{\bar\nu},a_{\bar\xi})$ by the assumption on $b_{Z'_I}$. This means that $\models\phi(c_{\bar\nu},c_{\bar\xi})$, as required.
\end{proof}
\begin{lemma}\label{Iind}
Let $\lambda=I\dot{\cup}J$ and  $Z\subseteq \lambda\times \mu$. Suppose $a_{Z}$ is such that $a_{Z_J}$ is $\mathfrak L_{\ot(J)}$-indiscernible over $a_{Z_I}E$ and $a_{Z_I}$ is $\mathfrak L_{\ot (I)}$-indiscernible over $a_{Z_J}E$. 
Then $a_{Z}$ is $\mathfrak L_{\lambda}$-indiscernible over $E$.  
\end{lemma}
\begin{proof}
We may assume $E=\emptyset$.
Let $\bar\nu, \bar\nu'\subseteq Z_I$ and $\bar\xi, \bar\xi'\subseteq Z_J$ be such that $\qftp_{\mathfrak L_{\lambda}}(\bar\nu,\bar \xi)=
\qftp_{\mathfrak L_{\lambda}}(\bar\nu',\bar \xi')$. It follows easily from the assumptions that $a_{\bar\nu, \bar\xi}\equiv  a_{\bar\nu, \bar\xi'}\equiv a_{\bar\nu', \bar\xi'}$.
\end{proof}

\begin{pr}\label{lambda_mod}
Let $A=(a_{ij})_{i<\lambda, j<\omega}$ be an array with each row consisting of compatible tuples, and $E$ a set of parameters. Then there is an array $(b_{ij})_{i<\lambda, j<\omega}$  which is $\mathfrak L_\lambda$-indiscernible over $E$ and $\mathrm{EM}_{\mathfrak L_\lambda}$-based on $A$ over $E$.
\end{pr}
\begin{proof}
We proceed by induction on $\lambda$. For $\lambda=0$ the statement is trivial.

\paragraph{Limit step.}
Let $\lambda$ be a nonzero limit ordinal and let $A$ be as in the statement of the proposition. We may assume that $E=\emptyset$. We will inductively replace the rows of $A$, making the new rows mutually indiscernible over the remaining ones and keeping the array $\mathrm{EM}_{\mathfrak L_{\lambda}}$-based on $A$. More formally, we construct recursively on $\alpha<\lambda$ arrays $A^\alpha=(a^\alpha_{ij})_{i<\lambda,j<\omega}$ such that \begin{enumerate}
\item\label{point:limitstep1} $\bar a^\alpha_{i}=\bar a^{i+1}_{i}$ for every $i<\alpha$,
\item\label{point:limitstep2} $\bar a^\alpha_{i}=\bar a_{i}$ for every $i\geq\alpha$,
\item\label{point:limitstep3} $A^\alpha$ is $\mathrm{EM}_{\mathfrak L_\lambda}$-based on $A$, and
\item\label{point:limitstep4} $a^\alpha_{<\alpha,<\omega}$ is $\mathfrak L_\alpha$-indiscernible over $a^{\alpha}_{\geq\alpha,<\omega}$.
\end{enumerate}

First, let us argue that this is enough to complete the limit step of the induction. Suppose we have $A^\alpha$ as above, and put $b_{ij}\coloneqq a^{i+1}_{ij}$ for all $i<\lambda$ and $j<\omega$. To see that $b_{<\lambda,<\omega}$ is $\mathrm{EM}_{\mathfrak L_\lambda}$-based on $A$, consider any $\phi(x_{\bar\nu})$ in the $\mathrm{EM}_{\mathfrak L_\lambda}$-type of $A$. Then there is $\alpha<\lambda$ such that $\bar\nu\subseteq \alpha\times \omega$. As $A^\alpha$ is $\mathrm{EM}_{\mathfrak L_\lambda}$-based on $A$, we get that $\models \phi(a^\alpha_{\bar \nu})$, but $a^\alpha_{\bar \nu}=b_{\bar\nu}$, so $\models\phi(b_{\bar\nu})$. Similarly, to see that $(b_{ij})_{i<\lambda,j<\omega}$ is $\mathfrak L_\lambda$-indiscernible, let $\bar\nu, \bar\nu'\subseteq \lambda\times \omega$ be finite tuples with $\qftp_{\mathfrak L_{\lambda}}(\bar\nu)=
\qftp_{\mathfrak L_{\lambda}}(\bar\nu')$ and choose $\alpha<\lambda$ such that $\bar\nu, \bar\nu'\subseteq \alpha\times \omega$. As $\bar a^\alpha_{<\alpha}$ is 
in particular $\mathfrak L_\alpha$-indiscernible over $\emptyset$, we get that $a^\alpha_{\bar \nu}\equiv a^\alpha_{\bar \nu'}$, i.e.\ $b_{\bar \nu}\equiv b_{\bar \nu'}$.

We now carry out the construction of the arrays $A^\alpha$.
If $\alpha<\lambda$ is a limit ordinal, we put  $a^\alpha_{ij}\coloneqq a^{i+1}_{ij}$ for $i<\alpha$  and $a^\alpha_{ij}=a_{ij}$ for $i\geq \alpha$. Then~\ref{point:limitstep1} and~\ref{point:limitstep2} are clearly satisfied by this choice and the inductive assumption, and~\ref{point:limitstep3} and~\ref{point:limitstep4} are satisfied as well, due to the inductive assumption, the fact that  every finite tuple of elements of $\alpha\times\omega$ is contained in $\beta \times \omega$ for some $\beta<\alpha$, and arguments similar to those in the previous paragraph.

For the successor step of the construction, suppose we have constructed $A^\alpha$ satisfying conditions~\ref{point:limitstep1}--\ref{point:limitstep4}. Let $\kappa\coloneqq \beth((2^{\lambda+\abs A+\abs T})^+)$ and let $\bar c=(c_i)_{i<\kappa}$  be a 
sequence $\mathrm{EM}_{<}$-based on $\bar a_{\alpha}$ over $\bar a^\alpha_{\neq \alpha}$, which exists by compactness.\footnote{Note that we do not need to require this sequence to be indiscernible; this could nonetheless be ensured by invoking \Cref{fact:DK54}.} By \Cref{I}, the array $\bar a^\alpha_{<\alpha}\frown \bar c\frown \bar a^\alpha_{>\alpha}$ (obtained by replacing the $\alpha$-th row of $A^\alpha$ with $\bar c$) is $\mathrm{EM}_{\mathfrak L_{\lambda}}$-based on $A^\alpha$, and hence on $A$ by inductive assumption on $\alpha$.
By the main inductive assumption, there is an array $(\hat a^\alpha_{ij})_{i<\alpha}$ which is $\mathrm{EM}_{\mathfrak L_\alpha}$-based on $\bar a^\alpha_{<\alpha}$ over $\bar a^\alpha_{>\alpha},\bar c$ and  ${\mathfrak L_\alpha}$-indiscernible over $\bar a^\alpha_{>\alpha},\bar c$.
By \Cref{I} again, $\hat a^\alpha_{<\alpha} \frown \bar c\frown \bar a^\alpha_{>\alpha}$ is $\mathrm{EM}_{\mathfrak L_\lambda}$-based on $\bar a^\alpha_{<\alpha}\frown \bar c\frown \bar a^\alpha_{>\alpha}$, and hence on $A$. As $\bar a^\alpha_{<\alpha}$ is ${\mathfrak L_{\alpha}}$-indiscernible over $\bar a^\alpha_{>\alpha}$, we have  $\EM_{\mathfrak L_{\lambda}}(\bar a^\alpha_{<\alpha}/ \bar a^\alpha_{>\alpha}, \bar c)\supseteq  \tp(\bar a^\alpha_{<\alpha}/ \bar a^\alpha_{>\alpha})$, so $\hat a^\alpha_{<\alpha} \bar a^\alpha_{>\alpha}\equiv \bar a^\alpha_{<\alpha} \bar a^\alpha_{>\alpha}$. Let $f$ be an automorphism over $\bar a^\alpha_{>\alpha}$ sending $\hat a^\alpha_{<\alpha}$ to $\bar a^\alpha_{<\alpha}$. Then $\bar a^\alpha_{<\alpha}$ is ${\mathfrak L_\alpha}$-indiscernible over $f(\bar c), \bar a^\alpha_{>\alpha}$, and $\bar a^\alpha_{<\alpha}\frown f(\bar c) \frown \bar a^{\alpha}_{>\alpha}$ is $\mathrm{EM}_{\mathfrak L_\lambda}$-based on $A$.

Now by \Cref{fact:extracting} we can choose $\bar d=(d_j)_{j<\omega}$ to be $\{<\}$-based on $f(\bar c)$ over $\bar a^\alpha_{<\alpha}, \bar a^\alpha_{>\alpha}$ and indiscernible over  $\bar a^\alpha_{<\alpha}, \bar a^\alpha_{>\alpha}$. By \Cref{I}, $\bar a^\alpha_{<\alpha}\frown \bar d \frown\bar a^\alpha_{>\alpha}$ is $\mathrm{EM}_{\mathfrak L_\lambda}$-based on $\bar a^\alpha_{<\alpha}\frown f(\bar c) \frown \bar a^\alpha_{>\alpha}$, and hence on $A$. Note that, since $\bar a^\alpha_{<\alpha}$ is $\mathfrak L_\alpha$-indiscernible over $f(\bar c), \bar a^\alpha_{>\alpha}$,  by \Cref{mod_inv} it is also $\mathfrak L_\alpha$-indiscernible over $\bar d,\bar a^\alpha_{>\alpha}$. 
As $\bar d$ is  indiscernible over $\bar a^\alpha_{<\alpha}, \bar a^\alpha_{>\alpha}$, it follows from \Cref{Iind} that $\bar a^\alpha_{<\alpha}\frown \bar d$ is $\mathfrak L_{\alpha+1}$-indiscernible over $\bar a^\alpha_{>\alpha}$. Thus putting $a^{\alpha+1}_{ij}\coloneqq a^{\alpha}_{ij}$ for $i\neq \alpha$, and $a^{\alpha+1}_{\alpha,j}\coloneqq d_j$ for all $j<\omega$, we get that $A^{\alpha+1}$ satisfies conditions~\ref{point:limitstep3} and~\ref{point:limitstep4}. It also clearly satisfies conditions~\ref{point:limitstep1} and~\ref{point:limitstep2} by the inductive assumption about $A^\alpha$ and the choice of $a^{\alpha+1}_{ij}$.
This finishes the recursive construction of $(A^\alpha)_{\alpha<\lambda}$, and hence the limit step of the inductive proof of the proposition.

\paragraph{Successor step.} We proceed similarly as in the successor step of the above construction. 
Namely, suppose the statement holds for some $\lambda$ and consider an array $(a_{ij})_{i<\lambda+1, j<\omega}$. We may assume that $E=\emptyset$. By compactness there is a sufficiently long sequence $\bar c$ that is  $\mathrm{EM}_<$-based on $\bar a_{\lambda}$ over $\bar a_{<\lambda}$.
By the inductive hypothesis, there is an $\mathfrak L_\lambda$-indiscernible over $\bar c $ array $(b_{ij})_{i<\lambda,j<\omega}$ that is  $\mathrm{EM}_{\mathfrak L_\lambda}$-based on $\bar a_{<\lambda}$ over $\bar c$.

Use \Cref{fact:extracting} to find a $\bar b_{<\lambda}$-indiscernible sequence $(b_{\lambda,i})_{i<\omega}$ that is based on $\bar c$ over $\bar b_{<\lambda}$.  By \Cref{mod_inv}, $\bar b_{<\lambda}$ is $\mathfrak L_{\lambda}$-indiscernible over $\bar b_\lambda$, hence we may apply \Cref{Iind} and conclude that $\bar b_{<\lambda+1}$  is $\mathrm{EM}_{\mathfrak L_{\lambda+1}}$-indiscernible.

We conclude by observing that, by repeated applications of \Cref{I},  $\bar b_{<\lambda+1}$ is $\mathrm{EM}_{\mathfrak L_{\lambda+1}}$-based on $\bar b_{<\lambda}\frown\bar c$, in turn $\mathrm{EM}_{\mathfrak L_{\lambda+1}}$-based on $\bar a_{<\lambda}\frown\bar c$, itself $\mathrm{EM}_{\mathfrak L_{\lambda+1}}$-based on $\bar a_{<\lambda+1}$. 
\end{proof}

\begin{thm}[Array Modelling]\label{ar_mod}
Let $(a_{ij})_{i,j<\omega}$ be an array of compatible tuples and $E$ a set of parameters. Then there is an array $(b_{ij})_{i,j<\omega}$  which is strongly indiscernible over $E$ and $\mathrm{EM}_{\ar}$-based on $(a_{ij})_{i,j<\omega}$ over $E$.
\end{thm}
\begin{proof}
  Without loss of generality $E=\emptyset$. 
Let $\lambda\coloneqq \beth((2^{\abs T+\abs{a_{0,0}}})^+)$. 
By compactness we can extend $(a_{ij})_{i,j<\omega}$ to an array $(a_{ij})_{i<\lambda, j<\omega}$ which is $\mathrm{EM}_{\mathfrak L_{\ar}}$-based on 
$(a_{ij})_{i,j<\omega}$. By \Cref{lambda_mod}, there is an $\mathfrak L_{\lambda}$-indiscernible array $(c_{ij})_{i<\lambda, j<\omega}$ which is $\mathrm{EM}_{\mathfrak L_\lambda}$-based on $(a_{ij})_{i<\lambda, j<\omega}$. In particular, by \Cref{rem:based}\ref{point:embased}, $(c_{ij})_{i<\lambda, j<\omega}$ is $\mathrm{EM}_{\mathfrak L_{\ar}}$-based on
$(a_{ij})_{i<\lambda, j<\omega}$. By \Cref{fact:extracting} there is an indiscernible sequence $(\bar b_{i})_{i<\omega}$ based on the sequence $\bar c_{<\lambda}$. Then, by \Cref{rem:lindischar},  $(b_{ij})_{i,j<\omega}$ is an $\mathfrak L_{\ar}$-indiscernible array $\mathrm{EM}_{\mathfrak L_{\ar}}$-based on $(a_{ij})_{i,j<\omega}$.
\end{proof}

In~\cite{Ka}, the thickness assumption was used in proving Array Modelling, but not in deriving its consequences. In other words, having proved Array Modelling without assuming thickness, we may derive, by the same proofs as in~\cite[Lemma 6.2 and Theorem 1.4]{Ka}, the following corollaries.

\begin{co}\label{tp2_indis}
A formula $\phi(x, y)$ has $k\mathsf{-TP}_2$ for $k \geq 2$ if and only if there is a strongly indiscernible array $(a_{ij})_{i,j < \omega}$ such that
\begin{enumerate}
\item $\{ \phi(x, a_{i, 0}) \mid  i < \omega\}$ is consistent, and
\item $\{ \phi(x, a_{0, i}) \mid i < \omega\}$ is $k$-inconsistent.
\end{enumerate}
\end{co}

\begin{co}\label{co:tp2iff2tp2} If $\phi(x,y)$ has $k\mathsf{-TP}_2$ for some $k\geq 2$, then $\bigwedge_{i=1}^{n} \phi(x,y_i)$ has $2\mathsf{-TP}_2$ for some $n$.
\end{co}
\begin{rem}\label{infinitary_tp2}
Suppose $\phi(x,y)$, $k$, $\psi(x,y_1,\ldots,y_k)$ and $(a_{ij})_{i,j<\omega}$ witness $\mathsf{TP}_2$. Assume that in the monster model $\phi\equiv \bigwedge_{i\in I} \phi_i$ 
and $\psi\equiv \bigwedge_{j\in J} \psi_j$. Then by compactness there are finite $I_0\subseteq I$ and $J_0\subseteq J$ such that, putting $\phi'\coloneqq \bigwedge_{i\in I_0}\phi_i$ and $\psi'\coloneqq \bigwedge_{j\in J_0}\psi_j$ the formula $\psi'(y_1,\ldots,y_k)\wedge \phi'(x,y_1)\wedge\ldots\wedge\phi'(x, y_n)$ is inconsistent. Then $\phi'(x,y)$, $k$, $\psi'(x,y_1,\ldots,y_k)$ and $(a_{ij})_{i,j<\omega}$ also witness $\mathsf{TP}_2$.
\end{rem}

We now move towards submultiplicativity of burden.

\begin{defin}\label{inp}
Let $\pi(x)$ be a partial type and $\kappa$ be a (possibly finite) cardinal. An \emph{$\inp$-pattern} in $\pi(x)$ of depth $\kappa$ consists of
$(a_{\alpha,i})_{\alpha<\kappa,i<\omega}$ and, for each $\alpha<\kappa$, a formula $\phi_{\alpha}(x,y_{\alpha})$, 
 such that
\begin{enumerate}
\item\label{point:inp1} for every $\alpha<\kappa$ there are  $k_{\alpha}<\omega$ and  $\psi_\alpha(y_{\alpha,0},\ldots,y_{\alpha,{k_\alpha-1}})$ such that \[\models \neg \left(\exists x,y_{\alpha,0},\ldots, y_{\alpha,{k_\alpha -1}} \;\left(\psi_\alpha(y_{\alpha,0},\ldots,y_{\alpha,{k_\alpha-1}})\wedge \bigwedge_{i<k_\alpha}\phi_\alpha(x,y_{\alpha,i})\right)\right)\] and $ \models \psi_{\alpha}(a_{\alpha,i_0}, \ldots,a_{\alpha,i_{k_\alpha-1}})$
for all $i_0<\ldots<i_{k_\alpha-1}<\omega$; and
\item for every $f\from \,\kappa\to\omega$, the set $\pi(x)\cup\left\{ \phi_{\alpha}(x,a_{\alpha,f(\alpha)})\mid \alpha<\kappa\right\}$
is consistent.
\end{enumerate}
The \emph{burden} of $\pi(x)$, denoted by $\bdn(\pi)$, is the supremum
of the depths of all $\inp$-patterns in $\pi(x)$, if one exists, and $\infty$ otherwise. By $\bdn(a/E)$
we mean $\bdn(\tp(a/E))$.
\end{defin}
Clearly, $\pi(x)\subseteq \rho(x)$ implies $\bdn(\pi)\geq\bdn(\rho)$.
Note that, by pigeonhole and compactness, we have $\bdn(\pi)<\infty\iff\bdn(\pi)<\abs T^{+}$. In particular, $T$ has $\mathsf{TP}_2$ if and only if there is partial type of infinite burden.

Following \cite{Chernikov}, we will say that an array $(a_{ij})_{i<\alpha, j<\beta}$ is \emph{almost mutually indiscernible} over $E$ iff for every $i_0<\alpha$ we have that 
$(a_{i_0,j})_{j<\beta}$ is an indiscernible sequence over $a_{<i_0,<\beta} a_{>i_0, 0}E$.

In the next lemma we correct \cite[Lemma 2.3]{Chernikov}, which we will see in \Cref{eg:23cntrex} not to hold as stated in \cite{Chernikov}.

\begin{lemma}
\label{lem: RotateIfConsistent} Let $\bar a=(a_i)_{i<\omega}$ be an $E$-indiscernible sequence and let $b$ be arbitrary. Let $p(x,a_{0})=\tp(b/a_{0}E)$,
and assume that $p_{\infty}(x)=\bigcup_{i<\omega}p(x,a_{i})$
is consistent. Then there is $\bar a'$
such that
\begin{enumerate}
\item $\bar{a}'\equiv_{a_{0}E}\bar{a}$ and
\item $\bar{a}'$ is indiscernible over $Eb$.
\end{enumerate}
\end{lemma}
\begin{proof}
We may assume $E=\emptyset$.
Let $b'\models p_{\infty}(x)$ and use \Cref{fact:DK54} to obtain some $\bar{a}''$ indiscernible over $b'$ and $\mathrm{EM}_<$-based
on $\bar{a}$ over $b'$. 

Note that $b'a''_0\equiv ba_0$, because for every $i$ we have  $b'a_i\equiv ba_0$ by choice of $b'$. Therefore, 
there is an automorphism $f$ sending $b'a_0''$ to $ba_0$. Put $\bar a'\coloneqq f(\bar a'')$.
 Clearly $\bar{a}'\equiv_{a_{0}}\bar{a}$ and, as $\bar a''$ is indiscernible over $b'$, we get that  $\bar a'=f(\bar a'')$ is indiscernible over $b=f(b')$.
\end{proof}

\begin{lemma}
\label{lem: burden by rotation} The following are equivalent for
a partial type $\pi(x)$ over $E$.
\begin{enumerate}
\item \label{point:noinppat}There is no $\inp$-pattern of depth $\kappa$ in $\pi$.
\item \label{point:mutind}For every $b\models \pi(x)$ and every almost mutually indiscernible over $E$ array $\left(\bar{a}_{\alpha}\right)_{\alpha<\kappa}$, there are $\beta<\kappa$
and $\bar{a}'$ indiscernible over $bE$ and such that $\bar{a}'\equiv_{a_{\beta,0}E}\bar{a}_{\beta}$.
\item\label{point:almmutind} For every $b\models \pi(x)$ and every mutually indiscernible over $E$ array $\left(\bar{a}_{\alpha}\right)_{\alpha<\kappa}$, there are $\beta<\kappa$
and $\bar{a}'$ indiscernible over $bE$ and such that $\bar{a}'\equiv_{a_{\beta,0}E}\bar{a}_{\beta}$.
\end{enumerate}
\end{lemma}
\begin{proof}
We may assume $E=\emptyset$.
\begin{description}[leftmargin=*]
\item[$\ref{point:noinppat}\Rightarrow\ref{point:mutind}$]  Let $b\models \pi(x)$ and an almost mutually indiscernible $\left(\bar{a}_{\alpha}\right)_{\alpha<\kappa}$ be given. To find  $\bar{a}'$ as in the conclusion it suffices to check that we are in the assumptions of \Cref{lem: RotateIfConsistent}. Let $p_{\alpha}(x,a_{\alpha,0})\coloneqq\tp(b/a_{\alpha,0})$ and $p_{\alpha,\infty}(x)\coloneqq\bigcup_{i<\omega}p_{\alpha}(x,a_{\alpha,i})$.  If $p_{\alpha,\infty}$ is inconsistent for every $\alpha$ then by compactness
  and indiscernibility of $\bar{a}_{\alpha}$  we can find for every $\alpha$ integers $k_\alpha$ and formulas $\phi_\alpha(x,y)\in \tp(b, a_{\alpha,0})$ and $\psi_\alpha$ satisfying condition~\ref{point:inp1} in \Cref{inp}. As we have that  $b\models\left\{ \phi_{\alpha}(x,a_{\alpha,0})\right\} _{\alpha<\kappa}$,
  by almost indiscernibility of $\left(\bar{a}_{\alpha}\right)_{\alpha<\kappa}$  we get that $\left(\bar{a}_{\alpha}\right)_{\alpha<\kappa}$ is an $\inp$-pattern
  of depth $\kappa$ in $\pi$, a contradiction. Thus $p_{\beta,\infty}(x)$ is consistent for some $\beta<\kappa$. 

  \item[$\ref{point:mutind}\Rightarrow\ref{point:almmutind}$] Obvious.

  \item[$\ref{point:almmutind}\Rightarrow\ref
    {point:noinppat}$] Assume that there is an $\inp$-pattern $\left(\bar{a}_{\alpha},\phi_{\alpha}\right)_{\alpha<\kappa}$ of depth
  $\kappa$ in $\pi(x)$. We may assume $\left(\bar{a}_{\alpha}\right)_{\alpha<\kappa}$ is mutually indiscernible by \Cref{lambda_mod}. 
  Let $b\models \pi(x)\cup \{\phi_\alpha(x,a_{\alpha,0})\mid \alpha<\kappa\}$.
  Then for every $\beta<\kappa$ we have that  $\models\phi_{\beta}(b,a_{\beta,0})$ and
  $\left\{ \phi_{\beta}(x,a_{\beta,i})\right\} _{i<\omega}$
  is inconsistent, so it is impossible to find an $\bar{a}'$
  as required.\qedhere
\end{description}
\end{proof}

\begin{thm}
\label{thm: product array} If there is an $\inp$-pattern of depth
$\kappa_{1}\times\kappa_{2}$ in $\tp(b_{1}b_{2}/E)$, then either
there is an $\inp$-pattern of depth $\kappa_{1}$ in $\tp(b_{1}/b_2E)$
or there is an $\inp$-pattern of depth $\kappa_{2}$ in $\tp(b_{2}/E)$.\end{thm}
\begin{proof}
Assume not. We may assume $E=\emptyset$. We show that condition~\ref{point:almmutind} in \Cref{lem: burden by rotation} is satisfied.

Let $\left(\bar{a}_{\gamma}\right)_{\gamma\in\kappa_{1}\times\kappa_{2}}$
be a mutually indiscernible array, where we consider the product $\kappa_{1}\times\kappa_{2}$
lexicographically ordered.

By induction on $\alpha<\kappa_{1}$ we
choose $\bar{a}_{\alpha}'$ and $\beta_{\alpha}<\kappa_{2}$ such
that:
\begin{enumerate}
\item\label{point:prodar1} $\bar{a}_{\alpha}'$ is indiscernible over $b_{2}\bar{a}_{<\alpha}'\bar{a}_{\geq(\alpha+1,0)}$,
\item\label{point:prodar2} $\bar{a}_{\alpha}'$ has the same type as $\bar{a}_{(\alpha,\beta_{\alpha})}$ over $a_{(\alpha,\beta_{\alpha}),0}\bar{a}_{<\alpha}'\bar{a}_{\geq(\alpha+1,0)}$, and
\item\label{point:prodar3} $\bar{a}_{\leq\alpha}'\cup\bar{a}_{\geq(\alpha+1,0)}$ is a mutually
indiscernible array.
\end{enumerate}
Assume we have managed up to $\alpha$, and we need to choose $\bar{a}_{\alpha}'$
and $\beta_{\alpha}$. Let $D=\bar{a}_{<\alpha}'\bar{a}_{\geq(\alpha+1,0)}$.
As $\left(\bar{a}_{(\alpha,\delta)}\right)_{\delta\in\kappa_{2}}$
is a mutually indiscernible array over $D$ (by assumption in the
case $\alpha=0$ and by~\ref{point:prodar3} of the inductive hypothesis in the other
cases) and there is no $\inp$-pattern of depth $\kappa_{2}$ in $\tp(b_{2}/D)$,
by \Cref{lem: burden by rotation} there are
 $\bar{a}_{\alpha}'$ indiscernible over $b_{2}D$ (which gives
us~\ref{point:prodar1}) and $\beta_{\alpha}<\kappa_{2}$ such that
$\bar{a}_{\alpha}'\equiv_{a_{(\alpha,\beta_{\alpha}),0}D}\bar{a}_{(\alpha,\beta_{\alpha})}$ (which gives us~\ref{point:prodar2}),  and as by the inductive assumption~\ref{point:prodar3} we have in particular that $\bar{a}_{<\alpha}'\frown \bar a_{(\alpha,\beta_\alpha)}\frown \bar{a}_{\geq(\alpha+1,0)}$ is mutually indiscernible, the condition $\bar{a}_{\alpha}'\equiv_D\bar{a}_{(\alpha,\beta_{\alpha})}$ gives us that~\ref{point:prodar3} is preserved as well.
This completes the induction.

As, by~\ref{point:prodar2}, we have ${a}_{\alpha,0}'={a}_{(\alpha,\beta_{\alpha}),0}$
for every $\alpha<\kappa_1$, we get by~\ref{point:prodar1} that $\left(\bar{a}_{\alpha}'\right)_{\alpha<\kappa_{1}}$
is an almost mutually indiscernible array over $b_{2}$.

As there is no $\inp$-pattern of depth $\kappa_{1}$ in $\tp(b_{1}/b_{2})$,
by \Cref{lem: burden by rotation} there are some $\delta<\kappa_{1}$
and $\bar{a}''$ indiscernible over $b_{1}b_{2}$ and such that $\bar{a}''\equiv_{a_{(\delta,\beta_{\delta}),0}}\bar{a}_{\delta}'\equiv_{a_{(\delta,\beta_{\delta}),0}}\bar{a}_{(\delta,\beta_{\delta})}$ (for the second equivalence, note that $a_{\delta,0}'=a_{(\delta,\beta_{\delta}),0}$). Therefore, $\bar a''$ witnesses that condition~\ref{point:almmutind} in \Cref{lem: burden by rotation} is satisfied.
\end{proof}

\begin{co}[Submultiplicativity of burden]\label{co:submulbur}
If $\bdn(a_{i}/E)<k_{i}$ for
$i<n$ with $k_{i}<\omega$, then $\bdn(a_{0},\ldots, a_{n-1}/E)<\prod_{i<n}k_{i}$.
\end{co}

\begin{co}\label{singlev} If $T$ has $\mathsf{TP}_2$, then some formula $\phi(x,y)$ with $\abs{x}=1$ has $\mathsf{TP}_2$. 
\end{co}

 We now provide a counterexample to \cite[Lemma 2.3]{Chernikov}, and also to its relaxation to ``almost mutually indiscernible''. The lemma is used in that paper only with $\kappa=1$ in which case it is true (it is \Cref{lem: RotateIfConsistent} above, specialised to first-order logic), hence this does not affect correctness of any other results in there.

 \begin{eg}\label{eg:23cntrex}
  Let $\mathcal{L}$ be the language with three sorts
  $\mathrm{A},\mathrm{B},\mathrm{C}$ and a single function symbol
  $f \from \mathrm{A} \times \mathrm{B} \to \mathrm{C}$. One checks
  amalgamation of finite $\mathcal L$-structures easily, obtaining a Fra\"iss\'e class. Let $T$ be the theory of its Fra\"iss\'e limit (see~\cite[Definition~1.2(viii)]{HM} for an axiomatisation).

Let $\bar c=(c_i)_{i<\omega}$ be a sequence of pairwise distinct
  elements of sort $\mathrm{C}$, and let $a_0$ be an arbitrary element
  of sort $\mathrm{A}$. Consider the $2$-row array with first row
  $\bar c$ and second row $\bar a$ constantly equal to $a_0$. Let $C=\emptyset$. This is a counterexample to \cite[Lemma 2.3]{Chernikov}.
\end{eg}
\begin{proof}
  By quantifier elimination, this array is mutually indiscernible over
  $\emptyset$.

  Choose any $b \in \mathrm{B}$ such that $f(a_0, b)=c_0$. By quantifier
  elimination and genericity of $f$, the partial type $p^\infty$ in \cite[Lemma~2.3]{Chernikov} is consistent (in fact, it defines the whole sort $\mathrm{B}$).

  Now take the promised $b$-mutually indiscernible array, say
  $\bar c'\frown\bar a'$. By clause (1) of the lemma, the first column is
  the same as that of $\bar c\frown \bar a$, from which it follows that
  $\bar a'=\bar a$ and that $f(a_0',b)=c_0'$. But the $c_i$ are
  pairwise distinct, hence by clause (1) so are the $c_i'$, hence we
  cannot also have $f(a_0',b)=c_1'$. This contradicts mutual
  indiscernibility over $b$, i.e.\ clause (2).
\end{proof}

\subsection{Weak quantifier elimination and weak stable embeddedness}
In this subsection we introduce the notions of weak quantifier elimination and (strict) weak stable embeddedness, and generalise to positive logic certain lemmas from~\cite{CH} concerning array extension.

Recall that usual inductive theories may be identified with \emph{Pillay} positive theories; see~\cite[Subsection 6.2]{DM} for details.

\begin{lemma}\label{lemma:qf}
Let $T$ be a Pillay theory, $E\subseteq \monster$ be small, $x$ a variable of sort\footnote{Here we regard the family of sorts as being closed under cartesian products.} $X$, and let $\mathcal F$ be a family of quantifier-free formulas $\phi(x)$ over $E$,  closed under Boolean combinations. Assume that, for all $a,b\in\monster$ of sort $X$, if $a$ and $b$ satisfy the same formulas in $\mathcal F$ then $\tp(a/E)=\tp(b/E)$. Then, in $\monster$, every quantifier-free formula over $E$ of sort $X$ is equivalent to a formula in $\mathcal F$.
\end{lemma}
\begin{proof}
   This is a standard application of a general topological fact, see e.g.~\cite[Exercise~3.1.1]{tentCourseModelTheory2012}. In the notation of the cited source, given $\phi(x)$ a quantifier-free formula, take as $X$ the space of maximal types in variables $x$, as $Y_1$ the set given by $\phi(x)$, which is clopen because $T$ is Pillay, and as $Y_2$ its complement.
\end{proof}
\begin{defin} The theory $T$ admits \emph{weak quantifier elimination} iff, whenever $a,b$ are compatible tuples from $\monster$, if $\qftp(a)=\qftp(b)$ then $\tp(a)=\tp(b)$.
\end{defin}

\begin{defin} We say that a $\emptyset$-definable set $D$ is
  \begin{enumerate}
  \item \emph{weakly stably embedded} iff every formula with parameters $\phi(x)$  implying $x\in D$ is equivalent to a partial type $\pi(x)$ with parameters in $D$, in the sense that they have the same set of solutions in every $\abs{T}^+$-saturated pec model; and
  \item  \emph{strictly weakly stably embedded} iff  furthermore, for every $c$ over which $\phi(x)$ is definable, $\pi(x)$ can be chosen over $D\cap \dcl(c)$.
  \end{enumerate}
\end{defin}
\begin{rem}
If $T$ is Boolean (that is, it is the translation in positive logic of a full first-order theory), then weak quantifier elimination is the same as quantifier elimination and weakly stably embedded is the same as stably embedded.
\end{rem}

\begin{defin} Let $D$ be a $\emptyset$-definable set. By $D_{\indu}(\monster)$ we denote the structure with universe $D(\monster)$ and, for all $n<\omega$ and every formula $\phi(\bla x1,n)$ over $\emptyset$,  a predicate for $\phi(\monster)\cap D(\monster)^n$.
\end{defin}

We now generalise~\cite[Lemmas~3.6, 3.7 and 3.8]{CH}. 

\begin{lemma}\label{trivial}
Let $(c_{ij})_{i,j<\omega}$ be a strongly indiscernible array such that $(\bar c_i)_{i<\omega}$ is indiscernible over $a$, and let $b$ be arbitrary. Then there are $(b_{ij})_{i,j<\omega}$ such that $(b_{ij}c_{ij})_{i,j<\omega}$ is a strongly indiscernible array, $b_{ij}c_{ij}\equiv bc_{00}$ for all $i,j$, and $(\bar b_i\bar c_i)_{i<\omega}$ is indiscernible over $a$.
\end{lemma}
\begin{proof}
Let $\lambda = \beth((2^{\abs T+\abs{abc_{00}}})^+)$.
By compactness, we can extend $(c_{ij})_{i,j<\omega}$ to a strongly indiscernible array $(c_{ij})_{i<\lambda, j<\omega}$ such that $\bar c_{<\lambda}$ is $a$-indiscernible.
For all $i<\lambda$ and $j<\omega$, let $b'_{ij}$ be such that $b'_{ij}c_{ij}\equiv b_{00}c_{00}$. By \Cref{ar_mod} we can find a strongly indiscernible array $(b''_{ij}c''_{ij})_{i<\lambda,j<\omega}$ which is $\mathrm{EM}_{\ar}$-based on $(\bar b'_{i}\bar c_{i})_{i<\lambda}$. As  $\bar c_{<\lambda}$ was strongly indiscernible, $\mathrm{EM}_{\ar}((\bar b'_{i}\bar c_{i})_{i<\lambda})\supseteq \tp(\bar c_{<\lambda})$ so
$\bar c''_{<\lambda}\equiv \bar c_{<\lambda}$. Hence, by applying an automorphism, we may assume $\bar c''_{<\lambda}= \bar c_{<\lambda}$. By \Cref{fact:extracting}, we can find an $a$-indiscernible sequence $(\bar b_i\bar c'''_i)_{i<\omega}$  based on $(\bar b''_i\bar c_i)_{i<\lambda}$ over $a$. As $\bar c_{<\lambda}$ was already indiscernible over $a$, we get that $\bar c'''_{<\omega}\equiv_a \bar c_{<\omega}$, so, by applying an automorphism over $a$, we may assume $\bar c'''_{<\omega}= \bar c_{<\omega}$. Then 
$(b_{ij})_{i,j<\omega}$ satisfies the requirements.
\end{proof}
\begin{defin}
  We say that a structure $M$ (as opposed to a positive theory) is (positively) \emph{$\mathsf{NTP}_2$} iff there is no infinite $\inp$-pattern with constant $\phi_\alpha$, $k_\alpha$, $\psi_\alpha$ and parameters in $M$.
\end{defin}
We will exclusively use the notion just introduced with $M=D_{\indu}(\monster)$ for $D$ some $\emptyset$-definable strictly weakly stably embedded set.
\begin{lemma}\label{seD} Let $D$ be a $\emptyset$-definable strictly weakly stably embedded set such that $D_{\indu}(\monster)$ is  $\mathsf{NTP}_2$. Let $b\subseteq D(\monster)$ and assume  $(\bar c_i)_{i<\kappa}$ is a mutually indiscernible array over $C$. If $\kappa\geq (\abs T+\abs b)^+$, then there are $i<\kappa$ and $\bar c'\equiv _{Cc_{i0}} \bar c_i$ which is indiscernible over $Cb$.
\end{lemma}
\begin{proof}
We may assume $C=\emptyset$. By making each $c_{ij}$ longer, we may also assume $\dcl(c_{ij})=c_{ij}$ for all $i,j$.
Let $p_i(x,c_{i0})=\tp(b/c_{i0})$ for every $i<\kappa$.
By \Cref{lem: RotateIfConsistent}, it is enough to show that $\bigcup_{j\in \omega}p_i(x,c_{ij})$ is consistent for some $i$. Suppose not. By strict weak stable embeddedness, for every $i<\kappa$ we have $p_i(x,c_{ij})\equiv p_i(x,c_{ij})|_{D(\monster)\cap c_{ij}}$ and hence there is a formula $\psi_i(x,y')\in p_i(x,y')$ such that $\{\psi_i(x,D(\monster)\cap c_{ij})\mid j<\omega\}$ is $k_i$-inconsistent for some $k_i<\omega$.
By pigeonhole we can assume $\psi_i$ and $k_i$ are constantly equal to some fixed $\psi$ and $k$. Now $(D(\monster)\cap c_{ij})_{i,j<\omega}$,  $(x\in D)\wedge \psi(x,y)$, and $k$ witness that $D_{\indu}(\monster)$ has $\mathsf{TP}_2$, because the consistency condition with the $0^\kappa$-path is witnessed by $b$, hence holds for all paths by mutual indiscernibility.
\end{proof}

\begin{lemma}\label{row_replacement}
	Suppose $C$ is a set of parameters and $(a_{i})_{i<\lambda,j<\omega}$ is almost mutually indiscernible over $C$. Then there is a mutually indiscernible over $C$ array $(a'_{ij})_{i<\lambda,j<\omega}$ such that $\bar a'_i\equiv_{a_{i0}C} \bar a_i$ for every $i<\lambda$.
\end{lemma}
\begin{proof}
We may assume that $C=\emptyset$.
	By \Cref{lambda_mod} there is a mutually indiscernible array $(\bar a''_i)_{i<\lambda}$ which is  $\mathrm{EM}_{\mathfrak L_\lambda}$--based on $\bar a_{<\lambda}$.
		 By the assumption on $\bar a_{<\lambda}$ we have that $\bar a_i$ is indiscernible for every $i<\lambda$, hence $\bar a''_i\equiv \bar a_i$, and that for all $k<\omega$,  $i_0<\ldots<i_{k-1}<\lambda$ and $j_0,\ldots,j_{k-1}<\omega$ 
		we have $a_{i_0j_0}\ldots a_{i_{k-1}j_{k-1}}\equiv_C a_{i_00}\ldots a_{i_{k-1}0}$, hence $(a''_{i0})_{i<\lambda}\equiv (a_{i0})_{i<\lambda}$. 
		Thus, we can find an automorphism $f$ sending $(a''_{i0})_{i<\lambda}$ to $(a_{i0})_{i<\lambda}$ and define $\bar a'_i=f(\bar a''_i)$ for each $i<\lambda$.
\end{proof}

In \Cref{careful_ext}, we assume that $T$ is thick; recall that this means that ``$(x_0,x_1,\ldots)$ is an indiscernible sequence'' is equivalent, in sufficiently saturated pec models, to a partial type.  Direct usage of thickness will be confined to the proof of \Cref{careful_ext}, and mediated by its following consequence.

\begin{rem}\label{rem:thickind}
  Assume $T$ is thick.
  \begin{enumerate}
  \item\label{point:tind1} If $\bar c$ is a sequence indiscernible over $\bar d E$ and $\bar d'$ is $\mathrm{EM}_<$-based on $\bar d$ over $\bar cE$ then $\bar c$ is indiscernible over $\bar d'E$.
  \item\label{point:tind2}   Let $(\bar a_i)_{i<\lambda}$ be $\mathrm{EM}_\lambda$-based on $(\bar b_j)_{j<\lambda}$ over $E$. If $\bar b_{<\lambda}$ is  mutually indiscernible over $E$, then so is $\bar a_{<\lambda}$.
  \end{enumerate}
\end{rem}

\begin{lemma}\label{careful_ext} Assume $T$ is thick.
Let $D$ be a $\emptyset$-definable strictly weakly stably embedded set and assume that $D_{\indu}(\monster)$ is $\mathsf{NTP}_2$. Let $b\subseteq D(\monster)$ be arbitrary and $(c_{ij})_{i,j<\omega}$ be a strongly indiscernible array such that $(\bar c_i)_{i<\omega}$ is indiscernible over $a$.
Then there are $(c^*_{ij})_{i,j<\omega}$ and $(b^*_{ij})_{i,j<\omega}$ such that
\begin{enumerate}
\item\label{careful_ext1} $(\bar b^*_i\bar c^*_i)_{i<\omega}$ is indiscernible over $a$
\item\label{careful_ext2} $(\bar b^*_i\bar c^*_i)_{i<\omega}$ is mutually indiscernible
\item\label{careful_ext3} $\bar c^*_i\equiv _{c_{i0}} \bar c_i $ for all $i<\omega$ (so in particular $c^*_{i0}=c_{i0}$)
\item\label{careful_ext4} $b^*_{00} c^*_{00}\equiv_a bc_{00}$
\end{enumerate}
\end{lemma}
\begin{proof}

\begin{claim} It is enough to find, for every $n<\omega$, sequences $\bar c^{**}_{<n}$ and $\bar b^{**}_{<n}$ satisfying~\ref{careful_ext2},~\ref{careful_ext3} and such that:
  \begin{enumerate}[label=(\arabic*$'$)]
    \setcounter{enumi}{3}
    \item\label{careful_ext4p} $b^{**}_{k0}c^{**}_{k0}\equiv_a bc_{00}$ for all $k<n$. 
  \end{enumerate}
\end{claim} 
\begin{claimproof}
Indeed, assume we can find such $\bar c^{**}_{<n}$ and $\bar b^{**}_{<n}$. Then, for $\kappa\coloneqq \beth((2^{\abs{T} + \abs{abc_{00}}})^+)$,  by compactness and \Cref{rem:thickind}\ref{point:tind2} we may find  $\bar c^{**}_{<\kappa}$ and $\bar b^{**}_{<\kappa}$ satisfying~\ref{careful_ext2}, together with~\ref{careful_ext3} (for all $i<\omega$), and with the analogue of~\ref{careful_ext4p} but for $i<\kappa$ instead of $k<n$, with $(c^{**}_{i0})_{i<\kappa}$ indiscernible over $a$. Then we can apply  \Cref{fact:extracting} to $(\bar b_i^{**}\bar c^{**}_i)_{i<\kappa}$ over $a$, obtaining an array $(\bar b'_i\bar c'_i)_{i<\omega}$ satisfying~\ref{careful_ext1},~\ref{careful_ext2}, and~\ref{careful_ext4}.
As $(c^{**}_{i0})_{i<\kappa}$ was indiscernible over $a$, we get that $(c'_{i0})_{i<\omega}\equiv_a (c^{**}_{i0})_{i<\omega}=(c_{i0})_{i<\omega}$. Let $f$ be an automorphism over $a$ sending $(c'_{i0})_{i<\omega}$ to $(c_{i0})_{i<\omega}$. Let $\bar b_i^{*}=f(\bar b'_i)$ and $\bar c_i^{*}=f(\bar c'_i)$. Then for every $i$ we have that $\bar c^*_i\equiv \bar c_i$ and $c^*_{i0}=c_{i0}$, yielding~\ref{careful_ext3}, while~\ref{careful_ext1},~\ref{careful_ext2} and~\ref{careful_ext4} are clearly preserved by $f$, hence they are satisfied by $(\bar b^*_i\bar c^*_i)_{i<\omega}$ as well.
\end{claimproof}

Fix any $n<\omega$. By compactness, stretch $(\bar c_i)_{i<\omega}$ to a strongly indiscernible over $a$ array 
$(\bar c_{(i,\ell)})_{(i,\ell)\in (n\times \omega,{<_{\lex} })}$ with $\bar c_{(0,i)}=\bar c_i$ for all $i<\omega$. For each $(i,\ell)\in n\times \omega$ find $b_{(i,\ell)0}$ with $b_{(i,\ell)0} c_{(i,\ell)0}\equiv_ab c_{00}$.
 By induction on $k<n$ we will find $(b_{k,j}^*c_{k,j}^*)_{j\in \omega}$ such that the following hold.

\begin{enumerate}[label=(\roman*)]
\item\label{point:careful_ext_r1} $(\bar b^*_i\bar c^*_i)_{i\le k}$  is almost mutually indiscernible over $(b_{(i,\ell)0}c_{(i,\ell)0})_{i\in [k+1, n-1],\ell<\omega}$.

\item\label{point:newr2}   $\bar c^*_{\leq k}\frown  (\bar c_{(i,\ell)})_{(i,\ell)\in ([k+1, n-1]\times \omega,<_{\lex})}$  is based on $(\bar c_i)_{i<\omega}$.

\item\label{point:newr3}  $(c^*_{i0})_{i\leq k}\frown ( c_{(i,\ell)0})_{(i,\ell)\in ([k+1, n-1]\times \omega,<_{\lex})}$  is based on $(c_{i0})_{i<\omega}$ over $a$.

\item\label{point:newr4}   $b^*_{k0}c^*_{k0} \equiv_abc_{00}$
\end{enumerate}
 Suppose we have constructed $(\bar b^*_i\bar c^*_i)_{i<k}$. Put $\lambda\coloneqq (\abs{T}+\abs{bc_{00}})^+$ and find by compactness  a sequence $(b_i^+\bar c_i^+)_{i<\lambda}$ that is $\mathrm{EM}_{<}$-based on $(b_{(k,\ell)0}\bar c_{(k,\ell)})_{\ell< \omega}$ over $aBC$, where 
\begin{gather*}
  C\coloneqq \bar c^*_{<k}(\bar c_{(i,\ell)})_{i\in [k+1, n-1],\ell< \omega },\\ B\coloneqq \bar b^*_{<k}(b_{(i,\ell)0})_{i\in [k+1, n-1],\ell<\omega}. 
\end{gather*}

Then $b^+_ic^+_{i0}\equiv_a bc_{00}$ for every $i<\lambda$, and, by the inductive assumption and \Cref{rem:thickind}\ref{point:tind1} applied  $k$ times, one for each row of $(\bar b^*_i \bar c^*_i)_{i<k}$,
\begin{equation}
  \label{eq:1}\tag{$\dagger$}
  \begin{split}
  (\bar b_{i}^*\bar c_{i}^*)_{i<k} \text { is almost mutually indiscernible}\\ \text{over } (b_i^+\bar c_i^+)_{i<\lambda}(b_{(i,\ell)0}\bar c_{(i,\ell)})_{i\in [k+1, n-1], \ell< \omega}.
  \end{split}
\end{equation}
Observe that, since~\ref{point:newr2} inductively holds, by choice of $b^+_i\bar c^+_{i}$  the sequence \[\bar c^*_{<k}\frown \bar c^+_{<\lambda}\frown (\bar c_{(i,\ell)})_{(i,\ell)\in ([k+1, n-1]\times \omega,<_{\lex})}\] is based on $\bar c_{<\omega}$.

It also follows by \Cref{rem:thickind}\ref{point:tind2} that $\bar c^+_{<\lambda}$ is a mutually indiscernible over $C$ array, hence, by 
 \Cref{seD}, there are $m<\lambda$ and $\bar d=(d_{i})_{i<\omega}\equiv_{ Cc^+_{m0}} \bar c^+_{m}$  with $(d_{i})_{i<\omega}$ indiscernible over $BC$. For each $j<\omega$, let $e_j$ be such that $e_jd_j\equiv_{BC} b^+_{m}c^+_{m0}$. By \Cref{fact:DK54} there is an indiscernible over $BC$ sequence 
 $(s_{j}t_{j})_{j<\omega}$ which is $\mathrm{EM}_{<}$-based on  $(e_jd_{j})_{j<\omega}$  over $BC$. Let $f$ be an automorphism over $BC$ sending $s_0t_0$ to $b^+_{m}c^+_{m0}$. 
 Then $\bar b^*_k \coloneqq f(\bar s)$ and $\bar c^*_k\coloneqq f(\bar t)$ (together with $\bar b^*_{<k}\bar c^*_{<k})$ satisfy~\ref{point:careful_ext_r1} because $(s_jt_j)_{j<\omega}$ is indiscernible over $BC$, hence so is $(b^*_{kj} c^*_{kj})_{j<\omega}$, and because by~\eqref{eq:1} the array $(b^*_i\bar c^*_i)_{i<k}$ is almost mutually indiscernible over $b^*_{k0}c^*_{k0}(b_{(i,\ell)0}c_{(i,\ell)0})_{i\in [k+1, n-1],\ell<\omega}$.

 By the same choice,  because $f$ fixes $C$ pointwise, and because $(\bar c_{(k,\ell)})_{\ell<\omega}$ is indiscernible over $C$, the sequence $\bar c_{\le k}$ also satisfies~\ref{point:newr2}. Also, \ref{point:newr3} follows by the fact that $c^*_{m0}=c^+_{m0}$ and the choice of $(b_i^+\bar c_i^+)_{i<\lambda}$, which implies, as $(c_{(k,\ell)0})_{\ell<\omega}$ is indiscernible over $aC$, that $(c_{i0}^+)_{i<\lambda}$ is based on $(c_{(k,\ell)0)})_{\ell<\omega}$ over $aC$. Point~\ref{point:newr4} holds because $b^*_{k0}c^*_{k0}=b^+_m c^+_{m0}$.  This completes the construction of $(\bar b^*_i\bar c^*_i)_{i<n}$.
 
 Applying \Cref{row_replacement} to $(\bar b^*_i \bar c^*_i)_{i<n}$ we obtain an array $(\bar b^{***}_i \bar c^{***}_i)_{i<n}$  satisfying~\ref{careful_ext2},~\ref{careful_ext4p}, and  $\bar c^{***}_i\equiv \bar c_{i}$ with $c^{***}_{i0}= c^{*}_{i0}$ for every $i<n$. By~\ref{point:newr3} we have in particular that $(c^{*}_{i0})_{i<n}\equiv_a (c_{i0})_{i<n}$, hence we can find an automorphism $g$ over $a$ with $g((c^{*}_{i0})_{i<n})= (c_{i0})_{i<n}$. Then $\bar b^{**}_i\coloneqq g(\bar b^{***}_i)$ and $\bar c^{**}_i\coloneqq g(\bar c^{***}_i)$ satisfy~\ref{careful_ext2},~\ref{careful_ext3} and~\ref{careful_ext4p}, as required.
\end{proof}

\subsection{$\mathsf{NTP}_2$ for existentially closed valued difference fields of residual characteristic zero}\label{subsec:ntp2foracvfa00}

In our proof of $\mathsf{NTP}_2$ for valued difference fields of residual characteristic zero we will need to Morleyise the residue field sort. For details on (partial) Morleyisation in positive logic see~\cite[Remark 2.2]{DK} and~\cite[proof of Remark~6.14]{DM}.

\begin{notation}\label{notation:languagessec43}
In this subsection we denote by $\mathcal L$  one of the four languages in \Cref{notation:mcL}. We let $\mor_k$ be the language of the Morleyisation of $\mathsf{ACFA}$. We write $\widehat{\mathcal L}$ for the expansion of $\mathcal L$ by the symbols in $\mor_k$, added on the $k$ sort. We write $\mathsf{ACVFA}_{0,0,\widehat{\mathcal L}}$ for the h-inductive theory of $\widehat{\mathcal L}$-valued difference fields of residual characteristic zero. By  $\mathsf{ACVFA}_{0,0}$ we mean the h-inductive theory of valued difference fields of residual characteristic zero in the language $\mathcal L_\sigma$ of \Cref{defin:languages}.
\end{notation}

Note that, contrary to $\mathsf{ACFA}$, the axioms of $\mathsf{ACVFA}_{0,0,\widehat{\mathcal L}}$ and of $\mathsf{ACVFA}_{0,0}$ do not require any form of genericity of $\sigma$. This is common in the literature (e.g.~\cite{HK,DM}), and due to the fact that, at any rate, we focus on the existentially closed models.
 \begin{rem}\label{rem_Pil}
$\mathsf{ACVFA}_{0,0,\widehat{\mathcal L}}$  is a Pillay theory, as the negations of the formulas $x<_\Gamma y$, $x=_{\Gamma}y$, 
and $x=_{\mathrm{VF}}y$ are equivalent to the positive  formulas $x>_\Gamma y\vee x=_\Gamma y$, $x<_\Gamma y \vee x>_\Gamma y$, 
and $(x-y)^{-1}(x-y)=_{\mathrm{VF}}1$, respectively, and the negation of every predicate coming from Morleyising the sort $k$ is also added to the language in the Morleyisation process.
\end{rem} 
\begin{pr}
The JCP-refinements of $\mathsf{ACVFA}_{0,0,\widehat{\mathcal L}}$ and of $\mathsf{ACVFA}_{0,0}$ are determined by fixing a characteristic zero completion of $\mathsf{ACFA}$.
\end{pr}
\begin{proof}Indeed, suppose $T$ is such a completion and  $A,B\models \mathsf{ACVFA}_{0,0,\widehat{\mathcal L}}\cup T$. Then the prime subfields of $k(A)$ and of $k(B)$ are $\mor_k$-isomorphic. By identifying both of them with the same $\mathrm{Mor}_k$-expansion of $\mathbb Q$, we get that $k(A)$ and $k(B)$ amalgamate over $\mathbb Q$, so, by \Cref{AP_VDF}, $A$ and $B$ amalgamate over $\mathbb Q$. This proves JCP.

Conversely, if $T\supseteq \mathsf{ACVFA}_{0,0,\widehat{\mathcal L}}$ is a JCP-refinement, then for every $A,B\models T$ we can embed $A$ and $B$ in a model $C \models T$ and so the copies of $\mathbb Q^{\alg}$ in $k(A)$ and $k(B)$ are isomorphic. By \cite[Proposition 1.4]{CHr}, this means that $T$ implies a completion of $\mathsf{ACFA}$ on $k$.

For $\mathsf{ACVFA}_{0,0}$, the only difference is that in the proof of the first implication we pass from $A$ and $B$ to some positively $\aleph_1$-saturated extensions, and equip them with angular components by using \Cref{co:section}. On $\mathbb Q$ the angular component coincides with the residue map, hence $\mathbb Q$ is still a common substructure of $A$ and $B$ in the expanded language, hence we can apply \Cref{AP_VDF}.
\end{proof}

\begin{pr}\label{qe}  $\mathsf{ACVFA}_{0,0,\widehat{\mathcal L}}$ admits weak quantifier elimination in $\widehat{\mathcal{L}}$. 
\end{pr}
\begin{proof}
 By \cite[Theorems 8.5.2 and 8.5.4]{Ho}, it suffices to show that the class of substructures of models of $\mathsf{ACVFA}_{0,0,\widehat{\mathcal L}}$ has the Amalgamation Property. 

 Let $L\leftarrow M \rightarrow N$ be substructures of $\widehat{\mathcal L}$-valued difference fields.  By \Cref{AP_VDF} it suffices to show that the amalgamation problem $k(L) \leftarrow k(M) \rightarrow k(N)$ has a solution. If $k_1$ and $k_2$ are characteristic zero models of $\mathsf{ACFA}$ in the language $\mor_k$ extending $k(L)$ and $k(N)$, respectively, such that $k_2$ is $\abs{k_1}^+$-saturated, then $k(M)$ is a $\mor_k$-substructure of $k_1$ and of $k_2$, hence $\id_{k(M)}$ is a partial $\mor_k$-elementary map between $k_1$ and $k_2$. Hence $\id_{k(M)}$ extends to an embedding of $k_1$ into $k_2$. 
\end{proof}

\begin{co}
	In every existentially closed model of	$\mathsf{ACVFA}_{0,0,\widehat{\mathcal L}}$ the sorts $\Gamma$ and $k$ are \emph{orthogonal}, i.e.\ for every pair of maximal types $p(x)$ and $q(y)$ with $x$ a tuple of variables of the sort $\Gamma$ and $y$ a tuple of variables of the sort $k$, the partial type $p(x)
	\cup q(y)$ implies a maximal type in variables $xy$.
\end{co}
\begin{proof}
By \Cref{qe} it suffices to check that if $\alpha\models p$ and $\gamma,\gamma'\models q$ then the structures generated by $\alpha\gamma$ and $\alpha\gamma'$ are isomorphic. This is easy and left to the reader.
\end{proof}

\begin{pr}\label{se}
In every existentially closed $\monster\models\mathsf{ACVFA}_{0,0,\widehat{\mathcal L}}$ the following hold.
\begin{enumerate}
\item\label{point:swse} The sorts $k$ and $\Gamma$ are strictly weakly stably embedded. 
\item\label{point:eqinfconj} The formulas in $k_{\indu}(\monster)$ are equivalent in $\monster$ to infinitary conjunctions of formulas in  $\mathrm{Mor}_k(k(\monster))$, and the formulas in $\Gamma_{\indu}(\monster)$ are equivalent to infinitary conjunctions of quantifier-free formulas over $\Gamma(\monster)$ in the language $\mathcal{L}_{\mathrm{oag}}\cup \set{\sigma,\sigma\inverse}$.
\end{enumerate}
\end{pr}
\begin{proof}We prove both statements in the case of $\Gamma$; the case of $k$ is completely analogous. In both cases, we apply \Cref{lemma:qf}. We first deal with the language with lift and section.

For~\ref{point:swse}, it is enough to show that if $M$ is a substructure of $\monster$ and $\gamma, \delta\in \Gamma(\monster)^n$ are such that $\qftp(\gamma/\Gamma(M))=\qftp(\delta/\Gamma(M))$, then $\tp(\gamma/M)=\tp(\delta/M)$. Indeed, if we have this, let  $\phi(x,c)$ be a quantifier-free formula,  with $x$ a tuple of variables of sort $\Gamma$, and set $M\coloneqq\dcl(c)$. Apply  \Cref{lemma:qf} with $\mathcal F$ the family of quantifier-free formulas over $\Gamma(M)$, obtaining that $\phi(x,c)$ is equivalent to such a formula. By \Cref{qe}, it suffices to check the definition of strict weak stable embeddedness on quantifier-free formulas, hence this suffices.

Observe that, in the proofs of \cite[Lemmas~ 4.9 and~4.10]{Kesting}, the assumption that $\sigma_\Gamma$ is multiplicative is never used. We apply~\cite[Lemma~4.10]{Kesting} for the $\Gamma$ case, the other lemma being used in the $k$ case. By the cited result, we have an isomorphism $f\from \langle M, \gamma \rangle\to \langle M, \delta \rangle$ extending $\id_{M}$ and sending $\gamma$ to $\delta$. As $k(\langle M, \gamma \rangle)=k(M)=k( \langle M, \delta \rangle)$, in particular, $f$ restricted to  $k(\langle M, \gamma \rangle)$ is $\mathrm{Mor}_k$-elementary.
Hence, by weak quantifier elimination we get that $f$ is a partial immersion, so in particular $\tp(\gamma/M)=\tp(\delta/M)$.

For~\ref{point:eqinfconj} we show that, if $\gamma, \delta\in \Gamma(\monster)^n$ and $\qftp_{\Gamma}(\gamma)=\qftp_{\Gamma}(\delta)$, then $\tp(\gamma)=\tp(\delta)$, and then apply \Cref{lemma:qf} and \Cref{qe}. Again by \cite[Lemma~4.10]{Kesting}, we have an isomorphism $\langle \mathbb Q, \gamma\rangle\to \langle \mathbb Q, \delta\rangle$ sending $\gamma$ to $\delta$. As $k(\langle \mathbb Q, \gamma\rangle)=\mathbb Q =k(\langle \mathbb Q, \delta\rangle)$, we have in particular that $f|_{k(\langle \mathbb Q, \gamma\rangle)}=\id_{\mathbb Q}$ is elementary. Hence, by weak quantifier elimination we get that $f$ is a partial immersion.

For the other three languages, the proof is analogous, and uses analogues of \cite[Lemmas~ 4.9 and~4.10]{Kesting} that are easier to prove and left to the reader.
\end{proof} 

\begin{lemma}\label{immediate}
If $K\subseteq L$  is an immediate extension of $\widehat {\mathcal{L}}$-valued difference subfields of $\monster$, and $a\in L$, then every realisation in $\monster$ of the quantifier-free type of $a$ over $K$ in the language $\mathcal L$ (see \Cref{notation:languagessec43}) realises $\tp(a/K)$.
\end{lemma} 
\begin{proof}
This follows by \Cref{qe} and inspection of $\widehat{\mathcal L}$.
\end{proof}

\begin{rem}
By \Cref{AP_VDF}, as observed in~\cite[Remark~6.55]{DM}, the positive theory $\mathsf{ACVFA}_{0,0,\widehat{\mathcal L}}$ is \emph{Hausdorff}, see~\cite[Definition~2.13]{DK}. In particular, it is thick.
\end{rem}
\begin{thm}\label{ntp2} The positive theories $\mathsf{ACVFA}_{0,0,\widehat{\mathcal L}}$ and $\mathsf{ACVFA}_{0,0}$ are $\mathsf{NTP}_2$.
\end{thm}
\begin{proof}
  If $M$ is an $\widehat{\mathcal L}$-valued difference field, or a valued difference field in the language $\mathcal L_\sigma$, as in the proof of \Cref{co:reductswithlifts} we may expand a suitable pec extension of it to an $\widehat{\mathcal L}_{s,\iota,\sigma}$-valued difference field.   It follows that it suffices to show that $\mathsf{ACVFA}_{0,0,s,\iota}$ is $\mathsf{NTP}_2$.
  
Suppose for a contradiction that a formula $\phi(x,y)$ has $\mathsf{TP}_2$. We may assume that $x$ is a single variable by \Cref{singlev}.
As the theory of $\Gamma$ is $\mathsf{NIP}$ (hence $\mathsf{NTP}_2$, see~\cite[Theorem~1.1]{DGK}) by \cite[Corollary~6.39]{DM}, by \Cref{se}\ref{point:eqinfconj} and \Cref{infinitary_tp2} the structure $\Gamma_{\indu}(\monster)$ is $\mathsf{NTP}_2$. Similarly for $k_{\indu}(\monster)$, using the fact that $k(\monster)\models \mathsf{ACFA}$, hence is simple~\cite{CHr}. Therefore, by \Cref{seD,lem: burden by rotation} we may further assume that $x$ is of sort $\mathrm{VF}$. By \Cref{qe} and \Cref{infinitary_tp2} we may furthermore assume that $\phi(x,y)$ is quantifier-free.

Let $(a_{ij})_{i,j<\omega}$ be a strongly indiscernible array witnessing that $\phi(x,y)$ has $\mathsf{TP}_2$. Let $a\models \{ \phi(x,a_{i0})\mid i<\omega\}$. By \Cref{fact:DK54} we may assume that the sequence $(\bar a_i)_{i<\omega}$ is $a$-indiscernible.  We run the analogues for the language $\widehat{\mathcal L}_{s,\iota,\sigma}$ of steps (I) to (V) in the proof of \cite[Theorem~4.1]{CH}, but we replace \cite[Lemmas~3.6 and 3.8]{CH} by \Cref{careful_ext,trivial}. We obtain, for $\beta<\omega$, strongly indiscernible arrays $(a_{ij}^\beta)_{i,j<\omega}$ such that $K\coloneqq \bigcup_{\beta<\omega} a^\beta_{00}$ is a substructure of $\monster$ (in fact, a valued difference subfield because, since  $\iota$ and $s$ are in the language, the maps $\res$ and $v$ are automatically surjective), the union is increasing, $L\coloneqq K\langle a\rangle$ is an immediate extension of $K$, 
\begin{enumerate}
\item for every $\beta$, the sequence $(\bar a_{i}^\beta)_{i<\omega}$ is indiscernible over $a$, and
\item\label{point:acvfantp22} for every $\alpha<\beta$ and $i,j<\omega$ we have a decomposition $a_{ij}^\beta=(a^*_{ij}, b_{ij}^*)$ such that $\bar a_i^*\equiv_{a_{i0}^{\alpha}} \bar a_i^\alpha$.
\end{enumerate}
Note that it follows from~\ref{point:acvfantp22} that $a^\beta_{i,0}\supseteq a_{i,0}$ for all $i,\beta<\omega$, so 
$a\models \{ \phi(x,a^\beta_{i0})\mid i<\omega\}$ and the rows $\{ \phi(x,a^\beta_{ij})\mid j<\omega\}$ are inconsistent for all $i,\beta<\omega$.

We now deviate from the proof of \cite[Theorem~4.1]{CH}: instead of constructing $(a^\omega_{ij})_{i,j<\omega}$ and running their step (VI), we argue as follows. By \Cref{immediate} and compactness there is a quantifier-free $\mathcal L$-formula $\psi(x,c)\in \qftp(a/K)$ such that $\psi(\monster,c)\subseteq \phi(\monster,a_{00})$. As $K=\bigcup_{\beta<\omega} a^\beta_{00}$ is an increasing union, we may assume $c=a_{00}^\beta$ for some $\beta<\omega$. Then $\psi(x,a^\beta_{00})$ is a quantifier-free formula, and by indiscernibility of the rows and strong homogeneity of $\monster$ we have  $\psi(\monster,a^\beta_{ij})\subseteq \phi(\monster,a_{ij}^\beta)$ for all $i,j<\omega$, preserving inconsistency, while $a\models \psi(x,a_{i0}^\beta)_{i<\omega}$ by indiscernibility of $(a^\beta_{i0})_{i<\omega}$ over $a$, so $(a^\beta_{ij})_{i,j<\omega}$ witnesses that $\psi$ has $\mathsf{TP}_2$.  In particular, $\psi$ has $\mathsf{IP}$ witnessed by some $(a_i)_{i<\omega}$ and $(b_{w})_{w\subseteq \omega}$. But then, writing $\psi(x,y)=\rho(x,\sigma(x),\ldots, \sigma^n(x),y,\sigma(y),\ldots,\sigma^n(y))$ for some $n$ and a formula $\rho$ in the language of $\mathcal L_{s,\iota}$-valued fields, we get that $\rho$, $(a^\beta_{i},\sigma(a^\beta_{i}),\ldots,\sigma^n(a^\beta_{i}))_{i<\omega}$ and $(b_{w},\sigma(b_w),\ldots, \sigma^n(b_w))_{w\subseteq \omega}$ witness that the $\mathcal L_{s,\iota}$--valued algebraically closed field $\monster$ has $\mathsf{IP}$. But algebraically closed $\mathcal L_{s,\iota}$-valued fields of residual characteristic zero are $\mathsf{NIP}$, as can be shown, e.g.\ by using~\cite[Theorem~2.2]{Kesting}. This is a contradiction.
\end{proof}

\addcontentsline{toc}{section}{\bibname}

\emergencystretch=1em

\renewcommand*{\bibfont}{\normalfont\footnotesize}
\printbibliography

\end{document}